\documentclass{amsart}

\usepackage{amsmath}
\usepackage{mathabx}
\usepackage{amsfonts}
\usepackage{amssymb}
\usepackage{amsthm}
\usepackage{comment}
\usepackage{epsfig}
\usepackage{psfrag}
\usepackage{mathrsfs}
\usepackage{amscd}
\usepackage[all]{xy}
\usepackage{rotating}
\usepackage{lscape}
\usepackage{amsbsy}
\usepackage{verbatim}
\usepackage{moreverb}
\usepackage{color}
\usepackage{eucal}
\usepackage{stmaryrd}
\usepackage{hyphenat}

\newtheorem{theorem}{Theorem}[section]
\newtheorem{prop}[theorem]{Proposition}
\newtheorem{lemma}[theorem]{Lemma}
\newtheorem{cor}[theorem]{Corollary}

\theoremstyle{definition}
\newtheorem{definition}[theorem]{Definition}

\newtheorem{observation}[theorem]{Observation}
\newtheorem{construction}[theorem]{Construction}
\newtheorem{terminology}[theorem]{Terminology}
\newtheorem{remark}[theorem]{Remark}
\newtheorem{convention}[theorem]{Convention}
\newtheorem{example}[theorem]{Example}
\newtheorem{q}[theorem]{Question}
\newtheorem{problem}[theorem]{Problem}
\newtheorem{notation}[theorem]{Notation}
\newtheorem{obs/def}[theorem]{Observation/Definition}

\theoremstyle{remark}

\definecolor{orange}{rgb}{.95,0.5,0}
\definecolor{light-gray}{gray}{0.75}
\definecolor{brown}{cmyk}{0, 0.8, 1, 0.6}
\definecolor{plum}{rgb}{.5,0,1}

\DeclareMathOperator{\CAlg}{\sf CAlg}

\DeclareMathOperator{\TwAr}{\sf TwAr}

\DeclareMathOperator{\Psh}{\sf PShv}
\DeclareMathOperator{\PShv}{\sf PShv}

\DeclareMathOperator{\colim}{{\sf colim}}

\DeclareMathOperator{\limit}{{\sf lim}}

\DeclareMathOperator{\Hom}{\sf Hom}

\DeclareMathOperator{\Fun}{{\sf Fun}}

\DeclareMathOperator{\Map}{{\sf Map}}

\DeclareMathOperator{\Cat}{{\sf Cat}}
\DeclareMathOperator{\cat}{{\sf Cat}}

\DeclareMathOperator{\CAT}{{\sf CAT}}

\DeclareMathOperator{\fCAT}{{\sf fCAT}}
\DeclareMathOperator{\fCat}{{\sf fCat}}

\DeclareMathOperator{\Ar}{{\sf Ar}}

\DeclareMathOperator{\symcat}{\sf Sym-fCAT}

\DeclareMathOperator{\psh}{\sf PShv}

\DeclareMathOperator{\op}{\mathsf{op}}

\DeclareMathOperator{\ev}{\mathsf{ev}}

\DeclareMathOperator{\sd}{\mathsf{sd}}

\DeclareMathOperator{\spaces}{\cS\mathsf{paces}}
\DeclareMathOperator{\Spaces}{\cS\mathsf{paces}}

\DeclareMathOperator{\SPACES}{\cS\mathsf{PACES}}

\DeclareMathOperator{\Corr}{{\sf Corr}}
\DeclareMathOperator{\corr}{{\sf Corr}}

\DeclareMathOperator{\LCorr}{{\sf LCorr}}
\DeclareMathOperator{\RCorr}{{\sf RCorr}}

\def\ot{\otimes}

\DeclareMathOperator{\oo}{\infty}

\DeclareMathOperator{\cCart}{\sf cCart}

\DeclareMathOperator{\Cart}{\sf Cart}

\DeclareMathOperator{\RFib}{\sf RFib}

\DeclareMathOperator{\LFib}{\sf LFib}

\DeclareMathOperator{\tr}{\triangleright}
\DeclareMathOperator{\tl}{\triangleleft}

\newcommand{\lag}{\langle}
\newcommand{\rag}{\rangle}

\newcommand{\w}{\widetilde}

\newcommand{\ov}{\overline}

\newcommand{\ra}{\rightarrow}
\newcommand{\la}{\leftarrow}
\newcommand{\xra}{\xrightarrow}
\newcommand{\xla}{\xleftarrow}

\def\cA{\mathcal A}\def\cB{\mathcal B}\def\cC{\mathcal C}\def\cD{\mathcal D}
\def\cE{\mathcal E}\def\cF{\mathcal F}\def\cG{\mathcal G}
\def\cI{\mathcal I}\def\cJ{\mathcal J}\def\cK{\mathcal K}
\def\cP{\mathcal P}
\def\cS{\mathcal S}
\def\cU{\mathcal U}\def\cW{\mathcal W}\def\cX{\mathcal X}
\def\cY{\mathcal Y}\def\cZ{\mathcal Z}

\def\sB{\mathsf B}

\def\bdelta{\mathbf\Delta}

\DeclareMathOperator{\id}{{\sf id}}

\DeclareMathOperator{\pr}{\sf pr}

\DeclareMathOperator{\Exp}{\sf EFib}

\DeclareMathOperator{\Span}{\sf Span}

\def\bDelta{\mathbf\Delta}

\DeclareMathOperator{\EFib}{\sf EFib}
\DeclareMathOperator{\efib}{\sf EFib}

\begin{document}

\title{Fibrations of $\infty$-categories}

\author{David Ayala \& John Francis}

\address{Department of Mathematics\\Montana State University\\Bozeman, MT 59717}
\email{david.ayala@montana.edu}
\address{Department of Mathematics\\Northwestern University\\Evanston, IL 60208}
\email{jnkf@northwestern.edu}
\thanks{This work was supported by the National Science Foundation under award 1508040.}

\begin{abstract}
We construct a flagged $\infty$-category ${\sf Corr}$ of $\infty$-categories and bimodules among them. We prove that ${\sf Corr}$ classifies exponentiable fibrations.  
This representability of exponentiable fibrations extends that established by Lurie of both coCartesian fibrations and Cartesian fibrations, as they are classified by the $\infty$-category of $\infty$-categories and its opposite, respectively.
We introduce the flagged $\infty$-subcategories ${\sf LCorr}$ and ${\sf RCorr}$ of ${\sf Corr}$, whose morphisms are those bimodules which are \emph{left-final} and \emph{right-initial}, respectively.
We identify the notions of fibrations these flagged $\infty$-subcategories classify, and show that these $\infty$-categories carry universal left/right fibrations.

\end{abstract}



\maketitle

\tableofcontents

\section*{Introduction}

The theory of fibrations of $\oo$-categories differs from that of fibrations of spaces in two respects. For one, there are a host of differing notions of fibrations for $\oo$-categories. For another, every map of spaces can, up to homotopy, be replaced by one which is a fibration; in contrast, not every functor is homotopy equivalent to one which is a fibration, depending which notion one uses.

\smallskip

The following diagram depicts a variety of notions of fibrations among $\oo$-categories, each of which can be thought of homotopy-invariantly. 
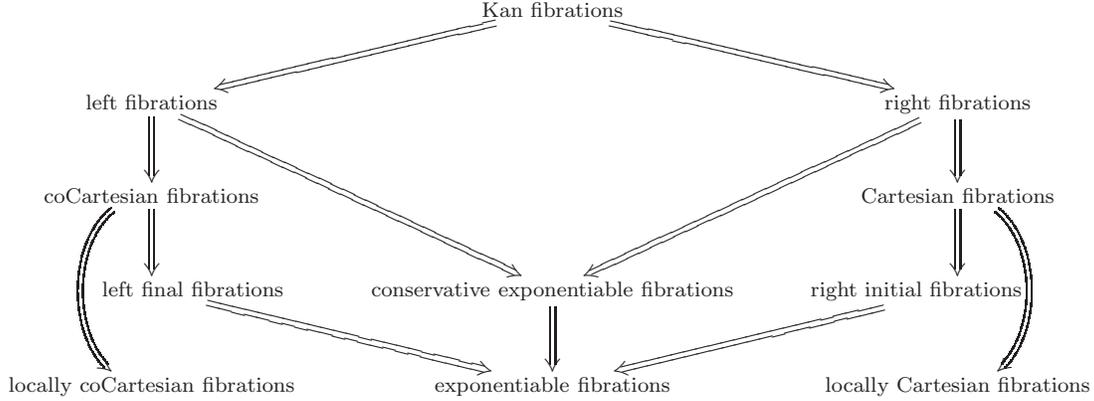
\[
\Small
\xymatrix{
&
{\rm Kan \ fibrations}\ar@{=>}[dr]\ar@{=>}[dl]
&
\\
{\rm left \ fibrations}\ar@{=>}[d]  \ar@{=>}[ddr]
&&
{\rm right \ fibrations}\ar@{=>}[d]  \ar@{=>}[ddl]
\\
{\rm coCartesian \ fibrations}\ar@{=>}[d]     \ar@(l,l)@{=>}[dd]
&&
{\rm Cartesian \ fibrations}     \ar@{=>}[d]    \ar@(r,r)@{=>}[dd]
\\
{\rm  \hspace{1.1cm} left \ final\ fibrations}    \ar@{=>}[dr]   
&
{\rm conservative \ exponentiable \ fibrations}  \ar@{=>}[d]
&
{\rm right \ initial \ fibrations  \hspace{1.1cm}}     \ar@{=>}[dl]
\\
{\rm locally \ coCartesian \ fibrations} &{\rm exponentiable \ fibrations} & {\rm locally \ Cartesian \ fibrations} }
\]
It is known~(\cite{HTT}) that each of the notions of fibration in the top half of the diagram have the following properties, which are familiar from topos theory:
\begin{enumerate}
\item They are closed under the formation of compositions.
\item They are closed under the formation of base change.
\item Base change along each is a left adjoint.  
\item They are classified by an $\infty$-category.

\end{enumerate}
In this work we explore the notion of an exponentiable fibration, and variations thereof.
We show that exponentiable fibrations are classified by an object $\Corr$, which carries a universal exponentiable fibration $\ov{\Corr}\to \Corr$.    
We identify $\Corr$ as the \emph{Morita} $\infty$-category, of $\infty$-categories and bimodules among them.
Phrased differently, we show that functors to this Morita $\infty$-category can be unstraightened as exponentiable fibrations, and every exponentiable fibration arises in this way.  
This result extends the unstraightening construction concerning (co)Cartesian fibrations, and (left)right fibrations, as established by Lurie in~\cite{HTT}.
There is then a diagram of classifying objects and monomorphisms among them
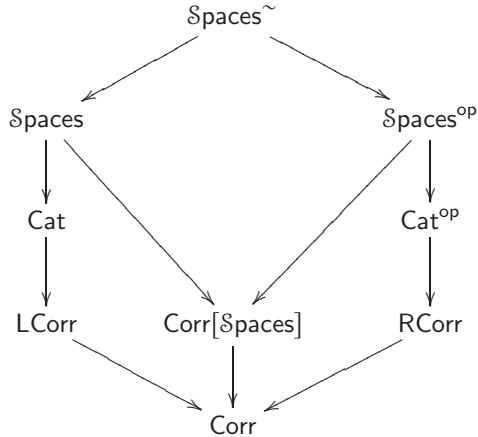
\[
\xymatrix{
&
\Spaces^\sim   \ar[dr]        \ar[dl]
&
\\
\Spaces  \ar[d]  \ar[ddr]
&&
\Spaces^{\op}  \ar[d]  \ar[ddl]
\\
\Cat  \ar[d]   
&&
\Cat^{\op}    \ar[d]  
\\
\LCorr    \ar[dr]   
&
\Corr[\Spaces]  \ar[d]
&
\RCorr     \ar[dl]
\\
&
\Corr
& 
}
\]
corresponding to the first diagram formed by notions of fibrations.

\begin{remark}
Our notion of an \emph{exponentiable fibration} is a homotopy-invariant formulation of that of a \emph{flat inner} fibration in the quasi-category model, developed by Lurie in \S B.3 of \cite{HA}. The relation of these notions follows from the equivalence of conditions (1) and (6) given in Lemma \ref{exp-char}.
There is an accessible survey~\cite{barwick.shah} on various notions of fibrations in the quasi-category model for $\infty$-categories.
Proposition 4.8 of that survey, whose proof is deferred to upcoming work of Peter Haine, is particularly consonant with our main results.
\end{remark}

\subsection{Main results}
We now precisely articulate the main results of this work.
To state them we give the following definition and basic results from~\cite{flagged}.
\begin{definition}[\cite{flagged}]\label{def.flagged}
A \emph{flagged $\infty$-category} is a functor $\cG \to \cC$ from an $\infty$-groupoid $\cG$ to an $\infty$-category $\cC$ which is {\it surjective}, i.e., for which every object in $\cC$ is equivalent to one in the image of $\cG$.  
For $\cG\to \cC$ a flagged $\infty$-category, its \emph{underlying $\infty$-groupoid} is $\cG$, while its \emph{underlying $\infty$-category} is $\cC$.
The $\infty$-category of flagged $\infty$-categories is the full $\infty$-subcategory of arrows
\[
{\sf f}{\sf CAT} ~\subset~ \Ar({\sf CAT})
\]
consisting of the flagged $\infty$-categories.  

\end{definition}

\begin{theorem}[\cite{flagged}]\label{maximal.underlying}
Evaluation at the target defines a left adjoint in a localization
\[
{\sf f}{\sf CAT} \longrightarrow {\sf CAT}
\]
whose right adjoint carries an $\infty$-category $\cC$ to the flagged $\infty$-category $(\cC^\sim \to \cC)$ whose underlying $\infty$-groupoid is the maximal $\infty$-subgroupoid of $\cC$.  
\end{theorem}

\begin{theorem}[\cite{flagged}]\label{flagged.thm}
The restricted Yoneda functors along $\bDelta \hookrightarrow \Cat \hookrightarrow \fCat$ determine fully-faithful functors
\[
{\sf fCAT} \hookrightarrow \Fun\bigl(\Cat^{\op}, \cS{\sf PACES}\bigr)
\qquad \text{ and }\qquad
{\sf fCAT} \hookrightarrow \Fun\bigl(\bDelta^{\op} , \cS{\sf PACES}\bigr)~.
\]
The image of the latter consists of those (large) presheaves on $\bDelta$ that satisfy the Segal condition, i.e., that preserve limit diagrams in the subcategory $\bDelta^{\sf inrt,\op}\subset \bDelta^{\op}$ of inert morphisms, which are the consecutive inclusions among finite non-empty linearly ordered sets.  

\end{theorem}

\begin{remark}\label{doesntdepend}
As formulated, the present work depends on the preceding results from \cite{flagged}. However, the dependence is slight: if one replaces every occurrence of ``flagged $\oo$-category" with ``Segal space" (or interprets Theorem \ref{flagged.thm} as the definition of a flagged $\oo$-category) then the present work becomes independent of \cite{flagged}.
\end{remark}

\begin{remark}\label{more.representables}
A flagged $\infty$-category is a stack on $\infty$-categories that satisfies descent with respect to those colimit diagrams among $\infty$-categories that additionally determine colimit diagrams among their maximal $\infty$-subcategories.
(We elaborate on this in~\cite{flagged}.)
This slight enlargement of $\infty$-categories $\Cat\hookrightarrow \fCat$ to \emph{flagged} $\infty$-categories accommodates representability of presheaves on $\Cat$ as \emph{flagged} $\infty$-categories which might not be representable by $\infty$-categories.  
Notably, as we will see in the present work, \emph{exponentiable fibrations} are \emph{not} classified by an $\infty$-category, though they are classified by a flagged $\infty$-category (Main Theorem 1, below).
The essential explanation for why exponentiable fibrations are \emph{not} classified by an $\infty$-category is because not all $\infty$-categories are idempotent complete; see Example~\ref{not.univalent} for more discussion.

\end{remark}

We recall the definition of an exponentiable fibration between $\infty$-categories, an $\oo$-categorical generalization of a notion first developed by Giraud \cite{giraud} and Conduch\'e \cite{conduche}.
\begin{definition}[\cite{emb1a}]\label{def.exp-fib}
A functor $\pi\colon \cE\to \cK$ between $\infty$-categories is an \emph{exponentiable fibration} if the pullback functor
\[
\pi^\ast\colon \Cat_{/\cK} \longrightarrow \Cat_{/\cE}
\]
is a left adjoint.
The $\infty$-category of exponentiable fibrations over an $\infty$-category $\cK$ is the full $\infty$-subcategory
\[
\EFib_{\cK}~\subset~\Cat_{/\cK}
\]
consisting of the exponentiable fibrations; its maximal $\infty$-subgroupoid is $\Exp^\sim_{\cK}$.  

\end{definition}

The following result articulates how exponentiable fibrations are classified by the flagged $\infty$-category of correspondences (among $\infty$-categories).
Its proof is the content of~\S\ref{sec.corr}.  
\begin{theorem}[Main Theorem 1]\label{main.thm}
\begin{enumerate}
\item[~]
\item
There is a (large) flagged $\infty$-category $\Corr$ with the following properties.
\begin{enumerate}
\item The underlying $\infty$-groupoid of $\Corr$ is that of (small) $\infty$-categories.
In particular, an object is the datum of a (small) $\infty$-category.

\item A morphism from $\cA$ to $\cB$ is the datum of a $(\cB,\cA)$-bimodule: 
\[
P\colon \cA^{\op}\times \cB \longrightarrow \Spaces~.
\]

\item For $P$ a $(\cB,\cA)$-bimodule, and for $Q$ a $(\cC,\cB)$-bimodule, their composition is the $(\cC,\cA)$-bimodule which is a coend over $\cB$:
\[
Q\circ P\colon \cA^{\op}\times\cC \xra{~P\underset{\cB}\otimes Q~} \Spaces~.
\]

\end{enumerate}

\item This flagged $\infty$-category $\Corr$ classifies exponentiable fibrations: for each $\infty$-category $\cK$, there is an equivalence between $\infty$-groupoids
\[
\fCat(\cK,\Corr)~\simeq ~ \EFib_{\cK}^\sim
\]
between that of functors from $\cK$ to $\Corr$ and that of exponentiable fibrations over $\cK$.

\item
This flagged $\infty$-category carries a natural symmetric monoidal structure.
On objects, this symmetric monoidal structure is given by products of $\infty$-categories:
\[
\bigotimes \colon (\cC , \cD)~\mapsto~ \cC\times \cD~,
\]

\end{enumerate}
\end{theorem}

The next result articulates how the classification of exponentiable fibrations of Main Theorem 1 extends the classification of other notions of fibrations.  
Its proof is the content of~\S\ref{sec.handed}.  
(Compare~\S2.1 of~\cite{HTT} for established definitions of left and right fibrations as well as Kan fibrations (in this context); see~\S2.4 of~\cite{HTT} for established definitions of coCartesian and Cartesian fibrations.)

\begin{theorem}[Main Theorem 2]\label{main.thm'}
The representability of exponentiable fibrations stated in Theorem~\ref{main.thm}(2) extends the representability of Kan fibrations, left fibrations, coCartesian fibrations, Cartesian fibrations, and right fibrations, in the following senses.

\begin{enumerate}

\item   
There are monomorphisms among flagged $\infty$-categories
\[
\Spaces^\sim ~\hookrightarrow~\Spaces~\hookrightarrow~\Cat~\hookrightarrow~ \Corr~\hookleftarrow~\Cat^{\op}~\hookleftarrow~\Spaces^{\op}~.
\]
With respect to finite products of $\infty$-categories, each of these monomorphisms lifts as symmetric monoidal monomorphisms between flagged $\infty$-categories.  

\item The images of the above monomorphisms are characterized as follows.
Let $F\colon \cK \xra{~\lag \cE\to \cK\rag~} \Corr$ be a functor from an $\infty$-category. 
\begin{enumerate}
\item There is a factorization $F\colon \cK \dashrightarrow \Cat \hookrightarrow \Corr$ if and only if any of the following equivalent conditions are satisfied:
\begin{enumerate}
\item $\cE\to \cK$ is a locally coCartesian fibration.
\item For each morphism $c_1\to \cK$, the fully-faithful functor $\cE_{|t} \hookrightarrow \cE_{|c_1}$ is a right adjoint.
\item For each morphism $c_1\to \cK$, the Cartesian fibration $\Fun_{/\cK}(c_1,\cE) \xra{\ev_s} \cE_{|s}$ is a right adjoint. 
\item For each point $c_0 \xra{~\lag y\rag~} \cK$, the fully-faithful functor $\cE_{|y} \longrightarrow \cE_{/y}:=\cE\underset{\cK}\times \cK_{/y}$ is a right adjoint.
\end{enumerate}

\item There is a factorization $F\colon \cK \dashrightarrow \Cat^{\op} \hookrightarrow \Corr$ if and only if any of the following equivalent conditions are satisfied:
\begin{enumerate}
\item $\cE\to \cK$ is a locally Cartesian fibration.
\item For each morphism $c_1\to \cK$, the fully-faithful functor $\cE_{|s} \hookrightarrow \cE_{|c_1}$ is a left adjoint.
\item For each morphism $c_1\to \cK$, the coCartesian fibration $\Fun_{/\cK}(c_1,\cE) \xra{\ev_t} \cE_{|t}$ is a left adjoint. 
\item For each point $c_0 \xra{~\lag x\rag~} \cK$, the fully-faithful functor $\cE_{|x} \longrightarrow \cE^{x/}:=\cE\underset{\cK}\times \cK^{x/}$ is a left adjoint.
\end{enumerate}

\item There is a factorization $F\colon \cK \dashrightarrow \Spaces \hookrightarrow \Corr$ if and only if any of the following equivalent conditions are satisfied:
\begin{enumerate}
\item $\cE\to \cK$ is a conservative locally coCartesian fibration.
\item $\cE \to \cK$ is a conservative coCartesian fibration.  
\item For each morphism $c_1\to \cK$, the functor $\Fun_{/\cK}(c_1,\cE) \xra{\ev_t}\cE_{|t}$ is an equivalence between $\infty$-groupoids.
\item For each point $c_0 \xra{~\lag y\rag~} \cK$, the functor $\cE_{|y} \longrightarrow \cE_{/y}:=\cE\underset{\cK}\times \cK_{/y}$ is an equivalence between $\infty$-groupoids.
\end{enumerate}

\item There is a factorization $F\colon \cK \dashrightarrow \Spaces^{\op} \hookrightarrow \Corr$ if and only if any of the following equivalent conditions are satisfied:
\begin{enumerate}
\item $\cE\to \cK$ is a conservative locally Cartesian fibration.
\item $\cE \to \cK$ is a conservative Cartesian fibration.  
\item For each morphism $c_1\to \cK$, the functor $\Fun_{/\cK}(c_1,\cE) \xra{\ev_s}\cE_{|s}$ is an equivalence between $\infty$-groupoids.
\item For each point $c_0 \xra{~\lag x\rag~} \cK$, the functor $\cE_{|x} \longrightarrow \cE^{x/}:=\cE\underset{\cK}\times \cK^{x/}$ is an equivalence between $\infty$-groupoids.
\end{enumerate}

\end{enumerate}

\end{enumerate}
\end{theorem}

The next result articulates a few other notions of fibrations, and offers flagged $\infty$-categories classifying them.
In future works, we find this result useful for constructing presheaves on various $\infty$-categories.
Its proof is the content of~\S\ref{sec.final}.  
\begin{theorem}[Main Theorem 3]\label{other.thm}
There are symmetric monoidal flagged $\infty$-subcategories
\[
\Corr[\Spaces]~,~\LCorr~,~\RCorr~\subset~\Corr
\]
with the following properties.  
Let $F\colon \cK \xra{\lag \cE\to \cK\rag} \Corr$ be a functor which classifies the exponentiable fibration $\cE\to \cK$.
\begin{enumerate}
\item There is a factorization $F\colon \cK \dashrightarrow \Corr[\Spaces] \hookrightarrow \Corr$ if and only if any of the following equivalent conditions are satisfied:
\begin{enumerate}
\item $\cE\to \cK$ is conservative.
\item For each morphism $c_1\to \cK$, the $\infty$-category $\Fun_{/\cK}(c_1,\cE)$ is an $\infty$-groupoid.
\item For each object $x\in \cK$, the fiber $\cE_{|x}$ is an $\infty$-groupoid.
\end{enumerate}

\item There is a factorization $F\colon \cK \dashrightarrow \LCorr \hookrightarrow \Corr$ if and only if any of the following equivalent conditions are satisfied:
\begin{enumerate}
\item For each morphism $c_1\to \cK$, the fully-faithful functor $\cE_{|t} \hookrightarrow \cE_{|c_1}$ is final.
\item For each morphism $c_1\to \cK$, the Cartesian fibration $\Fun_{/\cK}(c_1,\cE) \xra{\ev_s} \cE_{|s}$ is final.
\item For each point $c_0 \xra{~\lag y\rag~} \cK$, the fully-faithful functor $\cE_{|y} \longrightarrow \cE_{/y}:=\cE\underset{\cK}\times \cK_{/y}$ is final.
\end{enumerate}

\item There is a factorization $F\colon \cK \dashrightarrow \RCorr \hookrightarrow \Corr$ if and only if any of the following equivalent conditions are satisfied:
\begin{enumerate}
\item For each morphism $c_1\to \cK$, the fully-faithful functor $\cE_{|s} \hookrightarrow \cE_{|c_1}$ is initial.
\item For each morphism $c_1\to \cK$, the coCartesian fibration $\Fun_{/\cK}(c_1,\cE) \xra{\ev_t} \cE_{|t}$ is initial.
\item For each point $c_0 \xra{~\lag x\rag~} \cK$, the fully-faithful functor $\cE_{|x} \longrightarrow \cE^{x/}:=\cE\underset{\cK}\times \cK^{x/}$ is initial.
\end{enumerate}

\end{enumerate}
Furthermore, taking fiberwise groupoid completions of exponentiable fibrations defines morphisms between symmetric monoidal flagged $\infty$-categories:
\[
\sB\colon \LCorr \longrightarrow \Spaces
\qquad\text{ and }\qquad
\sB\colon \RCorr \longrightarrow \Spaces^{\op}~.
\]

\end{theorem}

\subsection{Motivation}
We make some informal comments on our motivations for this work, which is designed to support the $\oo$-categorical argumentation employed in our works on differential topology, such as \cite{fact}, \cite{aft1}, \cite{aft2}, \cite{striation}, and \cite{emb1a}.
We are often interested in constructing a functor $\cK \to \cZ$ between $\infty$-categories, where $\cZ$ is more or less fixed and $\cK$ is somewhat arbitrary.  
One strategy for doing so is to find an enlargement, specifically a monomorphism $\cZ \hookrightarrow \w{\cZ}$, then divide the problem of constructing a functor $\cK \to \cZ$ as two steps: first construct a functor $\cK \to \w{\cZ}$, then {\bf check} that it factors through $\cZ$.  
The problem of constructing a functor $\cK \to \w{\cZ}$ becomes a practical one to solve once $\w{\cZ}$ is recognized as classifying certain \emph{fibrations}; more precisely, a functor $\cK \to \w{\cZ}$ is determined by a functor $\cE\to \cK$ satisfying certain properties, which can be {\bf checked}.  
The weakest notion of such a fibration that admits such a classification is that of an \emph{exponentiable fibration}.
We demonstrate this technique for constructing a functor $\cK \to \cZ$ through a simple example.

Let $\cK$ be an $\infty$-category.
Constructing a presheaf
\[
\cK^{\op} \longrightarrow \Spaces
\]
is often not practical in $\infty$-category theory.  
This is pointedly demonstrated by the impracticality of constructing, for each $x\in \cK$, the representable presheaf:
\[
\cK(-,x) \colon \cK^{\op} \longrightarrow \Spaces~.
\]
On the otherhand, it is easy to construct the right fibration
\[
\cK_{/x}~\longrightarrow~ \cK~,
\]
as we now demonstrate.
\begin{itemize}

\item[{\bf Step 1:}]
For each functor $\cJ \xra{f} \cK$ between $\infty$-categories, declare the space of lifts
\[
\xymatrix{
&&
\cK_{/x}  \ar[d]
\\
\cJ  \ar[rr]^-f  \ar@{-->}[urr]
&&
\cK
}
\]
to be the space of fillers
\[
\xymatrix{
\ast  \ar[d]  \ar[drr]^-{\lag x\rag}  
&&
\\
\cJ^{\tr}  \ar@{-->}[rr]  
&&
\cK
\\
\cJ  \ar[u]  \ar[urr]_-{f}
&&
.
}
\]
It must be checked that this defines an $\infty$-category over $\cK$.
Said another way, because the (large) $\infty$-category of $\infty$-categories is presentable, it must be {\bf checked} that the presheaf
\[
\bigl(\Cat_{/\cK}\bigr)^{\op} \longrightarrow \Spaces~,\qquad \cJ \mapsto \Cat^{\cJ \amalg \ast/}(\cJ^{\tr},\cK)~,
\]
carries the opposites of colimit diagrams to limit diagrams.  
This check is manageable, using that the construction of right cones is a colimit construction.
(Note that the values of the above asserted presheaf on $\Cat$ are in terms of limit and colimit constructions.  
The requisite functoriality of this presheaf on $\Cat$ then follows, ultimately, from suitable functoriality of limit and colimit constructions.)

\item[{\bf Step 2:}]
It must be checked that the functor $\cK_{/x} \to \cK$ is a right fibration.  
Said another way, it must be {\bf checked} that this functor is conservative and locally Cartesian.  
This check is manageable, using that each solid diagram of $\infty$-categories
\[
\xymatrix{
\ast  \ar[rr]  \ar[d]
&&
\cK_{/x}  \ar[d]
\\
\cJ^{\tr}  \ar[rr]  \ar@{-->}[urr]
&&
\cK
}
\]
admits a unique filler.

\end{itemize}
We summarize:
\begin{itemize}
\item[~]

To construct the functor $\cK(-,x)\colon  \cK^{\op}\to \Spaces$ we construct a right fibration $\cK_{/x}\to \cK$; the latter which amounts to specifying $\cJ$-points of $\cK_{/x}$ over $\cJ$-points $\cK$, more precisely a presheaf on $\Cat_{/\cK}$, followed by a series of {\bf checks}.  

\end{itemize}
We see this as an adaptable technique for constructing a functor $\cK \to \cZ$ whenever $\cZ$ can be recognized as classifying certain types of fibrations. It is the essential technique we use for constructing functors between $\oo$-categories of differential topological origin.

\subsection{Questions/problems}

The following technical questions and problems are suggested by this work.

\begin{problem}
Identify the class of functors which have the left lifting property with respect to exponentiable fibrations.
\end{problem}

\begin{q}
Are epimorphisms among $\infty$-categories, or are localizations among $\infty$-categories, closed under base change along exponentiable fibrations?
\end{q}

\begin{q}
What are some practical criteria for detecting (co)limit diagrams, such as pushouts and pullbacks, in $\Corr$?
\end{q}

\begin{q}\label{q.2.Corr}
What is the native $(\infty,2)$-categorical enhancement of the flagged $\infty$-category $\Corr$ of Definition~\ref{def.Corr}?

\end{q}

\begin{q}\label{q.n.Corr}
What are some $(\infty,n)$-categorical counterparts of an exponentiable fibration?
(The plural is used to accommodate various conceivable degrees of laxness.)
For each such notion, is there a comprehensible flagged $(\infty,n)$-category that classifies this notion?  
\end{q}

The next problem makes reference to the following construction.
Denote by ${\sf Idem}\subset {\sf Ret}$ the $\infty$-categories corepresenting an idempotent and a retraction, respectively.  
Let $\cK$ be an $\infty$-category.
Consider the full $\infty$-subcategory $\Cat_{/^{\sf idem}\cK}\subset \Cat_{/\cK}$ consisting of those functors $\cE\to \cK$ that are \emph{idempotent complete}: i.e., each solid diagram of $\infty$-categories
\[
\xymatrix{
{\sf Idem}  \ar[rr]  \ar[d]
&&
\cE  \ar[d]
\\
{\sf Ret}  \ar[rr]  \ar@{-->}[urr]
&&
\cK
}
\]
admits a filler (which will necessarily be unique since ${\sf Idem}\to {\sf Ret}$ is an epimorphism).  
Presentability of $\Cat$ offers a localization functor $(-)^{\widehat{~}}_{\sf idem}\colon \Cat_{/\cK}  \rightleftarrows \Cat_{/^{\sf idem}\cK}$ which is left adjoint to the inclusion; we refer to the values of this left adjoint as \emph{idempotent completion} of a functor.  
\begin{problem}\label{q.idemp}
Consider the full flagged $\infty$-subcategory $\Corr^{\omega}\subset \Corr$ consisting of the idempotent complete $\infty$-categories.
Show that this flagged $\infty$-subcategory is in fact an $\infty$-category.
Show that the assignment of $\cK$-points
\[
(\cE\xra{\sf e.fib} \cK)\mapsto (\cE^{\widehat{~}}_{\sf idem} \to \cK)
\]
defines a functor $\Corr \to \Corr^{\omega}$.
Show that this functor identifies $\Corr^{\omega}$ as the underlying $\infty$-category of the flagged $\infty$-category $\Corr$.  

\end{problem}

The next problem involves the relation of correspondences with spans. A first account of $\infty$-categories of spans in an $\infty$-category is given in \cite{barwick-q} and further studied in \cite{barwick}. An $(\oo,2)$-categorical account is given in \cite{gaitsgory.rozenblyum}, and an $(\infty,n)$-categorical account is given in \cite{haugseng}. From a span $\cE_0 \la \cX \ra \cE_1$, one can construct a correspondence as the parametrized join $\cE_0\bigstar_{\cX}\cE_1\ra [1]$. Conversely, from a correspondence $\cE\ra [1]$, one can construct a span as the $\oo$-category of sections $\cE_{|0}\la \Gamma(\cE\ra [1])\ra \cE_{|1}$. This resulting span has the special property that the functor
\[
\Gamma(\cE\ra [1])\longrightarrow\cE_{|0}\times\cE_{|1}
\]
is a bifibration; see Lemma \ref{equiv.correspondences}.

\begin{problem}\label{q.spans}
Relate $\corr$ and ${\Span}(\Cat)$, the $\oo$-category of spans of $\oo$-categories. In particular, taking sections $\Gamma$ defines a {\it lax} functor from $\corr$ to ${\Span}(\Cat)$. This laxness appears due to the necessity of localizing in the composition rule (\ref{comp.bifib}) of Lemma \ref{composition}. Show that parametrized join defines a functor from ${\Span}(\Cat)$ to $\corr$, and that this functor restricts to an equivalence between the $\infty$-category $\Corr[\Spaces]$ and the $\infty$-category $\Span(\spaces)$ of spans among spaces.
\end{problem}

\begin{q}\label{q.cons.B}
Does $\infty$-groupoid completion define a, possibly symmetric monoidal, functor $\sB\colon \Corr \to \Corr[\spaces]$ between flagged $\infty$-categories?
Supposing not, is there a largest flagged $\infty$-subcategory of $\Corr$ on which this is defined?
For instance, is there a flagged $\infty$-subcategory of $\Corr$ classifying exponentiable fibrations $\cE\to \cK$ for which, for each morphism $c_1\xra{\lag x\to y\rag}\cK$ and each morphism $e_t\to e'_t$ in the fiber $\infty$-category $\cE_{|y}$, the canonical functor between $\infty$-overcategories $(\cE_{|s})_{/e_t} \to (\cE_{|s})_{/e'_t}$ induces an equivalence on classifying spaces?
(See the proof of Lemma~\ref{l.final.B.fib}.)

\end{q}

\begin{q}\label{cCart.lifting}
What is the class of functors that have the left lifting property with respect to Cartesian fibrations?  
What is a checkable condition on a functor between $\infty$-categories $\cJ_0 \to \cJ$ to be \emph{lax final}, by which we mean restriction of diagrams along this functor determines an equivalence on lax colimits? 
Is the condition that the functor is a right adjoint?
(See~\S\ref{sec.l.fib.lifting} for a discussion in the case of right fibrations.)

\end{q}

\begin{problem}\label{q.pshv}
Give a direct $\cK$-point description of a fully-faithful filler among flagged $\infty$-categories as in this diagram:
\[
\xymatrix{
\Cat  \ar[d]_-{\rm univ.Cart}  \ar[rr]^-{\PShv}
&&
{\sf Pr^L}_{/^{\sf lax}\Spaces}  
\\
\Corr^{\op}  \ar@{-->}[urr]_-{\PShv}
&&
.
}
\]
(Here, ${\sf Pr^L}_{/^{\sf lax}\Spaces}$ is an $\infty$-category for which a $\cK$-point is a diagram of $\infty$-categories
\[
\xymatrix{
\cP \ar[rr]    \ar[d]
&&
\Spaces
\\
\cK
&
}
\]
in which the vertical functor has presentable fibers and is a coCartesian fibration as well as a Cartesian fibration, and the horizontal functor preserves colimit diagrams in each fiber over $\cK$.)
\end{problem}

\begin{problem}\label{q.rcorr.pshv}
Premised on an answer to Problem~\ref{q.pshv}, find a direct $\cK$-point description of a lift among flagged $\infty$-categories,
\[
\xymatrix{
\RCorr^{\op}  \ar@{-->}[rr]^-{\PShv}  \ar[d]
&&
({\sf Pr^L})_{/\Spaces}  \ar[d]
\\
\Corr^{\op}  \ar[rr]^-{\PShv}
&&
{\sf Pr^L}_{/^{\sf lax}\Spaces},
}
\]
making this square a limit diagram.
\end{problem}

\subsection{Conventions}

\begin{remark}

In this work, we use the quasi-category model for $\oo$-categories, as developed by Joyal \cite{joyal1} and Lurie \cite{HTT}. However, this choice is primarily for ease of reference. All of the arguments herein are homotopy-invariant and so translate to any model for $\oo$-categories. We could just as well have used, for instance, Rezk's complete Segal spaces \cite{rezk}.

\end{remark}

\begin{terminology}
Given a simplicial space $\bDelta^{\op}\xra{\cF}\Spaces$ that satisfies the Segal condition, we say it is \emph{univalent}, or satisfies the \emph{univalence condition}, if it is \emph{complete}, or satisfies the \emph{completeness condition}, in the sense of~\cite{rezk}.  
\end{terminology}

\begin{convention}[Univalence/completeness]
After Rezk~\cite{rezk} and Joyal--Tierney~\cite{joyaltierney}, the restricted Yoneda functor associated to the inclusion $\bDelta\hookrightarrow \Cat$ induces an equivalence $\Cat \subset \psh(\bdelta)$ with the $\oo$-subcategory of Segal presheaves satisfying the {\it univalence}, or {\it completeness}, condition. In the terms of Theorem~\ref{flagged.thm}, univalence is equivalent to the condition on a flagged $\oo$-category $(\cG \ra \cC)$ that the natural map $\cG \ra \cC^\sim$ is an equivalence, i.e., that the underlying $\oo$-groupoid is equivalent to the space of isomorphisms.
\end{convention}

\begin{convention}
For $\cC\hookrightarrow \cD$ a fully-faithful functor between $\infty$-categories, we typically do not distinguish in notation or terminology between objects/morphisms in $\cC$ and their images in $\cD$. 
In particular, we regard posets, as well as ordinary categories, as $\infty$-categories without notation or further comment.  

\end{convention}

\begin{terminology}
The \emph{1-cell} is the poset $c_1:=\{s\to t\}$.

\end{terminology}

\begin{convention}
Unless stated otherwise, all diagrams are commutative.  
\end{convention}

\begin{convention}
For $x,y\in \cC$ two objects in an $\infty$-category, $\cC(x,y)$ is the space of morphisms in $\cC$ from $x$ to $y$.  
\end{convention}

\begin{terminology}

We require the use of both small, large, and very large $\oo$-categories in this work. Our convention is that all-capitalization represents the very large version of an $\oo$-category. In particular, $\Cat$ is the $\oo$-category whose objects are small $\oo$-categories; $\CAT$ is the very large $\oo$-category whose objects are $\oo$-categories. In particular, $\Cat$ is an object in $\CAT$. We use matching notation for the $\oo$-category $\Spaces$ and the very large $\oo$-category $\SPACES$.

\end{terminology}

\begin{terminology}\label{mxml.gpd}
The fully-faithful functor
\[
\Spaces~\hookrightarrow~ \Cat
\]
has both a left and a right adjoint.  
Let $\cC$ be an $\infty$-category.
We denote the value of this left adjoint on $\cC$ as $\sB\cC$, referring to it as the \emph{classifying space of $\cC$}, or sometimes as its \emph{$\infty$-groupoid completion}.
We denote the value of this right adjoint on $\cC$ as $\cC^\sim$, referring to it as its \emph{maximal $\infty$-subcategory}.  
In particular, for each $\infty$-groupoid $\cG$, the canonical map between spaces of morphisms
\[
\Spaces(\cG,\cC^\sim) \xra{~\simeq~} \Cat(\cG,\cC)
\]
is an equivalence.

\end{terminology}

\begin{convention}
For $\oo$-categories $\cC$ and $\cD$, we denote the $\oo$-category of functors as $\Fun(\cC,\cD)$. The {\it space} of functors is its underlying $\oo$-groupoid:
\[
\Map(\cC,\cD) = \Fun(\cC,\cD)^\sim~.
\]
We also denote this space of maps as $\Cat(\cC,\cD)$.
\end{convention}

\begin{convention}
For $\cC\ra \cD$ a functor and $d\in\cD$ an object, we write the associated $\oo$-overcategory as
\[
\cC_{/d} := \cC\underset{\cD}\times\cD_{/d}
\]
and likewise the $\oo$-undercategory as $\cC^{d/} := \cC\times_\cD \cD^{d/}$.
\end{convention}

\begin{terminology}
A functor $F\colon \cC\to \cD$ between $\infty$-categories is a \emph{localization} if, for each $\infty$-category $\cZ$, precomposition with $F$ defines a monomorphism between spaces of functors,
\[
-\circ F\colon 
\Cat(\cD,\cZ)
\longrightarrow
\Cat(\cC,\cZ)~,
\]
with image consisting of those functors $\cC\to \cZ$ that carry a morphism in $\cC$ to an equivalence in $\cZ$ whenever it is carried to an equivalence in $\cD$.
In this case, for $W:= F^{-1}(\cD^{\sim}) \subset \cC$, we write
\[
\cC[W^{-1}]
~:=~
\cD~.
\]
\end{terminology}
Note that a localization is, in particular, an epimorphism in $\Cat$.

\begin{terminology}\label{def.p.join}
Let $\cA \la \cX \ra \cB$ be a diagram of $\infty$-categories.
Their \emph{parametrized join} is the iterated pushout
\[
\cA \underset{\cX} \bigstar \cB
~:=~
\cA \underset{\cX\times \{s\}}\coprod \cX\times c_1 \underset{\cX\times \{t\}}\coprod \cB~.
\]
The unique morphism between diagrams $(\cA\la \cX \to \cB) \ra(\ast \la \ast \ra \ast)$ equips the parametrized joint with a canonical functor
\[
\cA \underset{\cX} \bigstar \cB
\longrightarrow
\ast \underset{\ast}\bigstar \ast
~\simeq~
c_1
\]
to the 1-cell.

\end{terminology}

\begin{remark}
We make substantial use of \emph{final} and \emph{initial} functors.  
We refer the reader to~\S\ref{subsec.final} for definitions and some observations concerning their basic features.  
	
\end{remark}

\subsection{Acknowledgements}
We thank both referees for considered and expert reviews of this work.
We are grateful to one of the referees for suggesting a simpler proof of Lemma~\ref{adjointcriterion}.
We thank Aaron Mazel-Gee for pointing out a counterexample to the statement of Lemma~\ref{lemma.new.loc.cocart} which appeared in an earlier version of this paper.

\section{Correspondences}\label{sec.corr}
We now construct a flagged $\infty$-category of $\infty$-categories and \emph{correspondences} among them.  
We then show that this flagged $\infty$-category classifies exponentiable fibrations.

\subsection{Correspondence between two $\infty$-categories}
We define a \emph{correspondence} beteen two $\infty$-categories.

\begin{definition}\label{def.a.corr}
A \emph{correspondence (from an $\infty$-category $\cE_s$ to an $\infty$-category $\cE_t$)} is a pair of pullback diagrams among $\infty$-categories,
\[
\xymatrix{
\cE_s  \ar[r]  \ar[d]
&
\cE  \ar[d]
&
\cE_t  \ar[l]  \ar[d]
\\
\{s\} \ar[r]
&
c_1
&
\{t\}  ~.\ar[l]  
}
\]
The space of correspondences from $\cE_s$ to $\cE_t$ is the maximal $\infty$-subgroupoid of
\[
\{\cE_s\}\underset{\Cat}\times \Cat_{/c_1} \underset{\Cat}\times \{\cE_t\}~.
\]
\end{definition}

\begin{example}\label{ex.identity.corr}
Let $\cC$ be an $\infty$-category.
The \emph{identity correspondence} is the projection $\cC\times c_1\xra{\pr}c_1$.  

\end{example}

\begin{remark}
Given a correspondence $\cE_{01} \to \{0<1\}$ from $\cE_0$ to $\cE_1$, and another $\cE_{12}\to \{1<2\}$ from $\cE_1$ to $\cE_2$, taking a pushout over $\cE_1$ determines an $\infty$-category
\[
\cE_{012}:=\cE_{01}\underset{\cE_1}\amalg \cE_{12}  \longrightarrow \{0<1\}\underset{\{1\}}\amalg \{1<2\} = [2]
\]
over $[2]$.  
Base change along $\{0<2\} \to [2]$ then determines an $\infty$-category $\cE_{02}\to\{0<2\}$, which is a correspondence from $\cE_0$ to $\cE_2$.   
This suggests that correspondences form the morphisms in an $\infty$-category -- an assertion which our main result (Theorem~\ref{main.thm}) states is essentially correct. 
This also suggests that this $\infty$-category is presented as a simplicial space for which a 2-simplex is an $\infty$-category $\cE\to [2]$; equivalently, the datum of a pair of composable correspondences, together with a choice of composite, is the datum of an $\infty$-category $\cE\to [2]$ over $[2]$.  
This, however, is \emph{not} correct.  
The obstruction is manifested in the non-contractibility of the collection of $\infty$-categories $\cE\to [2]$ over $[2]$ with specified restrictions over $\{0<1\}$ and $\{1<2\}$.  
The key to circumnavigating this obstruction is to observe that the $\infty$-category $\cE_{012}\to [2]$ over $[2]$ is \emph{not} an arbitrary $\infty$-category over $[2]$, but is in fact an \emph{exponentiable fibration} over $[2]$.  
The content of the coming subsection justifies that, indeed, there is an $\infty$-category whose morphisms are correspondences, and which is presented as a simplicial space for which a $p$-simplex is an exponentiable fibration over $[p]$.

\end{remark}

\subsection{Exponentiable fibrations}\label{sec.efibs}
Recall Definition~\ref{def.exp-fib} of an exponentiable fibration.

\begin{observation}\label{exp-op}
A functor $\cE\to \cK$ is an exponentiable fibration if and only if its opposite $\cE^{\op}\to \cK^{\op}$ is an exponentiable fibration.

\end{observation}

\begin{example}\label{over-point}
Each functor $\cE \to \ast
$ to the terminal $\oo$-category is an exponentiable fibration. The right adjoint to base change along this functor is \emph{global sections}:
\[
\Gamma: \Cat_{/\cE}\longrightarrow \Cat_{/\ast} = \Cat~,\qquad (\cX\to \cE) \mapsto \Fun_{/\cE}(\cE,\cX)~=:~\Gamma(\cX\ra \cE)~.
\]
\end{example}

\begin{example}\label{non-example}
The inclusion $\{0<2\}\hookrightarrow [2]$ is not an exponentiable fibration. For instance, base change along $\{0<2\}\hookrightarrow [2]$ fails to preserve the colimit $\{0<1\}\underset{\{1\}}\amalg \{1<2\} \xra{\simeq} [2]$.
\end{example}

Exponentiable fibrations are precisely those for which the next definition has meaning.
\begin{definition}\label{def.rel.fun}
Let $\cE\xra{\pi}\cK$ be an exponentiable fibration.
For each $\infty$-category $\cZ\to \cE$ over $\cE$, the $\infty$-category of \emph{relative functors (over $\cK$)}
\[
\bigl(\Fun_{\cK}^{\sf rel}(\cE,\cZ) \to\cK\bigr)~:=~\pi_\ast(\cZ\to \cE)
\]
is the $\infty$-category over $\cK$ that is the value of the right adjoint to the base change functor $\pi^\ast$ on $\cZ\to \cE$.

\end{definition}

\begin{observation}\label{rel.fun.global.sec}
Let $\cE\xra{\pi}\cK$ be an exponentiable fibration; let $\cZ\to \cE$ be an $\infty$-category over $\cE$.
For an $\infty$-category $\cJ\to \cK$ over $\cK$, there is a canonical identification between $\infty$-categories
\[
\Fun_{/\cK}\bigl(\cJ,\Fun_{\cK}^{\sf rel}(\cE,\cZ)\bigr)~\simeq~\Fun_{/\cE}(\cE_{|\cJ},\cZ)~.
\]
In particular, the global sections of the relative functor $\infty$-category
\[
\Fun_{/\cK}\bigl(\cK,\Fun_{\cK}^{\sf rel}(\cE,\cZ)\bigr)~\simeq~\Fun_{/\cE}(\cE,\cZ)
\]
is identified as the global sections of $\cZ$.

\end{observation}

\begin{prop}\label{exps.compose}
Let $\cE \xra{\pi} \cK \xra{\pi'} \cU$ be a composable sequence of functors between $\infty$-categories.
If both $\pi$ and $\pi'$ are exponentiable fibrations, then the composition $\pi'\circ\pi$ is an exponentiable fibration.

\end{prop}

\begin{proof}
Directly, the canonical morphism $(\pi'\circ\pi)_! \xra{\simeq}\pi'_!\circ \pi_!$ between functors $\Cat_{/\cE}\to \Cat_{/\cU}$ is an equivalence, as indicated.
It follows that the canonical morphism $\pi^\ast \circ {\pi'}^\ast \xra{\simeq} (\pi'\circ \pi)^\ast$ between functors $\Cat_{/\cU}\to \Cat_{/\cE}$ is an equivalence.  
By assumption, both $\pi^\ast$ and ${\pi'}^\ast$ are left adjoints.
Because the composition of left adjoints is a left adjoint, the composition $\pi^\ast \circ {\pi'}^\ast \simeq (\pi'\circ \pi)^\ast$ is a left adjoint, as desired.

\end{proof}

Our next main result, Lemma~\ref{exp-char}, gives useful criteria for exponentiability. In order to prove it, we will need to the following technical result for computing spaces of morphisms in certain pushouts.

\begin{lemma}\label{pushout}
Let $p>0$ be a positive integer.
Let
\[
\xymatrix{
\cE_{01}   \ar[d]
&
\cE_1   \ar[d]   \ar[r]  \ar[l]
&
\cE_{1p}    \ar[d]
\\
\{0<1\}
&
\{1\}   \ar[r]  \ar[l]
&
\{1<\dots<p\}
}
\]
be a diagram of $\infty$-categories in which each square is a pullback.
Consider the functor $\cE\to [p]$ between pushouts of the horizontal diagrams. Let $e_i,e_j\in \cE$ be objects over $i\leq j\in [p]$.
Then the space of morphisms in $\cE$ from $e_i$ to $e_j$ abides by the following expressions.
\begin{enumerate}
\item[~]
\begin{enumerate}
\item If $0\leq i \leq j \leq 1$, the canonical map between spaces
\[
\cE_{01}(e_i,e_j) \longrightarrow \cE(e_i,e_j)
\]
is an equivalence.

\item If $1\leq i \leq j \leq p$, the canonical map between spaces
\[
\cE_{1p}(e_i,e_j) \longrightarrow \cE(e_i,e_j)
\]
is an equivalence.

\item If $0=i<1<j\leq p$, composition defines a map from the coend
\[
\cE_{01}(e_0,-)\underset{\cE_{1}}\bigotimes \cE_{1p}(-,e_j)
\xra{~\circ~}
\cE(e_0,e_j)~,
\]
which is an equivalence between spaces.
\end{enumerate}
\end{enumerate}

\end{lemma}

\begin{proof}
Here is our strategy of proof.
We construct a simplicial space over $[p]$ via a left Kan extension, designed to achieve the colimit expression of~(c).
We verify that this simplicial space over $[p]$ satisfies the Segal and univalence conditions.
Finally, we verify that this univalent Segal space over $[p]$ possesses the universal property of the pushout, $\cE$.

For each object $[p]\in \bDelta$, consider the full subcategory $\cP_{[p]}  \subset \bDelta_{/[p]}$ of the overcategory consisting of those morphisms $[q] \to [p]$ for which the canonical diagram of $\infty$-categories
\[
\xymatrix{
[q]_{|\{1\}}  \ar[r]  \ar[d]
&
[q]_{|\{1<\dots<p\}}  \ar[d]
\\
[q]_{|\{0<1\}}  \ar[r]
&
[q]
}
\]
is a pushout.
Explicitly, $\cP_{[p]}$ consists of those $[q]\xra{\rho} [p]$ for which $\rho^{-1}(1) \neq \emptyset$ is not empty.  
Consider the full $\infty$-subcategory $Z \subset \Cat_{/[p]}$ consisting of the three objects $\{0<1\}\to [p]$, $\{1\}\to [p]$, and $\{1<\dots<p\}\subset [p]$.
This $\infty$-category corepresents a zig-zag among three objects.
Denote the composite functor $\alpha \colon Z\hookrightarrow \Cat_{/[p]} \to \Cat$.
Base change along each $(S\to [p])\in Z$ determines a functor
\begin{equation}\label{12}
\cP_{[p]} \longrightarrow \Fun(Z , \Cat)_{/\alpha}~.
\end{equation}

Let $\cZ\to [p]$ be an $\infty$-category over $[p]$; we use the same notation $\cZ\colon (\bDelta_{/[p]})^{\op} \to \Spaces$ for the restricted Yoneda presheaf.
The definition of $\cP_{[p]}$ ensures that, for each object $([q]\to [p])\in \cP_{[p]}$, the canonical diagram of spaces
\begin{equation}\label{32}
\xymatrix{
\cZ([q])  \ar[r]  \ar[d]
&
\cZ([q]_{|\{0<1\}}) \ar[d]
\\
\cZ([q]_{|\{1<\dots<q\}})  \ar[r]
&
\cZ([q]_{|\{1\}})
}
\end{equation}
is a pullback.
As a consequence, the canonical lax commutative diagram of $\infty$-categories
\begin{equation}\label{34}
\xymatrix{
\cP_{[p]}^{\op}  \ar@(u,u)[rrrr]^-{~}  \ar[d]_-{(\ref{12})}
&&
\Downarrow
&&
(\bDelta_{/[p]})^{\op}  \ar[d]^-{\cZ}
\\
\Fun(Z , \Cat)_{/\alpha}^{\op}  \ar[rrrr]^-{\Hom(-,\cZ_\bullet)}
&&
&&
\Spaces
}
\end{equation}
in fact commutes; here $\cZ_\bullet\in \Fun(Z,\Cat)_{/\alpha}$ is the diagram $\cZ_\bullet:=\bigl((\cZ_{01} \la \cZ_{0} \to \cZ_{1p})\to (\{0<1\}\la \{1\} \to \{1<\dots<p\})\bigr)$ given by way of base change.

Now, the diagram $\cE_\bullet :=\bigl((\cE_{01} \la \cE_{0} \to \cE_{1p})\to (\{0<1\}\la \{1\} \to \{1<\dots<p\})\bigr)$ defines an object in the $\infty$-category $ \Fun(Z,\Cat)_{/\alpha}$.
Precomposition with the presheaf represented by this object is a presheaf
\[
\w{\cE}\colon \cP_{[p]}^{\op} \xra{~(\ref{12})~}  ( \Fun(Z , \Cat)_{/\alpha})^{\op} \xra{~{\sf Hom}(-,\cE_\bullet)~} \Spaces~.
\]
Denote by
\[
\cE\colon (\bDelta_{/[p]})^{\op} \longrightarrow  \Spaces~,\qquad ([q]\to [p])\mapsto \cE([q])~,
\]
the presheaf which is the left Kan extension as in this diagram:
\[
\xymatrix{
\cP_{[p]}^{\op}  \ar[rr]^-{\w{\cE}}  \ar[d]
&&
\Spaces
\\
(\bDelta_{/[p]})^{\op}  \ar[urr]_-{{~}{~}{~}\cE:={\sf LKan}}
&&
.
}
\]
We first identify some values of $\cE$.
Let $([q]\to [p])$ be an object in $\bDelta_{/[p]}$.
\begin{itemize}
\item
Suppose $([q]\to [p])$ belongs to the full subcategory $\cP_{[p]}$.
In this case, we identify the value
\begin{eqnarray}
\label{20}
\cE([q])
&
\simeq
&
\w{\cE}([q])
\\
\nonumber
&
\simeq
&
\Map_{\Fun(Z,\Cat)_{/\alpha}}\bigl(([q]_{|\{0<1\}} \la [q]_{|\{1\}} \to [q]_{|\{1<\dots<p\}}), (\cE_{01}\la \cE_1 \to \cE_{1p})\bigr)
\\
\nonumber
&
\simeq
&
\cE_{01}([q]_{|\{0<1\}})\underset{  \cE_{1}([q]_{|\{1\}}) }\times \cE_{1p}([q]_{|\{1<\dots<p\}})~.
\end{eqnarray}
(In the last expression, for $\cK$ an $\infty$-category, and for $\cJ\to \cK$ and $\cC\to \cK$ two $\infty$-categories over $\cK$, we denote $\cC(\cJ):=\Cat_{/\cK}(\cJ,\cC)$ for the space of functors over $\cK$ from $\cJ$ to $\cC$.)

\item
Suppose that the object $([q]\to [p])$ does not belong to the full subcategory $\cP_{[p]}$.
Being a left Kan extension, the value is computed as a colimit,
\[
\cE([q])~\simeq~\colim\bigl((\cP_{[p]}^{([q]\to [p])/})^{\op}\simeq (\cP_{[p]})^{\op}_{/([q]\to [p])}  \to \cP_{[p]}^{\op} \xra{\w{\cE}}\Spaces \bigr)~,
\]
which we now simplify.
Consider the pullback functor
\[
 \cP_{[p]}^{([q]\to [p])/} \longrightarrow \Cat~,\qquad ([q]\to [q']\to [p])\mapsto [q']_{|\{1\}}~.
\]
The definition of $\cP_{[p]}$ is just so that this functor factors as
$
\cP_{[p]}^{([q]\to [p])/} \to  \bDelta .
$
This factorized functor is a right adjoint, with left adjoint
\[
\bDelta \longrightarrow \cP_{[p]}^{([q]\to [p])/} ~,\qquad [r]\mapsto ([q]\to [q]_{|\{0\}}\star [r] \star [q]_{|\{2<\dots<p\}}  \to [p])~,
\]
given in terms of joins.
In particular, this functor $\bDelta^{\op} \to (\cP_{[p]}^{([q]\to [p])/})^{\op}$ is final.
Combined with the identification~(\ref{20}), we thereby identify the value of $\cE$ on $([q]\to [p])$ as
\begin{eqnarray}
\label{30}
\cE([q])
&
\simeq
&
\colim\Bigl(\bDelta^{\op} \to (\cP_{[p]}^{([q]\to [p])/})^{\op}\simeq (\cP_{[p]})^{\op}_{/([q]\to [p])}  \to \cP_{[p]}^{\op} \xra{\w{\cE}}\Spaces \Bigr)
\\
\nonumber
&
\simeq
&
\bigl |  \cE_{01}([q]_{|\{0\}}\star [\bullet])\underset{\cE_1([\bullet])}\times \cE_{1p}([\bullet]\star [q]_{|\{2<\dots<p\}})  \bigr |~.
\end{eqnarray}

\end{itemize}

We now verify that $\cE$ presents an $\infty$-category over $[p]$.
Specifically, we show that $\cE$ satisfies the Segal condition over $[p]$, and the univalence condition over $[p]$.
Let $[q]\to [p]$ be an object in $\bDelta_{/[p]}$.
Consider the canonical diagram of spaces:
\begin{equation}\label{13}
\xymatrix{
\cE([q])  \ar[r]  \ar[d]
&
\cE(\{0<1\}) \ar[d]
\\
\cE(\{1<\dots<q\})  \ar[r]
&
\cE(\{1\}) .
}
\end{equation}
We show this square is a pullback through a few cases.
\begin{itemize}
\item Suppose the object $([q]\to [p])$ belongs to the full subcategory $\cP_{[p]}\in \bDelta_{/[p]}$.
In this case, the square
\begin{equation}\label{14}
\xymatrix{
(\{1\}\to [p])  \ar[r]  \ar[d]
&
(\{0<1\}\to [p])  \ar[d]
\\
(\{1<\dots<q\}\to [p])    \ar[r]
&
([q]\to [p])
}
\end{equation}
in $\bDelta_{/[p]}$ in fact belongs to the full subcategory $\cP_{[p]}\subset \bDelta_{/[p]}$.
From the definition of $\cE$ as the left Kan extension along the fully-faithful inclusion $\cP_{[p]}^{\op} \hookrightarrow (\bDelta_{/[p]})^{\op}$, the square~(\ref{13}) is identified as the square
\begin{equation}\label{15}
\xymatrix{
\w{\cE}([q])  \ar[r]  \ar[d]
&
\w{\cE}(\{0<1\}) \ar[d]
\\
\w{\cE}(\{1<\dots<q\})   \ar[r]
&
\w{\cE}(\{1\}) .
}
\end{equation}
Observe that the functor~(\ref{12}) carries the square~(\ref{14}) in $\cP_{[p]}$ to a pushout square in the $\infty$-category $\Fun(Z , \Cat)_{/\alpha}$.
Consequently, as the functor $\w{\cE}$ is a restricted Yoneda functor, the diagram~(\ref{15}) is a pullback, as desired.

\item Suppose the object $([q]\to [p])$ does not belong to the full subcategory $\cP_{[p]}\subset \bDelta_{/[p]}$.
From the explicit description of $\cP_{[p]}$, the image of $[q]\to [p]$ does not contain $1\in [p]$.  
In particular, the composition $\{1\}\hookrightarrow [q]\to [p]$ factors through either $\{0\}\hookrightarrow [p]$ or $\{2<\dots<p\}\hookrightarrow [p]$.  
There are two cases.
\begin{itemize}

\item Suppose the composite functor $\{1\} \hookrightarrow [q]\to [p]$ factors through $\{0\}\hookrightarrow [p]$.
In this case the expression~(\ref{30}) identifies the square~(\ref{15}) as
\[
\Small
\xymatrix{
\bigl |  \cE_{01}([q]_{|\{0\}}\star [\bullet])\underset{\cE_1([\bullet])}\times \cE_{1p}([\bullet]\star [q]_{|\{2<\dots<p\}})  \bigr |  \ar[rr]  \ar[d]
&&
\cE_{01}(\{0<1\})   \ar[d]
\\
\bigl |  \cE_{01}([q-1]_{|\{0\}}\star [\bullet])\underset{\cE_1([\bullet])}\times \cE_{1p}([\bullet]\star [q-1]_{|\{2<\dots<p\}})  \bigr | \ar[rr]
&&
\cE_{01}(\{1\}),
}
\]
where, here, we use the condensed notation $[q-1]:=\{1<\dots<q\}\subset [q]$.
Because the formation of joins preserves colimits in each of its arguments, for each $[r]\in \bDelta$, the canonical functor from the pushout
\[
\{0<1\}   \underset{ \{1\}} \amalg ([q-1]_{|\{0\}}\star [r]) \xra{~\simeq~} [q]_{|\{0\}}\star[r]
\]
is an equivalence between $\infty$-categories over $\{0<1\}$.
Because $\cE_{01} \to \{0<1\}$ is an $\infty$-category over $\{0<1\}$, using that base change in $\Spaces$ preserves colimits, we identify this last square as the canonical square among spaces
\[
\Small
\xymatrix{
\cE_{01}(\{0<1\})\underset{\cE_{1}(\{1\})  }\times
\bigl |  \cE_{01}([q-1]_{|\{0\}}\star [\bullet])\underset{\cE_1([\bullet])}\times \cE_{1p}([\bullet]\star [q-1]_{|\{2<\dots<p\}})  \bigr |  \ar[rr]  \ar[d]
&&
\cE_{01}(\{0<1\})   \ar[d]
\\
\bigl |  \cE_{01}([q-1]_{|\{0\}}\star [\bullet])\underset{\cE_1([\bullet])}\times \cE_{1p}([\bullet]\star [q-1]_{|\{2<\dots<p\}})  \bigr |  \ar[rr]
&&
\cE_{01}(\{1\}).
}
\]
This square is a pullback, as desired.

\item Suppose the composite functor $\{1\} \hookrightarrow [q]\to [p]$ factors through $\{2<\dots<p\}\hookrightarrow [p]$.
This case is nearly identical to that above; we omit the details.

\end{itemize}

\end{itemize}
We now establish that $\cE$ is univalent over $[p]$.
Consider a diagram of $\infty$-categories:
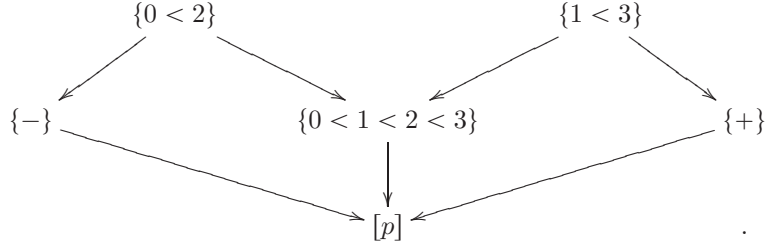
\[
\xymatrix{
&
\{0<2\}  \ar[dl]  \ar[dr]
&&
\{1<3\}  \ar[dl]  \ar[dr]
&
\\
\{-\}  \ar[drr]
&&
\{0<1<2<3\}  \ar[d]
&&
\{+\}  \ar[dll]
\\
&&
[p]
&&.
}
\]
This is the datum of a functor from the pushout
\[
\{-\}\underset{\{0<2\}}\amalg \{0<1<2<3\}\underset{\{1<3\}}\amalg \{+\} \xra{~\simeq~} \ast \xra{~\lag i\rag~} [p]~.
\]
We must show that the canonical diagram involving spaces of lifts
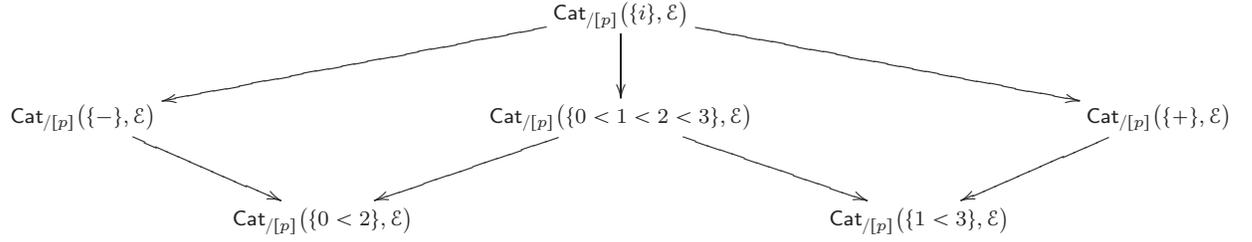
\[
\Small
\xymatrix{
&&
\Cat_{/[p]}\bigl(\{i\},\cE\bigr)  \ar[d]  \ar[drr]  \ar[dll]
&&
\\
\Cat_{/[p]}\bigl(\{-\} , \cE\bigr)  \ar[dr]
&&
\Cat_{/[p]}\bigl(\{0<1<2<3\}, \cE\bigr)   \ar[dr]  \ar[dl]
&&
\Cat_{/[p]}\bigl(\{+\}, \cE\bigr)   \ar[dl]
\\
&
\Cat_{/[p]}\bigl(\{0<2\}, \cE\bigr)
&&
\Cat_{/[p]}\bigl(\{1<3\}, \cE\bigr)
&
}
\]
is a limit diagram.
Through the composition-restriction adjunction $\Cat_{/\{i\}} \rightleftarrows \Cat_{/[p]}$, this diagram of spaces is identified as the diagram of spaces
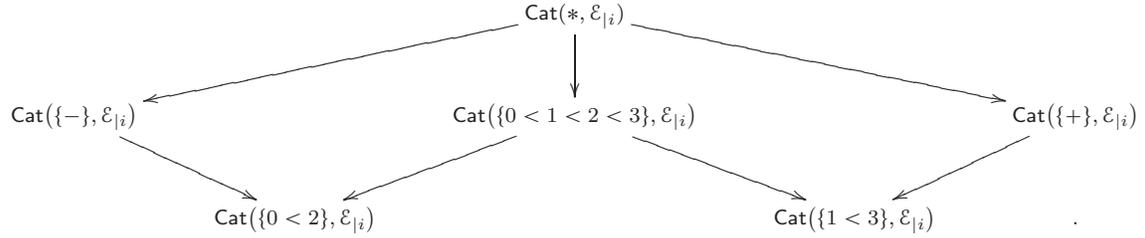
\[
\Small
\xymatrix{
&&
\Cat(\ast , \cE_{|i})   \ar[d]  \ar[drr]  \ar[dll]
&&
\\
\Cat\bigl(\{-\},\cE_{|i}\bigr)  \ar[dr]
&&
\Cat\bigl(\{0<1<2<3\}, \cE_{|i}\bigr)   \ar[dr]  \ar[dl]
&&
\Cat\bigl(\{+\},\cE_{|i}\bigr) \ar[dl]
\\
&
\Cat\bigl(\{0<2\}, \cE_{|i}\bigr)
&&
\Cat\bigl(\{1<3\}, \cE_{|i}\bigr)
&.
}
\]
From the definition of $\cE$, we have equivalences $\cE_{|i}\simeq (\cE_{01})_{|i}$ if $i\leq 1$ and $\cE_{|i}\simeq (\cE_{1p})_{|i}$ if $i\geq 1$.
Consequently, this square is a limit diagram precisely because $\cE_{01}$ and $\cE_{1p}$ are each $\infty$-categories over $[p]$.
This concludes the verification that $\cE \to [p]$ is an $\infty$-category over $[p]$.

The construction of this $\infty$-category $\cE$ over $[p]$ was tailored to satisfy the conditions of this result.
Namely, (a) follows directly from the expression~(\ref{20}), applied to the case that $[q]=[1]$ and the morphism $[q]\to [p]$ factors through $\{0<1\}\hookrightarrow [p]$.
Statement~(b) follows directly from expression~(\ref{20}), applied to the case that $[q]=[1]$ and the morphism $[q]\to [p]$ factors through $\{1<\dots<p\}\hookrightarrow [p]$.
Statement~(c) follows directly from expression~(\ref{30}), applied to the case that $[q]=[1]$ and the morphism $[q]\to [p]$ does not factor through either of the monomorphisms $\{0<1\}\hookrightarrow [p]\hookleftarrow\{1<\dots<p\}$.

It remains to show that $\cE\to [p]$ presents the named pushout, as in the statement of the lemma.
First, from the construction of $\cE$, it fits into a diagram of $\infty$-categories over $[p]$:
\[
\xymatrix{
\cE_{1}  \ar[r]  \ar[d]
&
\cE_{1p}  \ar[d]
\\
\cE_{01}  \ar[r]
&
\cE.
}
\]
We must show that this diagram of $\infty$-categories over $[p]$ is a pushout.
Let $\cZ \to [p]$ be an $\infty$-category over $[p]$.
We must show that the canonical square among spaces of functors over $[p]$,
\[
\xymatrix{
\Cat_{/[p]}(\cE,\cZ)  \ar[r]  \ar[d]
&
\Cat_{/[p]}(\cE_{1p}, \cZ)  \ar[d]
\\
\Cat_{/[p]}(\cE_{01},\cZ)  \ar[r]
&
\Cat_{/[p]}(\cE_1,\cZ)
}
\]
is a pullback.
From the definition of $\cE$ as a left Kan extension, this square is canonically identified with the square
\[
\xymatrix{
\Map^{\cP_{[p]}^{\op}}(\w{\cE},\cZ_{|\cP_{[p]}})  \ar[r]  \ar[d]
&
\Cat_{/\{1<\dots<p\}}(\cE_{1p}, \cZ_{|\{1<\dots<p\}})  \ar[d]
\\
\Cat_{/\{0<1\}}(\cE_{01},\cZ_{|\{0<1\}})  \ar[r]
&
\Cat_{/\{1\}}(\cE_1,\cZ_{|\{1\}})~.
}
\]
This last square is a pullback because the lax commutative diagram~(\ref{34}) in fact commutes -- indeed, commutativity of~(\ref{34}) precisely gives that spaces of morphisms to $\cZ_{|\cP_{[p]}}$ is such a pullback. 
\end{proof}

We now arrive at our characterization of exponentiable fibrations, useful both for identifying examples and for structural results.  
This characterization is an $\infty$-categorical version of the Conduch\'e criterion~(\cite{conduche}, \cite{giraud}). A quasi-categorical account of parts of this appears in Appendix B.3 of~\cite{HA}. 
This result and its proof is from ~\cite{emb1a}, which we include here for ease of reference.

\begin{lemma}[\cite{emb1a}]\label{exp-char}
The following conditions on a functor $\pi: \cE \ra \cK$ between $\infty$-categories are equivalent. 
\begin{enumerate}

\item The functor $\pi$ is an exponentiable fibration.

\item The base change functor $\pi^\ast \colon {\Cat}_{/\cK} \to {\Cat}_{/\cE}$ preserves colimits.  

\item\label{used} For each functor $[2]\to \cK$, the diagram of pullbacks
\[
\xymatrix{
\cE_{|\{1\}}  \ar[r]  \ar[d]
&
\cE_{|\{1<2\}}   \ar[d]
\\
\cE_{|\{0<1\}}   \ar[r]
&
\cE_{|[2]} 
}
\]
is a pushout among $\infty$-categories.

\item\label{coend.condition} For each functor $[2]\to \cK$, and for each lift $\{0\} \amalg\{2\} \xra{\{e_0\}\amalg \{e_2\}} \cE$ along $\pi$, the canonical functor from the coend
\[
\cE_{|\{0<1\}}(e_0,-)\underset{\cE_{|\{1\}}}\bigotimes \cE_{|\{1<2\}}(-,e_2)
\xra{~\circ~}  
\cE_{|[2]}(e_0,e_2)
\]
is an equivalence between spaces.

\item
For each functor $[2]\to \cK$, the canonical map between spaces
\[
\underset{[p]\in \bdelta^{\op}}\colim~\Map_{/\cK}([p]^{\tl \tr}, \cE)\xra{~\circ~} \Map_{/\cK}(\{0<2\},\cE)
\]
is an equivalence.
Here we have identified $[2]\simeq \ast^{\tl\tr}$ as the suspension of the terminal $\infty$-category, and we regard each suspension $[p]^{\tl\tr}$ as an $\infty$-category over $\ast^{\tl\tr}$ by declaring the fiber over the left/right cone point to be the left/right cone point.

\item For each functor $[2]\to \cK$, and for each lift $\{0<2\}\xra{(e_0\xra{h} e_2)}\cE$ along $\pi$, the $\infty$-category of factorizations of $h$ through $\cE_{|\{1\}}$ over $[2]\to \cK$
\[
\sB ({\cE_{|\{1\}}}^{e_0/})_{/(e_0\xra{h}e_2)}~\simeq~\ast~\simeq~\sB ({\cE_{|\{1\}}}_{/e_2})^{(e_0\xra{h}e_2)/}
\]
has contractible classifying space.  
Here, the two $\infty$-categories in the above expression agree and are alternatively expressed as the fiber of the functor
\[
\ev_{\{0<2\}}\colon \Fun_{/\cK}([2],\cE) \to \Fun_{/\cK}(\{0<2\},\cE)
\]
over $h$.

\end{enumerate}

\end{lemma}

\begin{proof}
By construction, the $\infty$-category $\Cat$ is presentable, and thereafter each over $\infty$-category ${\Cat}_{/\cC}$ is presentable.
The equivalence of~(1) and~(2) follows by way of the adjoint functor theorem (Cor. 5.5.2.9 of~\cite{HTT}), using that base-change is defined in terms of finite limits.  
The equivalence of~(4) and~(6) follows from Quillen's Theorem A.
The equivalence of~(4) and~(5) follows upon observing the map of fiber sequences among spaces
\[
\xymatrix{
\cE_{|\{0<1\}}(e_0,-)\underset{\cE_{|\{1\}}}\bigotimes \cE_{|\{1<2\}}(-,e_2)  \ar[r]  \ar[d]^-{\circ}
&
\underset{[p]\in \bdelta^{\op}}\colim~\Map_{/\ast^{\tl\tr}}([p]^{\tl \tr}, \cE_{|\ast^{\tl\tr}})  \ar[r]^-{\ev_{0,2}}  \ar[d]^-{\circ}
&
\cE_{|\{0\}}^{\sim} \times \cE_{|\{2\}}^{\sim}  \ar[d]^-{=}
\\
\cE_{|[2]}(e_0,e_2)  \ar[r]
&
\Map_{/\{0<2\}}(\{0<2\},\cE_{|\{0<2\}})  \ar[r]^-{\ev_{0,2}} 
&
\cE_{|\{0\}}^{\sim} \times \cE_{|\{2\}}^{\sim}
,
}
\]
where the top sequence is indeed a fibration sequence because colimits are universal in the $\infty$-category of spaces.  
By construction, there is the pushout expression $\{0<1\}\underset{\{1\}}\amalg \{1<2\} \xra{\simeq}[2]$ in $\Cat$; this shows~(2) implies~(3).

\medskip

We now prove the equivalence between~(3) and~(5).
Consider an $\infty$-category $\cZ$ under the diagram $\cE_{|\{0<1\}} \la \cE_{|\{1\}} \to \cE_{|\{1<2\}}$.  
We must show that there is a unique functor $\cE_{|[2]} \to \cZ$ under this diagram.  
To construct this functor, and show it is unique, it is enough to do so between the univalent Segal spaces these $\infty$-categories present:
\[
\Map([\bullet],\cE_{|[2]})~ \overset{\exists !}\dashrightarrow~ \Map([\bullet],\cZ)
\]
under $\Map([\bullet],\cE_{|\{0<1\}}) \la \Map([\bullet],\cE_{|\{1\}}) \to \Map([\bullet],\cE_{|\{1<2\}})$.

So consider a functor $[p]\xra{f} [2]$ between finite non-empty linearly ordered sets.  
Denote the linearly ordered subsets $A_i:=f^{-1}(i)\subset [p]$.  
We have the diagram of $\infty$-categories
\begin{equation}\label{two-squares}
\xymatrix{
A_1  \ar[r]  \ar[d]
&
A_1\star A_2  \ar[d]
&&
\{1\}  \ar[r]  \ar[d]
&
\{1<2\}  \ar[d]
\\
A_0\star A_1  \ar[r]
&
[p]
&
\text{ over the diagram }
&
\{0<1\} \ar[r]
&
[2].
}
\end{equation}
We obtain the solid diagram of spaces of functors
\begin{equation}\label{flat}
\Small
\xymatrix{
\Map_{/\{0<1\}}(A_0\star A_1,\cE_{|\{0<1\}})   \ar[ddd]
&
\Map_{/\{1\}}(A_1,\cE_{|\{1\}})  \ar[r]  \ar[l]  \ar@(l,l)[ddd]
&
\Map_{/\{1<2\}}(A_1\star A_2, \cE_{|\{1<2\}})  \ar[ddd]
\\
&
\Map_{/[2]}([p],\cE) \ar@{-->}[d]^-{\exists !}  \ar[ul]  \ar[ur]
&
\\
&
\Map([p],\cZ)   \ar[dl]  \ar[dr]
&
\\
\Map_{/\{0<1\}}(A_0\star A_1,\cZ) 
&
\Map_{/\{1\}}(A_1,\cZ)  \ar[r]  \ar[l]
&
\Map_{/\{1<2\}}(A_1\star A_2, \cZ)  
}
\end{equation}
and we wish to show there is a unique filler, as indicated.  
\\
{\bf Case that $f$ is consecutive:}
In this case the left square in~(\ref{two-squares}) is a pushout.
It follows that the upper and the lower flattened squares in~(\ref{flat}) are pullbacks.
And so there is indeed a unique filler making the diagram~(\ref{flat}) commute.  
\\
{\bf Case that $f$ is not consecutive:}
In this case $A_1=\emptyset$ and $A_0\neq \emptyset \neq A_2$.  
Necessarily, there are linearly ordered sets $B_0$ and $B_2$ for which $B_0^{\tr} \simeq A_0$ and $B_2^{\tl} \simeq A_2$.  
We recognize $B_0^{\tr} \underset{\{0\}}\amalg \{0<2\}\underset{\{2\}} \amalg B_2^{\tl} \xra{\simeq} [p]$ as an iterated pushout.  
So the canonical maps among spaces to the iterated pullbacks
\[
\Small
\xymatrix{
\Map_{/[2]}([p],\cE_{|[2]})  \ar[rr]^-{\simeq}
&&
\Map(B_0^{\tr},\cE_{|\{0\}})  \underset{\cE_{|\{0\}}^{\sim}}\times \Map_{/\{0<2\}}(\{0<2\},\cE_{|\{0<2\}}) \underset{\cE_{|\{2\}}^{\sim}} \times \Map(B_2^{\tl},\cE_{|\{2\}})  
}
\]
and
\[
\Small
\xymatrix{
\Map([p],\cZ)  \ar[rr]^-{\simeq}
&&
\Map(B_0^{\tr},\cZ)  \underset{\cZ^{\sim}}\times \Map(\{0<2\},\cZ) \underset{\cZ^{\sim}} \times \Map(B_2^{\tl},\cZ)   
}
\]
are equivalences.
This reduces us to the case that $[p] \to [2]$ is the functor $\{0<2\} \to [2]$.  
We have the solid diagram of spaces
\[
\xymatrix{
\Map_{/\{0<2\}}(\{0<2\},\cE_{|\{0<2\}})   \ar@{-->}[rr]^-{\exists !}
&&
\Map(\{0<2\},\cZ)  
\\
\bigl|\Map_{/[2]}([\bullet]^{\tl \tr}, \cE_{|\ast^{\tl\tr}})\bigr|  \ar[rr]  \ar[u]^-{\circ}
&&
\bigl|\Map([\bullet]^{\tl \tr}, \cZ)\bigr|  \ar[u]^-{\simeq}_-{\circ}
.
}
\]
(Here, $\bigl| - \bigr|$ denotes geometric realization of the simplicial space whose space of $[\bullet]$-points is present in the notation.)
The right vertical map is an equivalence by the Yoneda lemma for $\infty$-categories. 
(Alternatively, the domain is the classifying space of the $\infty$-category which is the unstraightening of the indicated functor from $\bdelta^{\op}$ to spaces, and the codomain maps to this $\infty$-category finally.)
Assumption~(5) precisely gives that the left vertical map is an equivalence.  
The unique filler follows.

\medskip

It remains to show~(4) implies~(1). To do this we make use of the presentation $\Cat\hookrightarrow \Psh(\bdelta)$ as univalent Segal spaces.  
Because limits and colimits are computed value-wise in $\Psh(\bdelta)$, and because colimits in the $\infty$-category $\Spaces$ are universal, then colimits in $\Psh(\bdelta)$ are universal as well.
Therefore, the base change functor
\[
\pi^\ast \colon \Psh(\bdelta)_{/\cK} \leftrightarrows \Psh(\bdelta)_{/\cE} \colon \w{\pi}_\ast
\]
has a right adjoint, $\w{\pi}_\ast$.
Because the presentation $\Cat \hookrightarrow \Psh(\bdelta)$ preserves limits, then the functor $\cE\xra{\pi}\cK$ is exponentiable provided this right adjoint $\w{\pi}_\ast$ carries univalent Segal spaces over $\cE$ to univalent Segal spaces over $\cK$.  

So let $\cA \to \cE$ be a univalent Segal space over $\cE$.  
To show the simplicial space $\w{\pi}_\ast \cA$ satisfies the Segal condition we must verify that, for each functor $[p]\to \cE$ with $p>0$, the canonical map of spaces of simplicial maps over $\cE$
\[
\Map_{/\cE}([p],\w{\pi}_\ast \cA) \longrightarrow \Map_{/\cE}(\{0<1\},\w{\pi}_\ast \cA)  \underset{\Map_{/\cE}(\{1\},\w{\pi}_\ast \cA)} \times  \Map_{/\cE}(\{1<\dots<p\},\w{\pi}_\ast \cA)
\]
is an equivalence.
Using the defining adjunction for $\w{\pi}_\ast$, this map is an equivalence if and only if the canonical map of spaces of functors
\[
\Map_{/\cK}(\pi^\ast [p],\cA) \longrightarrow \Map_{/\cK}(\pi^\ast \{0<1\},\cA)  \underset{\Map_{/\cK}(\pi^\ast \{1\},\cA)} \times  \Map_{/\cK}(\pi^\ast \{1<\dots<p\}, \cA)
\]
is an equivalence.
This is the case provided the canonical functor among pullback $\infty$-categories from the pushout $\infty$-category
\[
\cE_{|\{0<1\}} \underset{\cE_{|\{1\}}} \amalg \cE_{|\{1<\dots<p\}} \longrightarrow \cE_{|[p]}
\]
is an equivalence between $\infty$-categories over $\cK$.  
(Here we used the shift in notation $\pi^\ast \cJ := \cE_{|\cJ}$ for each functor $\cJ\to \cK$.)
This functor is clearly essentially surjective, so it remains to show this functor is fully-faithful.
Let $e_i$ and $e_j$ be objects in $\cE$, each of which lies over the object in $[p]$ indicated by the subscript.
We must show that the map between spaces of morphisms
\[
\bigl(\cE_{|\{0<1\}} \underset{\cE_{|\{1\}}} \amalg \cE_{|\{1<\dots<p\}} \bigr)(e_i,e_j) \longrightarrow \cE_{|[p]}(e_i,e_j)
\]
is an equivalence.  
Lemma~\ref{pushout} immediately grants that this is the case in the ranges $1<i\leq j\leq p$ or $0\leq i \leq j \leq 1$.  
We are reduced to the case $i=0<j$.  
Using Lemma~\ref{pushout}, this map is identified with the map from the coend
\[
\cE_{|\{0<1\}}(e_0,-)\underset{\cE_{|\{1\}}}\bigotimes \cE_{|\{1<j\}}(-,e_j)
\xra{~\circ~}  
\cE_{|\{0<1<j\}}(e_0,e_j)~.
\]
Condition~(4) exactly gives that this map is an equivalence, as desired.

It remains to verify this Segal space $\w{\pi}_\ast \cA$ satisfies the univalence condition.
So consider a univalence diagram $\cU^{\tr} \to \cK$.  
We must show that the canonical map
\[
\Map_{/\cK}(\ast,\w{\pi}_\ast \cA) \longrightarrow \Map_{/\cK}(\cU,\w{\pi}_\ast \cA)
\]
is an equivalence of spaces of maps between simplicial spaces over $\cK$.
Using the defining adjunction for $\w{\pi}_\ast$, this map is an equivalence if and only if the map of spaces
\[
\Map_{/\cE}(\cE_{|\ast},\cA)\longrightarrow \Map_{/\cK}(\cE_{|\cU},\cA)
\]
is an equivalence.
Because the presentation of $\cK$ as a simplicial space is complete, there is a canonical equivalence $\cE_{|\cU} \simeq \cE_{|\ast}\times \cU$ over $\cU$.  
That the above map is an equivalence follows because the presentation of $\cA$ as a simplicial space is complete.

\end{proof}

\begin{remark}
The equivalence of (3) and (6) was shown previously by Lurie in Proposition B.3.14 of \cite{HA}.
\end{remark}

We have immediate corollaries, using condition (\ref{used}) of Lemma \ref{exp-char}.

\begin{cor}\label{over-cell}
Each functor $\cE \to [1]
$ to the 1-cell is an exponentiable fibration.  
\end{cor}
\qed

\begin{cor}\label{cons.exp}
For each $0\leq i \leq j \leq n$, the standard inclusion $\{i<\ldots <j\}\hookrightarrow [n]$ is an exponentiable fibration.
\end{cor}
\qed

\begin{cor}\label{over-groupoid}
Each functor $\cE\to \cG$ to an $\infty$-groupoid is an exponentiable fibration.  
\end{cor}
\begin{proof}
Since $\cG$ is an $\oo$-groupoid, any functor $[2] \ra \cG$ factors through the classifying space $\sB[2]\simeq\ast$. Therefore, the base change $\cE_{|[2]}$ is canonically equivalent with a product $\cE_{|[2]}\simeq \cE_{|\{1\}}\times [2]$ over $[2]$; likewise for the further base changes over $\{0<1\}$ and $\{1<2\}$. 
Consequently, the relevant diagram from condition (\ref{used}) of Lemma \ref{exp-char} becomes
\[
\xymatrix{
\cE_{|\{1\}} \times\{1\}  \ar[r]  \ar[d]
&
\cE_{|\{1\}}  \times\{1<2\}   \ar[d]
\\
\cE_{|\{1\}} \times\{0<1\}   \ar[r]
&
\cE_{|\{1\}} \times[2] ~.
}
\]
This is a pushout since the product functor $\cE_{|\{1\}}\times(-)$ preserves colimits.
\end{proof}

\begin{cor}\label{chopped.from.lemma.pushout}
Consider $\cE\ra[p]$ an $\oo$-category over $[p]$ such that the natural functor from the pushout
\[
\cE_{|\{0<1\}}    \underset{\cE_{|\{1\}}}  \amalg  \cE_{|\{1<\dots<p\}} \longrightarrow \cE
\]
is an equivalence (as in the statement of Lemma~\ref{pushout}).  If the restriction 
$\cE_{|\{1<\dots<p\}} \to \{1<\dots<p\}$ is an exponentiable fibration, then the functor $\cE\to [p]$ is an exponentiable fibration.
\end{cor}
\qed

\subsection{Classifying correspondences}
We define a presheaf on $\Cat$ classifying exponentiable fibrations.  
Later, we will show this presheaf is representable, in a certain sense.

The following corollary of Lemma~\ref{exp-char} shows that the assignment of exponentiable fibrations defines a functor.

\begin{cor}\label{exps-pullback}
Exponentiable fibrations are stable under base change.  That is, given a pullback square among $\infty$-categories
\[
\xymatrix{
\cE'  \ar[r]  \ar[d]
&
\cE  \ar[d]
\\
\cK'  \ar[r]
&
\cK,
}
\]
in which $\cE \ra \cK$ is an exponentiable fibration, then $\cE'\ra \cK'$ is an exponentiable fibration.

\end{cor}

\begin{proof}
This follows from Lemma~\ref{exp-char}(3).  

\end{proof}

\begin{cor}\label{exp-functor}
Base change defines a functor
\[
\efib\colon \Cat^{\op} \longrightarrow \CAT~,\qquad \cK\mapsto \efib_\cK~.  
\]
\end{cor}
\qed

\begin{definition}\label{def.EFib}
The functor $\EFib^\sim$ is the composite
\[
\EFib^\sim\colon  \cat^{\op} \xra{~\efib~}\CAT\xra{~(-)^\sim~}\SPACES~,\qquad \cK\mapsto \EFib_\cK^\sim~,
\]
whose value on an $\infty$-category $\cK$ is the $\infty$-groupoid of exponentiable fibrations over $\cK$.
\end{definition}

\begin{definition}\label{ov.efib}
For an $\oo$-category $\cK$, the $\oo$-category
\[
\ov{\efib}_\cK~ \subset ~(\CAT_{/\cK})^{\cK/}
\]
is the full $\oo$-subcategory of exponentiable fibrations $\cE\to  \cK$ equipped with a section. 

\end{definition}

We have a further corollary:

\begin{cor}\label{ov.efib.functor}
Base change defines functors
\[
\ov{\efib}\colon \Cat^{\op} \longrightarrow \CAT~,\qquad \cK\mapsto \ov{\efib}_\cK~.  
\]

\end{cor}
\qed

\begin{definition}\label{def.ov.efib}
The functor $\ov{\EFib}^\sim$ is the composite
\[
\ov{\EFib}\colon \Cat^{\op} \xra{~\ov{\EFib}~} \Cat \xra{~(-)^\sim~} \SPACES~,\qquad \cK\mapsto \ov{\efib}_\cK^\sim~, 
\]
whose value on an $\infty$-category $\cK$ is the $\infty$-groupoid of exponentiable fibrations over $\cK$ equipped with a section.
\end{definition}

Note the canonical morphisms among functors from $\Cat^{\op}$: 
\[
\ov{\efib}\ra \efib
\qquad \text{ and } \qquad
\ov{\EFib}^\sim   \ra \EFib^\sim~.
\]

\begin{observation}\label{sections.fiber}
For $\cE\to \cK$ an exponentiable fibration, the canonical square among $\infty$-categories
\[
\xymatrix{
\Fun_{/\cK}(\cK,\cE)   \ar[r]\ar[d]
&
\ov{\EFib}_{\cK}\ar[d]
\\
\ast\ar[r]^{\lag \cE\ra\cK\rag }
&
\EFib_{\cK}
}
\]
is a pullback.  
\end{observation}

In the statement of the following proposition, we use the same notation for an object in $\Cat$ and its representable presheaf on $\Cat$.

\begin{prop}\label{pre.universal}
For every exponentiable fibration $\cE \ra \cK$, there is a canonical pullback diagram in $\PShv(\Cat)$
\[
\xymatrix{
\cE   \ar[d] \ar[rr]^-{\lag \cE\underset{\cK}\times \cE \rightleftarrows \cE\rag} 
&&
\ov{\EFib}^\sim   \ar[d]
\\
\cK    \ar[rr]^-{\lag \cE\to \cK\rag}
&&
\EFib^\sim
}~.
\]
\end{prop}
\begin{proof}
It suffices to show that the canonical map $\cE\to \ov{\EFib}_{|\cK}$ to the pullback presheaf on $\Cat_{/\cK}$ is an equivalence.
Let $\cJ \to \cK$ be an $\infty$-category over $\cK$.
By definition, the space of $\cJ$-points of $\ov{\EFib}_{|\cK}$ over this $\cJ$-point of $\cK$ is the space of sections $\Cat_{/\cJ}(\cJ,\cE_{|\cJ})$.
The map in question evaluates on this $\cJ$-point of $\cK$ as the map between spaces
\[
\Cat_{/\cK}(\cJ,\cE)\longrightarrow \Cat_{/\cJ}(\cJ,\cE_{|\cJ})~,
\]
which is an equivalence.

\end{proof}

\begin{cor}\label{corr.segal}
Both of the restrictions
\[
\EFib^\sim_{|\bdelta^{\op}}:\bdelta^{\op}\longrightarrow \SPACES
\qquad \text{ and } \qquad
\ov{\EFib}^\sim_{|\bDelta^{\op}}\colon \bdelta^{\op}\longrightarrow \SPACES
\]
are Segal spaces.
\end{cor}

\begin{proof}
We first establish the statement concerning $\EFib$.
Let $p>0$ be a positive integer. 
Consider the canonical square among $\infty$-categories:
\[
\xymatrix{
\EFib_{[p]}  \ar[rr]  \ar[d]
&&
\EFib_{\{1<\dots<p\}}  \ar[d]
\\
\EFib_{\{0<1\}}  \ar[rr]
&&
\EFib_{\{1\}}.
}
\]
We must show that the resulting square among spaces is a pullback.
This follows once we show that this square among $\infty$-categories is a pullback.  
Consider the canonical functor to the pullback:
\begin{equation}\label{efib.segal.map}
\EFib_{[p]} \longrightarrow \EFib_{\{0<1\}}\underset{\EFib_{\{1\}}}\times \EFib_{\{1<\dots<p\}}~.
\end{equation}
Lemma~\ref{pushout}(2) gives that this functor is a right adjoint, with left adjoint given by taking pushouts,
\[
(\cE_{01}\underset{\cE_0}\amalg \cE_{1p})\mapsfrom (\cE_{01}\mapsto \cE_0\mapsfrom \cE_{1p})
\qquad,~\EFib_{[p]} \longleftarrow \EFib_{\{0<1\}}\underset{\EFib_{\{1\}}}\times \EFib_{\{1<\dots<p\}}~.
\]

We now show that both the counit and the unit for this adjunction are equivalences.
Consider the value of the counit for this adjunction on an exponentiable fibration $\cE\to [p]$:
\[
 \cE_{|\{0<1\}}\underset{\cE_{|\{1\}}}\amalg \cE_{|\{1<\dots<p\}}       \xra{~{\rm over}~[p]~}    \cE ~.
\]
This functor over $[p]$ is an equivalence precisely because $\cE\to [p]$ is an exponentiable fibration, using the fact that the canonical functor from the colimit $\{0<1\}\underset{\{1\}}\amalg \{1<\dots<p\} \xra{\simeq} [p]$ is an equivalence between $\infty$-categories.
We now prove that the unit for this adjunction is an equivalence.
Let $(\cE_{01}\mapsto \cE_1 \mapsfrom \cE_{1p})$ be an object in the codomain of the functor~(\ref{efib.segal.map}).  
Denote the value of the left adjoint on this object as the exponentiable fibration $\cE \to [p]$.
The value of the unit on this object is the morphism
\[
(\cE_{01}\mapsto \cE_1 \mapsfrom \cE_{1p}) \longrightarrow \bigl(  \cE_{|\{0<1\}}\mapsto \cE_{|\{1\}} \mapsfrom \cE_{|\{1<\dots<p\}}\bigr)~.
\]
This morphism is an equivalence if and only if the canonical functors $\cE_{01} \to \cE_{|\{0<1\}}$ and $\cE_1 \to \cE_{|\{1\}}$ and $\cE_{1p}\to \cE_{|\{1<\dots<p\}}$ are each equivalences between $\infty$-categories.  
This is so via Lemma~\ref{pushout}(1).

We now establish the statement concerning $\ov{\EFib}$.
Consider the square among spaces:
\[
\xymatrix{
\ov{\EFib}^\sim_{[p]}    \ar[d]  \ar[rr]
&&
\ov{\EFib}^\sim_{\{0<1\}}\underset{\ov{\EFib}^\sim_{\{1\}}}\times \ov{\EFib}^\sim_{\{1<\dots<p\}}  \ar[d]
\\
\EFib^\sim_{[p]} \ar[rr]
&&
\EFib^\sim_{\{0<1\}}\underset{\EFib^\sim_{\{1\}}}\times \EFib^\sim_{\{1<\dots<p\}}.
}
\]
Let $\cE \to [p]$ be an exponentiable fibration.  
Through Observation~\ref{sections.fiber}, the map from the fiber of the left vertical map over $\cE\to \cK$ to the fiber of the right vertical map of its image is
\[
\Cat_{/[p]}([p],\cE)  \longrightarrow \Cat_{/\{0<1\}}\bigl(\{0<1\},\cE\bigr) \underset{\Cat_{/\{1\}}\bigl(\{1\},\cE\bigr)}\times
\Cat_{/\{1<\dots<p\}}\bigl(\{1<\dots<p\},\cE\bigr)~.
\]
This map is an equivalence precisely because the canonical functor from the pushout $\{0<1\}\underset{\{1\}}\amalg \{1<\dots<p\}\xra{\simeq} [p]$ is an equivalence between $\infty$-categories.  
Thus, the above square of spaces is a pullback. 
Above, we established that the bottom horizontal map is an equivalence.  
We conclude that the top horizontal map in the above square is an equivalence, as desired.

\end{proof}

The following is the main result of this section.

\begin{theorem}\label{representablecorr}
Both of the presheaves $\EFib^\sim$ and $\ov{\EFib}^\sim$ on $\Cat$ are representable by flagged $\infty$-categories; that is, both presheaves lie in the image of the restricted Yoneda functor $\fCAT \hookrightarrow \psh(\CAT)$ of Theorem~\ref{flagged.thm}.
\end{theorem}

\begin{proof}
From Theorem \ref{flagged.thm} (or by definition, see Remark \ref{doesntdepend}), the restricted Yoneda functor $\fCAT \xra{\simeq} \PShv^{\sf Segal}(\bdelta)$ is an equivalence from the $\infty$-category of flagged $\oo$-categories to that of Segal spaces.
Consequently, to establish that a presheaf $\cF$ on $\Cat$ is a flagged $\infty$-category it suffices to show these two assertions.
\begin{enumerate}
\item Its restriction $\cF_{|\bdelta^{\op}}\colon \bDelta^{\op}\to \SPACES$ is a Segal space.
\item The morphism
$
\cF \to {\sf RKan}(\cF_{|\bDelta^{\op}})
$
is an equivalence, where this morphism is the unit
of the restriction right Kan extension adjunction $\PShv(\Cat) \rightleftarrows \PShv(\bDelta)$ on $\cF$.
\end{enumerate}
Assertion~(1), as it concerns both $\EFib$ and $\ov{\EFib}$, is Corollary~\ref{corr.segal}.

Verifying assertion~(2) for $\EFib^\sim$ is to verify, for each $\infty$-category $\cK$, that the map between spaces
\begin{equation}\label{corr.from.cat}
\EFib^\sim_{\cK} \longrightarrow \limit \bigl((\bDelta_{/\cK})^{\op} \to \bDelta^{\op} \hookrightarrow \Cat^{\op} \xra{\EFib^\sim} \SPACES\bigr)
\end{equation}
is an equivalence.
Consider the canonical diagram of $\infty$-categories
\begin{equation}\label{cat.efib.sq}
\xymatrix{
\EFib_\cK \ar[rr]  \ar[d]
&&
\limit\bigl((\bDelta_{/\cK})^{\op} \to \bDelta^{\op} \hookrightarrow \Cat^{\op} \xra{\EFib} \CAT  \bigr)  \ar[d]
\\
\PShv(\bDelta)_{/\cK}  \ar[rr]
&&
\limit\Bigl((\bDelta_{/\cK})^{\op} \to \bDelta^{\op} \hookrightarrow \PShv(\bDelta)^{\op} \xra{\PShv(\bDelta)_{/-}} \PShv(\bDelta)  \Bigr),
}
\end{equation}
The vertical arrows are determined by the composite functors $\EFib_- \to \Cat_{/-} \to \PShv(\bDelta)_{/-}$, the first of which is fully-faithful by definition and the second of which is fully-faithfulness of the restricted Yoneda functor $\Cat \ra \PShv(\bdelta)$ (by \cite{joyaltierney}); so the vertical arrows are fully-faithful.
The $\infty$-category $\PShv(\bDelta)$ is an $\infty$-topos, since it is a presheaf $\oo$-category.
As a direct consequence of Theorem~6.1.0.6 of~\cite{HTT}, the bottom horizontal functor above is an equivalence between $\infty$-categories.
The left adjoint equivalence is given by taking colimits:
\begin{equation}\label{actually.equiv}
\Small
\PShv(\bDelta)_{/\cK}
\xla{~\colim~}
\Fun\bigl(\bDelta_{/\cK}, \PShv(\bDelta)_{/\cK}\bigr)
\xla{~\rm forget~}
\limit\bigl((\bDelta_{/\cK})^{\op} \to \bDelta^{\op} \hookrightarrow \Cat^{\op} \xra{\PShv(\bDelta)_{/-}} \PShv(\bDelta)  \bigr) \colon \colim~.
\end{equation}
It remains to show that the top horizontal functor is essentially surjective.
In light of the lower equivalence, we must show the following assertion:
\begin{itemize}
\item[~]
Let $\cE\to \cK$ be a map from a presheaf on $\bDelta$ to that represented by an $\infty$-category $\cK$.
Suppose, for each functor $[p]\to \cK$ from an object in $\bDelta$, that the pullback presheaf on $\bDelta$
\[
\cE_{|[p]}~:=~[p]\underset{\cK}\times \cE
\]
is represented by an $\infty$-category for which the projection $\cE_{|[p]} \to [p]$ is an exponentiable fibration.
Then $\cE$ is represented by an $\infty$-category, and the functor $\cE\to \cK$ is an exponentiable fibration.

\end{itemize}
We first show that $\cE$ is an $\infty$-category over $\cK$, then we show that the functor $\cE\to \cK$ is an exponentiable fibration.
We first show $\cE$ satisfies the Segal condition over $\cK$.
Let $p>0$ be a positive integer.
Let $[p]\to \cK$ be a functor.
We must show that the canonical square among spaces of lifts
\[
\xymatrix{
\Cat_{/\cK}\bigl([p],\cE) \ar[rr]  \ar[d]
&&
\Cat_{/\cK}\bigl(\{1<\dots<p\},\cE\bigr)  \ar[d]
\\
\Cat_{/\cK}\bigl(\{0<1\},\cE\bigr) \ar[rr]
&&
\Cat_{/\cK}\bigl(\{1\},\cE\bigr)
}
\]
is a pullback.
Through the composition-restriction adjunction $\Cat_{/[p]} \rightleftarrows \Cat_{/\cK}$, this square among spaces is identified as the square among spaces
\[
\xymatrix{
\Cat_{/[p]}\bigl([p],\cE_{|[p]}) \ar[rr]  \ar[d]
&&
\Cat_{/[p]}\bigl(\{1<\dots<p\},\cE_{|[p]}\bigr)  \ar[d]
\\
\Cat_{/[p]}\bigl(\{0<1\},\cE_{|[p]}\bigr) \ar[rr]
&&
\Cat_{/[p]}\bigl(\{1\},\cE_{|[p]}\bigr) .
}
\]
This square is a pullback precisely because, by assumption, the pullback presheaf $\cE_{|[p]} \to [p]$ is an $\infty$-category.

We now establish that $\cE$ is univalent over $\cK$.
Consider a diagram of $\infty$-categories:
\[
\xymatrix{
&
\{0<2\}  \ar[dl]  \ar[dr]
&&
\{1<3\}  \ar[dl]  \ar[dr]
&
\\
\{-\}  \ar[drr]
&&
\{0<1<2<3\}  \ar[d]
&&
\{+\}  \ar[dll]
\\
&&
\cK
&&.
}
\]
The colimit of the upper 5-term diagram over $\cK$ is a functor $\ast \xra{\lag x \rag} \cK$ selecting an object in $\cK$.
We must show that the canonical diagram involving spaces of lifts
\[
\Small
\xymatrix{
&&
\Cat_{/\cK}\bigl(\{x\},\cE\bigr)  \ar[d]  \ar[drr]  \ar[dll]
&&
\\
\Cat_{/\cK}\bigl(\{-\} , \cE\bigr)  \ar[dr]
&&
\Cat_{/\cK}\bigl(\{0<1<2<3\}, \cE\bigr)   \ar[dr]  \ar[dl]
&&
\Cat_{/\cK}\bigl(\{+\}, \cE\bigr)   \ar[dl]
\\
&
\Cat_{/\cK}\bigl(\{0<2\}, \cE\bigr)
&&
\Cat_{/\cK}\bigl(\{1<3\}, \cE\bigr)
&
}
\]
is a limit diagram.
Through the composition-restriction adjunction $\Cat_{/\ast} \rightleftarrows \Cat_{/\cK}\colon (-)_{|x}$, this diagram of spaces is identified as the diagram of spaces
\[
\Small
\xymatrix{
&&
\Cat(\ast , \cE_{|x})   \ar[d]  \ar[drr]  \ar[dll]
&&
\\
\Cat\bigl(\{-\},\cE_{|x}\bigr)  \ar[dr]
&&
\Cat\bigl(\{0<1<2<3\}, \cE_{|x}\bigr)   \ar[dr]  \ar[dl]
&&
\Cat\bigl(\{+\},\cE_{|x}\bigr) \ar[dl]
\\
&
\Cat\bigl(\{0<2\}, \cE_{|x}\bigr)
&&
\Cat\bigl(\{1<3\}, \cE_{|x}\bigr)
&.
}
\]
This square is a limit diagram precisely because, by assumption, the pullback presheaf $\cE_{|x}$ is an $\infty$-category.
We conclude that $\cE \to \cK$ is indeed an $\infty$-category over $\cK$.

We now show that this functor $\cE\to \cK$ is exponentiable.
We employ Lemma~\ref{exp-char}(3).
So let $[2]\to \cK$ be a functor.
We must show that the canonical square among $\infty$-categories
\[
\xymatrix{
\cE_{|\{1\}}  \ar[r]  \ar[d]
&
\cE_{|\{1<2\}}   \ar[d]
\\
\cE_{|\{0<1\}}   \ar[r]
&
\cE_{|[2]}
}
\]
is a pushout.
This follows because, by assumption $\cE_{|[2]} \to [2]$ is an exponentiable fibration.
This finishes the proof that the map~(\ref{corr.from.cat}) is an equivalence between spaces, as desired.

We now verify assertion~(2) for $\ov{\EFib}^\sim$.
Consider the canonical square among spaces
\begin{equation}\label{cat.ovefib.sq}
\xymatrix{
\ov{\EFib}^\sim_\cK \ar[rr]  \ar[d]
&&
\limit\bigl((\bDelta_{/\cK})^{\op} \to \bDelta^{\op} \hookrightarrow \Cat^{\op} \xra{\ov{\EFib}^\sim} \CAT  \bigr)  \ar[d]
\\
\EFib^\sim_\cK \ar[rr]
&&
\limit\bigl((\bDelta_{/\cK})^{\op} \to \bDelta^{\op} \hookrightarrow \Cat^{\op} \xra{\EFib^\sim} \CAT  \bigr) .
}
\end{equation}
We wish to show that the top horizontal map is an equivalence between spaces.
Above, we established that the bottom horizontal map is an equivalence between spaces.
Therefore, it is enough to show that this map restricts as an equivalence between fibers.
So let $(\cE\to \cK)\in \EFib^\sim_\cK$ be a point in the bottom left space of this square.
Through Observation~\ref{sections.fiber}, this map of fibers is identified as the map between spaces
\[
\Cat_{/\cK}(\cK,\cE) \longrightarrow
\limit\bigl((\bDelta_{/\cK})^{\op} \to  (\Cat_{/\cK})^{\op} \xra{\Cat_{/\cK}(-,\cE)} \CAT  \bigr)~.
\]
This map is an equivalence precisely because the canonical functor from the colimit $\colim\bigl(\bDelta_{/\cK}\to \Cat_{/\cK}\bigr) \xra{\simeq} (\cK\xra{=}\cK)$ is an equivalence in $\Cat_{/\cK}$.
We conclude that the diagram~(\ref{cat.ovefib.sq}) among spaces is a pullback, which completes this proof.

\end{proof}

\begin{definition}\label{def.Corr}
The flagged $\infty$-category $\Corr$ represents the functor $\EFib^\sim$, in the sense of Theorem~\ref{representablecorr}.
The flagged $\infty$-category $\ov{\Corr}$ represents the functor $\ov{\EFib}^\sim$, in the sense of Theorem~\ref{representablecorr}.
The \emph{universal exponentiable fibration} is the resulting canonical functor between flagged $\infty$-categories
\[
\ov{\Corr} \longrightarrow \Corr~.
\]

\end{definition}

\begin{remark}\label{yep.universal}
Proposition~\ref{pre.universal} justifies calling the canonical functor $\ov{\Corr}\to \Corr$ the \emph{universal} exponentiable fibration. 

\end{remark}

\begin{example}\label{not.univalent}
We demonstrate that $\Corr$ is \emph{not} an $\infty$-category; more precisely, that the functor $\EFib^\sim\colon \Cat^{\op}\to \SPACES$ is \emph{not} representable.  
We do this by demonstrating a colimit diagram in $\Cat$ that $\EFib^\sim$ does not carry to a limit diagram in $\Spaces$.
Specifically, consider the identification of the colimit $\ast \underset{\{0<2\}}\amalg \{0<1<2<3\}\underset{\{1<3\}}\amalg \ast \xra{\simeq} \ast$ in $\Cat$; note the differing identification of the colimit of this same diagram in $\fCat$ as  $(\{-,+\}\to \ast)$. 
There is a canonical map between spaces
\begin{equation}\label{from.E}
\Cat^\sim \simeq \EFib^\sim_\ast  \longrightarrow \EFib^\sim_\ast \underset{\EFib^\sim_{\{0<2\}}}\times \EFib^\sim_{\{0<1<2<3\}}  \underset{\EFib^\sim_{\{1<3\}}}\times \EFib^\sim_\ast~.
\end{equation}
We will demonstrate a point in the righthand space that is not in the image of this map.  

Consider the $\infty$-category ${\sf Ret}$ corepresenting a retraction, and the full $\infty$-subcategory ${\sf Idem}\subset {\sf Ret}$ corepresenting an idempotent.  
The functor ${\sf Idem}\to {\sf Ret}$ determines the pair of bimodules:
\begin{equation}\label{actually.not}
{\sf Idem}^{\op} \times {\sf Ret} \xra{{\sf Ret}(-,-)}\Spaces
\qquad  \text{ and }\qquad
{\sf Ret}^{\op} \times {\sf Idem} \xra{{\sf Ret}(-,-)} \Spaces~.
\end{equation}
Consider the two composite bimodules:
\[
{\sf Idem}^{\op} \times {\sf Idem} \xra{{\sf Ret}(-,-)\underset{{\sf Ret}}\otimes {\sf Ret}(-,-)}\Spaces
\qquad  \text{ and }\qquad
{\sf Ret}^{\op} \times {\sf Ret} \xra{{\sf Ret}(-,-)\underset{{\sf Idem}}\otimes {\sf Ret}(-,-)}\Spaces~.
\]
Because the canonical functor ${\sf Idem} \to {\sf Ret}$ witnesses an idempotent completion, the left composite bimodule is identified as the identity bimodule.
Also, because ${\sf Idem} \to {\sf Ret}$ is fully-faithful, the restriction of the right composite bimodule is canonically identified as the left composite bimodule.
Now, both ${\sf Idem} \to {\sf Ret}$ and ${\sf Idem}^{\op} \to {\sf Ret}^{\op}$ are idempotent completions.
Because $\Spaces$ is idempotent complete, the right composite bimodule is the unique extension of the left composite bimodule.
Therefore, the right composite bimodule is also the identity bimodule.  
We have demonstrated how the pair~(\ref{actually.not}) determines a point in righthand term of~(\ref{from.E}).
Since ${\sf Idem} \to {\sf Ret}$ is not an equivalence between $\infty$-categories, for it is not essentially surjective, then this point is not in the image of the map~(\ref{from.E}).

\end{example}

\begin{remark}
The defining equivalence of spaces $\Map(\cK, \corr) \simeq \efib_\cK^\sim$ does \emph{not} extend to an equivalence of $\oo$-categories between $\Fun(\cK,\corr)$ and $\efib_\cK$. 
They differ even in the case $\cK= \ast$.
Presumably, this discrepancy could be explained through the structure of $\Corr$ as a certain flagged $(\infty,2)$-category; namely, that represented by the very functor $\EFib\colon \Cat^{\op} \to \Cat$ itself.
See Question~\ref{q.2.Corr}.
\end{remark}

\subsection{Symmetric monoidal structure}\label{sym.stctr}
We endow the flagged $\infty$-category $\Corr$ with a natural symmetric monoidal structure.

Note that the full $\infty$-category $\fCAT\subset \Ar(\CAT)$ is closed under finite products.  
Consequently, the Cartesian symmetric monoidal structure on $\fCAT$ makes it a symmetric monoidal $\infty$-category. 
The $\infty$-category of \emph{symmetric monoidal flagged $\infty$-categories}
\[
\symcat ~:= ~\CAlg(\fCAT)
\]
is that of commutative algebras in the Cartesian symmetric monoidal $\infty$-category $\fCAT$.  
Because restricted Yoneda functors preserve finite products, Theorem~\ref{flagged.thm} gives a pullback diagram of $\infty$-categories:
\[
\xymatrix{
\symcat    \ar[rr]  \ar[d]
&&
\CAlg\bigl( \PShv(\bDelta)\bigr)     \ar[rr]^-{\simeq}
&& 
\Fun\bigl(\bDelta^{\op}, \CAlg(\SPACES)\bigr)  \ar[d]
\\
\fCAT  \ar[rrrr]
&&&&
\PShv(\bDelta).
}
\]
Consequently, to construct a symmetric monoidal structure on $\corr$, it suffices to give a natural lift to $\CAlg(\SPACES)$ of the space-valued functor $\EFib^\sim$ it represents. 
To do so, we observe the following.

\begin{observation}\label{fiber-products}
\begin{enumerate}
\item[~]

\item
For each $\infty$-category $\cK$, the $\infty$-category $\Cat_{/\cK}$ of $\infty$-categories over $\cK$ admits finite products, which are given by fiber products over $\cK$.

\item
For each functor $f\colon \cK\to \cK'$ between $\infty$-categories, the base change functor
$
f^\ast \colon \Cat_{/\cK'} \to \Cat_{/\cK}
$
preserves finite products.  

\item
Fiber products among $\infty$-categories over a common base defines a lift
\[
\xymatrix{
&&
\CAlg(\CAT)  \ar[d]
\\
\Cat^{\op}  \ar[rr]_-{\Cat_{/-}}  \ar@{-->}[urr]^-{\Cat_{/-}}
&&
\CAT.
}
\]

\end{enumerate}

\end{observation}

In the next result we use that the maximal $\infty$-subgroupoid functor $(-)^\sim \colon \Cat \to \Spaces$ preserves finite products.  
\begin{lemma}\label{exp-products}
\begin{enumerate}
\item[~]

\item
For each $\infty$-category $\cK$, the full $\infty$-subcategory $\EFib_\cK \subset \Cat_{/\cK}$ is closed under the formation of finite products.  

\item 
The subfunctor $\EFib\subset \Cat_{/-}$ is closed under the symmetric monoidal structure of Observation~\ref{fiber-products}.  
In particular, there is a lift
\[
\xymatrix{
&&
\CAlg(\CAT)  \ar[d]
\\
\Cat^{\op}  \ar[rr]_-{\EFib}  \ar@{-->}[urr]^-{\EFib}
&&
\CAT.
}
\]

\item 
The composition $\Cat^{\op} \xra{\EFib} \CAlg(\CAT) \xra{(-)^\sim} \CAlg(\SPACES)$ is represented by a symmetric monoidal flagged $\infty$-category.

\end{enumerate}

\end{lemma}

\begin{proof}
Point~(2) follows from point~(1); point~(3) follows from point~(1) and the existence of $\corr$ as a flagged $\oo$-category.
We now establish point~(1).
Let $I \xra{i\mapsto (\cE_i\to \cK)} \EFib_\cK$ be a functor from a finite set.
The limit of the composite functor $I\to \EFib_\cK \to \Cat_{/\cK}$ is the $\infty$-category over $\cK$ which is the $I$-fold fiber product over $\cK$:
\[
\Bigl(\prod_\cK\Bigr)_{i\in I} \cE_i
\longrightarrow
\cK~.
\]
We must show that this functor is an exponentiable fibration.  
In the case that the cardinality of $I$ is less than $2$, this is tautologically true.
So assume that the cardinality of $I$ is at least $2$.  
Let $i_0\in I$ be an element.  
The above functor, factors as a composition
\[
\Bigl(\prod_\cK\Bigr)_{i\in I} \cE_i
~\cong~
\cE_{i_0}\underset{\cK}\times \Bigl(\prod_\cK\Bigr)_{i\in I\smallsetminus \{i_0\}} \cE_i
\xra{~\pr~}
\Bigl(\prod_\cK\Bigr)_{i\in I\smallsetminus \{i_0\}} \cE_i
\longrightarrow
\cK~.
\]
By induction on the cardinality of the finite set $I$, the last of these functors is an exponentiable fibration. 
By Corollary~\ref{exps-pullback}, which states that exponentiable fibrations are closed under the formation of base change, the functor $\pr$ is an exponentiable fibration.  
We conclude from Proposition~\ref{exps.compose} that the composite is an exponentiable fibration, as desired.  
\end{proof}

\begin{cor}\label{corr-mon}
Finite products among $\infty$-categories defines a symmetric monoidal structure on the flagged $\infty$-category $\Corr$.

\end{cor}
\qed

\begin{notation}\label{def.corr-mon}
The symmetric monoidal flagged $\infty$-category of Corollary~\ref{corr-mon} is again denoted as $\Corr$; this symmetric monoidal structure will be implicitly understood.

\end{notation}

\begin{remark}\label{not-cartesian}
The monoidal structure on $\Corr$ is not Cartesian.
Namely, consider two $\infty$-categories $\cC$ and $\cD$, which we regard as objects in the flagged $\infty$-category $\Corr$.  
While projections define a diagram
\[
\cC \xla{~\pr~} \cC\times \cD  \xra{~\pr~} \cD
\]
in $\Corr$, it is generally not a limit diagram.  

\end{remark}

\subsection{Conservative exponentiable correspondences}
We explain that conservative exponentiable fibrations are classified by the full $\infty$-subcategory $\Corr[\Spaces]\subset\Corr$ consisting of the $\infty$-groupoids.

The following definitions and observations lie in parallel with the development in~\S\ref{sec.efibs}.
\begin{definition}\label{def.efib.cons}
A \emph{conservative exponentiable fibration} is an exponentiable fibration $\cE\to \cK$ that is \emph{conservative}, i.e., for which the fiber product $\cE_{|\cK^\sim}$ is an $\infty$-grouopid.  
The $\infty$-category of \emph{conservative exponentiable fibrations over $\cK$} is the full $\infty$-subcategory
\[
\EFib^{\sf cons}_{\cK}~\subset~\Cat_{/\cK}
\]
consisting of the conservative exponentiable fibrations; its $\infty$-subgroupoid is $\EFib^{\sf cons,\sim}_\cK$.

\end{definition}

\begin{example}
Both left fibrations and right fibrations are conservative exponentiable fibrations: see Lemma~\ref{exp-examples}.

\end{example}

\begin{example}
For $X$ a space, the canonical functor from the parametrized join,
\[
\ast \underset{X}\bigstar \ast:= \ast \underset{X\times \{s\}} \amalg X\times c_1\underset{X\times\{t\}}\amalg  \ast \longrightarrow c_1~,
\]
is a conservative exponentiable fibration, by Corollary~\ref{over-cell}.  This conservative exponentiable fibration is neither a left fibration nor a right fibration, so long as $X\neq\ast$ is not contractible.

\end{example}

\begin{lemma}\label{efib.cons.basechange}
Conservative exponentiable fibrations have the following closure properties.
\begin{enumerate}
\item 
For each pullback square among $\infty$-categories
\[
\xymatrix{
\cE'  \ar[r] \ar[d]
&
\cE  \ar[d]
\\
\cK' \ar[r]
&
\cK,
}
\]
the left vertical functor is a conservative exponentiable fibration whenever the right vertical functor is a conservative exponentiable fibration.  

\item 
For $\cE\to \cK$ and $\cK \to \cB$ conservative exponentiable fibrations, the composite functor $\cE\to \cB$ is a conservative exponentiable fibration.

\item
An exponentiable fibration $\cE \to c_1$ over the 1-cell is a conservative exponentiable fibration if and only if each base change $\cE_{|s} \to \ast$ and $\cE_{|t}\to \ast$ is a functor from an $\infty$-groupoid.

\item 
For $\cE \to [2]$ an exponentiable fibration for which each base change $\cE_{|\{0,1\}} \to \{0<1\}$ and $\cE_{|\{1<2\}} \to \{1<2\}$ is conservative, then the functor $\cE\to [2]$ is conservative.  

\end{enumerate}

\end{lemma}

\begin{proof}
The first two statements are immediate from the Definition~\ref{def.efib.cons}, knowing that the statements are true for exponentiable fibrations.
The third statement is immediate from the Definition~\ref{def.efib.cons}.
The fourth statement is an immediate consequence of Lemma~\ref{exp-char}(4).

\end{proof}

\begin{cor}\label{efib.cons.fctr}
Base change defines functors
\[
\EFib^{\sf cons}\colon \Cat^{\op} \longrightarrow \CAT
\qquad\text{ and }\qquad
\EFib^{\sf cons,\sim}\colon \Cat^{\op} \longrightarrow \SPACES~.
\]
Fiber products over a common base defines lifts of these functors
\[
\EFib^{\sf cons}\colon \Cat^{\op} \longrightarrow \CAlg(\CAT)
\qquad\text{ and }\qquad
\EFib^{\sf cons,\sim}\colon \Cat^{\op} \longrightarrow \CAlg(\SPACES)~.
\]
The functor $\EFib^{\sf cons,\sim}\colon \Cat^{\op} \to \CAlg(\SPACES)$ is representable, in the sense of Theorem~\ref{flagged.thm}, by a full symmetric monoidal $\infty$-subcategory of the flagged $\infty$-category $\Corr$ of Definition~\ref{def.Corr}.

\end{cor}
\qed

\begin{definition}\label{def.corr.spaces}
The symmetric monoidal $\infty$-category of \emph{correspondences of spaces} is the flagged $\infty$-subcategory 
\[
\Corr[\Spaces]~\subset ~ \Corr
\]
representing the functor $\EFib^{\sf cons,\sim}$ of Corollary~\ref{efib.cons.fctr}. 

\end{definition}

\begin{lemma}\label{indeed.not.flagged}
The monomorphism $\Corr[\Spaces]\hookrightarrow \Corr$ is fully-faithful, with image consisting of the $\infty$-groupoids.  

\end{lemma}

\begin{proof}
This follows from Lemma~\ref{efib.cons.basechange}, because an exponentiable fibration $\cE\to \cK$ is conservative if and only if, for each $\ast \to \cK$, the fiber $\cE_{|\ast}$ is an $\infty$-groupoid.

\end{proof}

\section{Cartesian and coCartesian fibrations}\label{sec.handed}
We discuss (co)Cartesian fibrations through exponentiable fibrations starting in \S\ref{cart.exp}, after reviewing the theory as due to Lurie in \S\ref{cart.review}.

\subsection{Basics about (co)Cartesian fibrations}\label{cart.review}
In this subsection, we recall definitions and some basic assertions concerning (co)Cartesian fibration of $\oo$-categories from \cite{HTT}.
We recall the straightening-unstraightening equivalence of~\cite{HTT}.

The following definition is very close to Definition~2.4.1.1 of~\cite{HTT}. 
An exact comparison between that definition and the next definition appears as Corollary~3.4 in~\cite{MG1}, where a friendly discussion of (co)Cartesian fibrations among $\infty$-categories is offered.
\begin{definition}\label{def.coCart}
Let $\pi\colon \cE\to \cK$ be a functor between $\infty$-categories.
\begin{enumerate}
\item 
\begin{enumerate}
\item 
A morphism $c_1\xra{\lag e_s\to e_t\rag} \cE$ is \emph{$\pi$-coCartesian} if the diagram of $\infty$-undercategories
\[
\xymatrix{
\cE^{e_t/}  \ar[r]  \ar[d]
&
\cE^{e_s/}  \ar[d]
\\
\cK^{\pi e_t/}  \ar[r]
&
\cK^{\pi e_s/}
}
\]
is a pullback.

\item 
The functor $\cE\xra{\pi}\cK$ is a \emph{coCartesian fibration} if each solid diagram of $\infty$-categories
\[
\xymatrix{
\ast   \ar[d]_-{\lag s \rag}  \ar[r]
&
\cE  \ar[d]^-{\pi}  
\\
c_1  \ar[r]  \ar@{-->}[ur]
&
\cK
}
\]
admits a $\pi$-coCartesian filler.  

\item 
The functor $\cE\xra{\pi}\cK$ is \emph{locally coCartesian} if, for each morphism $c_1\to \cK$, the base change $\cE_{|c_1} \to c_1$ is a coCartesian fibration.

\item 
The $\infty$-category of \emph{coCartesian fibrations (over $\cK$)} is the $\infty$-subcategory
\[
\cCart_{\cK}~\subset~\Cat_{/\cK}
\]
consisting of those objects $(\cE\xra{\pi} \cK)$ that are coCartesian fibrations, and those morphisms, which are diagrams among $\infty$-categories
\[
\xymatrix{
\cE  \ar[rr]^-F  \ar[dr]_-{\pi}
&&
\cE'  \ar[dl]^-{\pi'}
\\
&
\cK
&
}
\]
in which the downward arrows are coCartesian fibrations, for which $F$ carries $\pi$-coCartesian morphisms to $\pi$-coCartesian morphisms.  

\end{enumerate}

\item 
\begin{enumerate}
\item 
A morphism $c_1\xra{\lag e_s\to e_t\rag} \cE$ is \emph{$\pi$-Cartesian} if the diagram of $\infty$-overcategories
\[
\xymatrix{
\cE_{/e_s}  \ar[r]  \ar[d]
&
\cE_{/e_t}  \ar[d]
\\
\cK_{/\pi e_s}  \ar[r]
&
\cK_{/\pi e_s}
}
\]
is a pullback.

\item 
The functor $\cE\xra{\pi}\cK$ is a \emph{Cartesian fibration} if each solid diagram of $\infty$-categories
\[
\xymatrix{
\ast   \ar[d]_-{\lag t \rag}  \ar[r]
&
\cE  \ar[d]^-{\pi}  
\\
c_1  \ar[r]  \ar@{-->}[ur]
&
\cK
}
\]
admits a $\pi$-Cartesian filler.  

\item 
The functor $\cE\xra{\pi}\cK$ is \emph{locally Cartesian} if, for each morphism $c_1\to \cK$, the base change $\cE_{|c_1} \to c_1$ is a Cartesian fibration.

\item 
The $\infty$-category of \emph{Cartesian fibrations (over $\cK$)} is the $\infty$-subcategory
\[
\Cart_{\cK}~\subset~\Cat_{/\cK}
\]
consisting of those objects $(\cE\xra{\pi} \cK)$ that are Cartesian fibrations, and those morphisms, which are diagrams among $\infty$-categories
\[
\xymatrix{
\cE  \ar[rr]^-F  \ar[dr]_-{\pi}
&&
\cE'  \ar[dl]^-{\pi'}
\\
&
\cK
&
}
\]
in which the downward arrows are Cartesian fibrations, for which $F$ carries $\pi$-Cartesian morphisms to $\pi$-Cartesian morphisms.  

\end{enumerate}

\end{enumerate}
\end{definition}

\begin{remark}
The definition of (co)Cartesian fibration from \cite{HTT} is formulated in model-specific terms for quasi-categories and also requires that the functor $p$ be an inner fibration. This is for technical convenience, since then the pullback above can be taken to be the point-set pullback of underlying simplicial sets. Since every morphism between quasi-categories is equivalent to an inner fibration with the same codomain, we omit this condition, and instead make the convention that all pullbacks are in the $\oo$-category of $\oo$-categories (i.e., are homotopy pullbacks in a model category of $\oo$-categories). Modifying the definition in this slight way has the advantage that then being a coCartesian fibration becomes a homotopy-invariant property of a functor, and so it can be equally well formulated in any model for $\oo$-categories.
\end{remark}

\begin{example}
For $\cX$ and $\cK$ $\infty$-categories, the projection $\cK\times \cX\to \cK$ is both a coCartesian fibration as well as a Cartesian fibration.

\end{example}

\begin{example}
Let $\cE_s \xra{f} \cE_t$ be a functor between $\infty$-categories. The \emph{cylinder} is the pushout.
\[
{\sf Cyl}(f):= (\cE_s\times c_1) \underset{\cE_s\times \{t\}} \amalg \cE_t \longrightarrow c_1~.
\]
The fibers over $\{s\}$ and $\{t\}$, namely $\cE_s$ and $\cE_t$, are full $\oo$-subcategories; the mapping space between objects $e_s\in\cE_s$ and $e_t\in \cE_t$ is
\[
{\sf Cyl}(f)(e_s,e_t) \simeq \cE_t(fe_s,e_t)
\]
and there are no morphisms from $e_t$ to $e_s$. (Compare with Lemma~\ref{equiv.correspondences} for a more general expression for mapping spaces in a parametrized join.) The functor ${\sf Cyl}(f)\ra c_1$ is a coCartesian fibration; the coCartesian morphisms with respect to this projection are those sections $c_1 \to {\sf Cyl}(f)$ of the form $c_1  \simeq {\sf Cyl}(\{e_s\}\xra{=}\{e_s\}) \to {\sf Cyl}(f)$, which are determined by selecting an object $e_s\in \cE_s$.  
To see this, consider the defining diagram from Definition~\ref{def.coCart}, which becomes
\[
\xymatrix{
\cE^{fe_s/}  \ar[r]  \ar[d]
&
\cE^{e_s/}  \ar[d]
\\
\{1\}  \ar[r]
&
c_1~.}
\]
The statement that this is a pullback is equivalent to the statement that the natural functor $\cE_t^{fe_s/}\ra \cE_t^{e_s/}$ is an equivalence, which is exactly given by expression for mapping spaces in ${\sf Cyl}(f)$ above.
Likewise, the projection from the \emph{reversed cylinder} 
\[
{\sf Cylr}(f):=\cE_t \underset{\cE_s\times \{t\}} \amalg (\cE_s \times c_1^{\op}) \longrightarrow c_1^{\op} \simeq c_1
\]
is a Cartesian fibration; the Cartesian morphisms with respect to this projection are those sections $c_1 \to {\sf Cyl}(f)$ of the form $c_1  \simeq {\sf Cylr}(\{e_s\}\xra{=}\{e_s\}) \to {\sf Cylr}(f)$, which are determined by selecting an object $e_s\in \cE_s$.

\end{example}

\begin{example}\label{ar.cCart}
Let $\cK$ be an $\infty$-category.
Consider its $\infty$-category of arrows, $\Ar(\cK):=\Fun(c_1,\cK)$.
Evaluation at the target,
\[
\ev_t\colon \Ar(\cK)\longrightarrow \cK
\]
is a coCartesian fibration.
A morphism $c_1\to \Ar(\cK)$ is $\ev_t$-coCartesian if and only if its adjoint $c_1\times c_1 \to \cK$ factors through the epimorphism $c_1\times c_1 \to (c_1\times c_1)\underset{\{t\}\times c_1 } \amalg \ast \simeq  [2]$.
Alternatively, a morphism $c_1\to \Ar(\cK)$ is $\ev_t$-coCartesian if and only if the composite functor $c_1\to \Ar(\cK)\xra{\ev_s} \cK$ selects an equivalence in $\cK$.  
Evaluation at the source,
\[
\ev_t\colon \Ar(\cK)\longrightarrow \cK
\]
is a Cartesian fibration.
A morphism $c_1\to \Ar(\cK)$ is $\ev_s$-Cartesian if and only if its adjoint $c_1\times c_1 \to \cK$ factors through the epimorphism $c_1\times c_1 \to (c_1\times c_1)\underset{\{s\}\times c_1} \amalg \ast \simeq [2]$.  
Alternatively, a morphism $c_1\to \Ar(\cK)$ is $\ev_s$-Cartesian if and only if the composite functor $c_1\to \Ar(\cK)\xra{\ev_t} \cK$ selects an equivalence in $\cK$.  
\end{example}

\begin{observation}\label{cCart.ops}
A functor $\cE\to \cK$ is a coCartesian fibration if and only if its opposite $\cE^{\op} \to \cK^{\op}$ is a Cartesian fibration.  

\end{observation}

\begin{lemma}\label{cCart.compose}
Let $\cE \xra{\pi} \cK \xra{\pi'} \cU$ be a composable sequence of functors between $\infty$-categories.
\begin{enumerate}
\item If $\pi$ and $\pi'$ are coCartesian, then the composition $\pi'\circ\pi$ is coCartesian.

\item If $\pi$ and $\pi'$ are Cartesian, then the composition $\pi'\circ\pi$ is Cartesian.

\end{enumerate}

\end{lemma}

\begin{proof}
The two assertions imply one another, as implemented by taking opposites.  
We are therefore reduced to proving assertion~(1).
That is, we show that every morphism $c_1 \xra{u_s\ra u_t} \cU$ with specified lift $e_s \in \cE_{|u_s}$ can be lifted to a $(\pi'\circ\pi)$-coCartesian morphism in $\cE$.
Using, in sequence, that $\pi'$ and $\pi$ are coCartesian fibrations, we can first lift $u_s\ra u_t$ to a $\pi'$-coCartesian morphism $\pi e_s \ra k_t$ for some $k_t$, and then lift the morphism $\pi e_s\ra k_t$ to a $\pi$-coCartesian morphism $e_s \ra e_t$ for some $e_t$.
This is represented in the following diagram:
\[
\xymatrix{
\ast\ar[dd]_{\langle s\rangle}\ar[rrr]^{\langle e_s\rangle}&&&\cE\ar[d]^{\pi}\\
&&&\cK\ar[d]^{\pi'}\\
c_1\ar@{-->}[urrr]_{\langle\pi e_s \ra k_t\rangle}\ar@{-->}[uurrr]^{\langle e_s \ra e_t\rangle}\ar[rrr]_{\langle u_s\ra u_t\rangle}&&&\cU}
\]
It remains to show that the lift $c_1\xra{e_s\ra e_t}\cE$ is a $(\pi'\circ\pi)$-coCartesian morphism. Consider the pair of commutative squares
\[
\xymatrix{
\cE^{e_t/}\ar[r]\ar[d]&\cE^{e_s/}\ar[d]\\
\cK^{k_t/}\ar[r]\ar[d]&\cK^{\pi e_s/}\ar[d]\\
\cU^{u_t/}\ar[r]&\cU^{u_s/}
}
\]
where the top square is a pullback since $e_s\ra e_t$ is a $\pi$-coCartesian morphism and the bottom square is a pullback since $\pi e_s \ra k_t$ is a $\pi'$-coCartesian morphism. Consequently, the outer rectangle is a pullback diagram, which is exactly the condition of $e_s \ra e_t$ being a $(\pi'\circ\pi)$-coCartesian morphism.

\end{proof}

(Co)Cartesian fibrations are closed under base change, as the next result shows.
\begin{lemma}\label{Cart.base.change}
Let 
\[
\xymatrix{
\cE'  \ar[d]_-{\pi'}  \ar[r]^-{\w{F}}
&
\cE  \ar[d]^-{\pi}  
\\
\cK'  \ar[r]^-{F}
&
\cK
}
\]
be a pullback diagram of $\infty$-categories.
\begin{enumerate}
\item  
If $\pi$ is a coCartesian fibration, then $\pi'$ is a coCartesian fibration.

\item  
If $\pi$ is a Cartesian fibration, then $\pi'$ is a Cartesian fibration.

\end{enumerate}

\end{lemma}

\begin{proof}
Assertion~(1) and assertion~(2) imply one another by taking opposites.
We are therefore reduced to proving assertion~(1).  

Suppose $\pi$ is a coCartesian fibration.
Consider a solid diagram of $\infty$-categories:
\[
\xymatrix{
\ast   \ar[d]_-{\lag s \rag}  \ar[rr]^-{\lag e_s'\rag}
&&
\cE'  \ar[d]    \ar[rr]
&&
\cE  \ar[d]
\\
c_1  \ar[rr]_-{\lag x' \xra{f'} y'\rag}   \ar@{-->}[urrrr]    \ar@{-->}[urr]
&&
\cK'  \ar[rr]
&&
\cK.
}
\]
Choose a $\pi$-coCartesian morphism as in the lower lift.
Denote the target of this lower filler as $e_t$.
Because the given square among $\infty$-categories is a pullback, this filler is equivalent to a higher filler, as indicated.
We must show that this higher lift is a $\pi$-coCartesian morphism.
Denote the target of this higher filler as $e'_t$.
Consider the canonical diagram of $\infty$-categories:
\[
\xymatrix{
{\cE'}^{e'_t/}  \ar[dr]  \ar[rrr]  \ar[ddd]
&
&
&
{\cE'}^{e'_s/}  \ar[dl]  \ar[ddd]
\\
&
\cE^{e_t/}  \ar[r]  \ar[d]
&
\cE^{e_s/}  \ar[d]
&
\\
&
\cK^{Fx'/}  \ar[r]
&
\cK^{Fy'/}  
&
\\
{\cK'}^{x'/}  \ar[ur]   \ar[rrr]
&
&
&
{\cK'}^{y'/}.  \ar[ul]
}
\]
By definition of a $\pi$-coCartesian morphism, the inner square is a pullback.
Because the given square is a pullback, then so too are the left and right squares in the above diagram.  
It follows that the outer square is a pullback.

\end{proof}

The next auxiliary result states the equivalences between (co)Cartesian fibrations are detected on fibers.  
\begin{lemma}\label{l.fib.fibers}
Consider a commutative diagram 
\[
\xymatrix{
\cE  \ar[rr]^-F  \ar[dr]_-{\pi}
&&
\cE'  \ar[dl]^-{\pi'}
\\
&
\cK
&
}
\]
among $\infty$-categories.
The functor $F\colon \cE\to \cE'$ is an equivalence between $\infty$-categories provided either of the following conditions.  
\begin{itemize}
\item
Both $\pi$ and $\pi'$ are coCartesian fibrations, and $F$ carries $\pi$-coCartesian morphisms to $\pi'$-coCartesian morphisms.

\item
Both $\pi$ and $\pi'$ are Cartesian fibrations, and $F$ carries $\pi$-Cartesian morphisms to $\pi'$-Cartesian morphisms.

\end{itemize}
\end{lemma}

\begin{proof}
The assertion concerning coCartesian fibrations implies that for Cartesian fibrations, as implemented by taking opposites.
We are therefore reduced to proving the assertion concerning coCartesian fibrations.

The condition that the functor between each fiber is an equivalence guarantees, in particular, that $F$ is surjective.
It remains to show that $f$ is fully-faithful.
Let $a,b\in \cE$.  
We intend to show that the top horizontal map in the diagram of spaces of morphisms,
\[
\xymatrix{
\cE(a,b) \ar[rr]^-F  \ar[dr]_-{\pi}
&&
\cE'(Fa,Fb)  \ar[dl]^-{\pi'}
\\
&
\cK(\pi a, \pi b)
&
,
}
\]
is an equivalence.
For this it is enough to show that, for each morphism $\pi a \xra{f}\pi b$ in $\cK$, the map between fibers
\[
\cE(a,b)_{|f}
\longrightarrow
\cE'(Fa,Fb)_{|f}
\]
is an equivalence between spaces.
Using the assumption that both $\pi$ and $\pi'$ are coCartesian fibrations and that $F$ preserves coCartesian morphisms, we identify this map between fibers as the map
\[
\cE_{|\pi b}(f_! a , b) \longrightarrow \cE_{|\pi' Fb}(f_! Fa , F b)
\]
between spaces of morphisms in the fibers of $\pi$ and $\pi'$ over $\pi b \simeq \pi' Fb \in \cK$; here, $(a \to f_! a)$ and $(Fa \to f_! Fa)$ are respective coCartesian morphisms in $\cE$ and $\cE'$.
The assumption that $F$ restricts as an equivalence between $\infty$-categories of fibers for $\pi$ and $\pi'$ implies this map is an equivalence.
This concludes this proof.

\end{proof}

Lemma~\ref{Cart.base.change} has this immediate result.
In the statement of this result we reference the Cartesian symmetric monoidal structures on $\CAT$ and on $\SPACES$.
\begin{cor}\label{cCart.functor}
Base change defines functors
\[
\cCart\colon \Cat^{\op} \longrightarrow \CAT~,\qquad \cK\mapsto \cCart_{\cK}~,
\qquad\text{ and }\qquad
\Cart\colon \Cat^{\op} \longrightarrow \CAT~,\qquad \cK\mapsto \Cart_{\cK}~,
\]
as well as 
\[
\cCart^\sim \colon \Cat^{\op} \xra{~\cCart~} \CAT \xra{~(-)^\sim~}\SPACES
\qquad\text{ and }\qquad
\Cart^\sim \colon \Cat^{\op} \xra{~\Cart~} \CAT \xra{~(-)^\sim~}\SPACES~.
\]
Fiber products over a common base defines lifts of these functors
\[
\cCart\colon \Cat^{\op} \longrightarrow \CAlg(\CAT)
\qquad\text{ and }\qquad
\Cart \colon \Cat^{\op} \longrightarrow \CAlg(\CAT)~,
\]
as well as 
\[
\cCart^\sim \colon \Cat^{\op} \longrightarrow  \CAlg(\SPACES)
\qquad\text{ and }\qquad
\Cart^\sim \colon \Cat^{\op} \longrightarrow \CAlg(\SPACES)~.
\]

\end{cor}
\qed

The following construction of~\cite{HTT} is an $\infty$-categorical version of the Grothendieck construction.
We give the description of this construction from Theorem 1.1, after Definition 2.8, of~\cite{gepner-haugseng-nikolaus}.

\begin{construction}\label{def.un}
Let $\cK$ be an $\infty$-category.
The \emph{unstraightening} construction (for coCartesian fibrations) is the functor
\[
{\sf Un}\colon \Fun(\cK,\Cat)\longrightarrow \Cat_{/\cK}~,\qquad
(\cK\xra{F}\Cat)\mapsto   \bigl(\cK^{\bullet/} \underset{\cK} \otimes F \to \cK\bigr)~,
\]
whose values are given by coends, with respect to the standard tensor structure $\otimes \colon \Cat_{/\cK}\times \Cat \xra{\times} \Cat_{/\cK}$.  
The \emph{unstraightening} construction (for Cartesian fibrations) is the functor
\[
{\sf Un}\colon \Fun(\cK^{\op},\Cat)\longrightarrow \Cat_{/\cK}~,\qquad
(\cK^{\op}\xra{G}\Cat)\mapsto   \bigl(G \underset{\cK} \otimes \cK_{/\bullet} \to \cK\bigr)~,
\]
whose values are given by coends, with respect to the standard tensor structure $\otimes \colon \Cat\times \Cat_{/\cK}\xra{\times} \Cat_{/\cK}$.  

\end{construction}

\begin{example}
For $c_1\xra{\lag \cE_s\xra{f}\cE_t\rag} \Cat$ a functor, its unstraightening (as a coCartesian fibration) is the cylinder construction: ${\sf Cyl}(f) \to c_1$.
For $c_1 \xra{\lag \cE_t \xla{f} \cE_s\rag} \Cat^{\op}$ a functor, its unstraightening (as a Cartesian fibration) is the reverse cylinder construction: ${\sf Cylr}(f) \to c_1$.  

\end{example}

\begin{observation}\label{right.un}
For each $\infty$-category $\cK$, the unstraightening constructions
\[
\Fun(\cK,\Cat)\xra{~\sf Un~} \Cat_{/\cK}
\qquad \text{ and }\qquad
\Fun(\cK^{\op},\Cat)\xra{~\sf Un~} \Cat_{/\cK}
\]
are each left adjoints; their respective right adjoints are given by taking ends:
\[
\Cat_{/\cK} \longrightarrow  \Fun(\cK,\Cat)
~,\qquad
(\cE\to \cK)\mapsto \Fun_{/\cK}(\cK^{\bullet/},\cE)
\]
and
\[
\Cat_{/\cK} \longrightarrow  \Fun(\cK^{\op},\Cat)
~,\qquad
(\cE\to \cK)\mapsto \Fun_{/\cK}(\cK_{/\bullet},\cE)~.
\]

\end{observation}

The following principal result of Lurie explains how the unstraightening construction implements representatives of the functors of Corollary~\ref{cCart.functor}.
\begin{theorem}[\cite{HTT}]\label{st.un.cCart}
The functor $\cCart^\sim \colon \Cat^{\op} \to \CAlg(\SPACES)$ is represented by the Cartesian symmetric monoidal $\infty$-category $\Cat$; specifically, for each $\infty$-category $\cK$, the unstraightening construction implements a canonical equivalence between $\infty$-groupoids
\[
{\sf Un}\colon \CAT(\cK,\Cat)~\simeq~ \cCart_{\cK}^\sim~.
\]
The functor $\Cart^\sim \colon \Cat^{\op} \to \CAlg(\SPACES)$ is represented by the coCartesian symmetric monoidal $\infty$-category $\Cat^{\op}$; specifically, for each $\infty$-category $\cK$, the unstraightening construction implements a canonical equivalence between $\infty$-groupoids
\[
{\sf Un}\colon \CAT(\cK,\Cat^{\op})~\simeq~ \Cart_{\cK}^\sim~.
\]

\end{theorem}

\subsection{Characterizing (co)Cartesian fibrations}\label{cart.exp}
We establish a useful characterization for (co)Cartesian fibrations, in the context of exponentiable fibrations.
We do this as two steps; we first characterize locally (co)Cartesian fibrations, we then characterize (co)Cartesian fibrations in terms of locally (co)Cartesian fibrations.

We observe that (co)Cartesian fibrations are examples of exponentiable fibrations.  
Later, in Theorem~\ref{theorem-cocart}, we characterize which exponentiable fibrations are (co)Cartesian fibrations.  
\begin{lemma}\label{exp-examples}
Cartesian fibrations and coCartesian fibrations are exponentiable fibrations.

\end{lemma}

\begin{proof}

Using Observation~\ref{exp-op} and Observation~\ref{cCart.ops} the coCartesian case implies the Cartesian case.
So let $\pi\colon \cE \to \cK$ be a coCartesian fibration.  
We invoke the criterion of Lemma~\ref{exp-char}(6). 
So fix a functor $[2]\to \cK$.
Extend this functor as a solid diagram of $\infty$-categories:
\begin{equation}\label{22}
\xymatrix{
\{0<2\} \ar[rr]^-{\lag e_0\to e_2\rag}     \ar[d]
&&
\cE  \ar[d]^-{\pi}
\\
[2]  \ar[rr]  \ar@{-->}[urr]
&&
\cK.
}
\end{equation}
An object in the $\infty$-category $({\cE_{|1}}^{e_0/})_{/(e_0\to e_2)}$ is an indicated filler in this diagram.  
Choose a $\pi$-coCartesian lift 
\begin{equation}\label{33}
\xymatrix{
\{0\} \ar[rr]^-{\lag e_0\rag}     \ar[d]
&&
\cE  \ar[d]^-{\pi}
\\
\{0<1\}  \ar[rr]  \ar@{-->}[urr]^-{\lag e_0\to e_1\rag}
&&
\cK.
}
\end{equation}
By definition of a $\pi$-coCartesian morphism, there is a unique filler of the diagram~(\ref{22}) extending the diagram~(\ref{33}), thereby determining an object in the $\infty$-category $({\cE_{|1}}^{e_0/})_{/(e_0\to e_2)}$.
Precisely because the lift in~(\ref{33}) is a $\pi$-coCartesian morphism, this object in $({\cE_{|1}}^{e_0/})_{/(e_0\to e_2)}$ is initial.
We conclude that its classifying space $\sB({\cE_{|1}}^{e_0/})_{/(e_0\to e_2)}\simeq \ast$ is terminal, as desired.

\end{proof}

Here is a simpler criterion for assessing if a functor is a (co)Cartesian fibration.
\begin{lemma}\label{lemma.reformulation}
Let $\pi\colon \cE\to \cK$ be a functor between $\infty$-categories.
\begin{enumerate}
\item
\begin{enumerate}
\item
A morphism $c_1\xra{\lag e_s \to e_t\rag} \cE$ is $\pi$-coCartesian if and only if it is initial as an object in the fiber product $\infty$-category $\cE^{e_s/} \underset{\cK^{x/}} \times \cK^{y/}$.

\item 
The functor $\pi$ is a coCartesian fibration if and only if each solid diagram of $\infty$-categories
\[
\xymatrix{
\ast   \ar[d]_-{\lag s \rag}  \ar[rrr]^-{\lag e_s\rag}
&&&
\cE  \ar[d]^-{\pi}  
\\
c_1  \ar[rrr]_-{\lag x \xra{f} y\rag}   \ar@{-->}[urrr]^-{\lag e_s \to f_!e_s\rag}
&&&
\cK,
}
\]
admits a filler that is initial in the fiber product $\infty$-category $\cE^{e_s/} \underset{\cK^{x/}} \times \cK^{y/}$.

\end{enumerate}

\item
\begin{enumerate}
\item
A morphism $c_1\xra{\lag e_s \to e_t\rag} \cE$ is $\pi$-Cartesian if and only if it is final as an object in the fiber product $\infty$-category $\cE_{/e_t} \underset{\cK_{/y}} \times \cK_{/x}$.

\item 
The functor $\pi$ is a Cartesian fibration if and only if each solid diagram of $\infty$-categories
\[
\xymatrix{
\ast   \ar[d]_-{\lag t \rag}  \ar[rrr]^-{\lag e_t\rag}
&&&
\cE  \ar[d]^-{\pi}  
\\
c_1  \ar[rrr]_-{\lag x \xra{f} y\rag}   \ar@{-->}[urrr]^-{\lag f^\ast e_t \to e_t\rag}
&&&
\cK,
}
\]
admits a filler that is initial in the fiber product $\infty$-category  $\cE_{/e_t} \underset{\cK_{/y}} \times \cK_{/x}$.
\end{enumerate}

\end{enumerate}

\end{lemma}

\begin{proof}
Assertion~(1) and assertion~(2) imply one another, as implemented by replacing $(\cE \xra{\pi}\cK)$ by its opposite, $(\cE^{\op}\xra{\pi^{\op}}\cK^{\op})$.
We are therefore reduced to proving assertion~(1).

Inspecting the Definition~\ref{def.coCart} of a coCartesian fibration, assertion~(a) implies assertion~(b).  
We are therefore reduced to proving assertion~(a).
Let $c_1\xra{\lag e_s\to e_t\rag}\cE$ be a morphism.
We show that the condition in assertion~(a) is equivalent to the condition that this morphism is $\pi$-coCartesian.
The given morphism $(e_s\to e_t)$ determines the diagram, $\gamma$,
\[
\xymatrix{
\ast  \ar[d]_-{\lag s\rag}  \ar[rr]^-{\lag e_s \rag}
&&
\cE  \ar[d]^-{\pi}  
\\
c_1  \ar[rr]_-{\lag \pi e_s \to \pi e_t \rag}     \ar[urr]^-{\lag e_s\to e_t\rag}
&&
\cK,
}
\]
which we regard as an object in the fiber product $\infty$-category $\cE^{e_s/}\underset{\cK^{x/}}\times\cK^{y/}$.
Observe the canonical identification between $\infty$-undercategories:
\[
\cE^{e_t/}\xra{~\simeq~} \bigl(\cE^{e_s/}\underset{\cK^{x/}}\times\cK^{y/}\bigr)^{\gamma/}~.
\]
Through this identification we see that $\gamma$ is an initial object in this fiber product $\infty$-category if and only if the canonical functor
\[
\cE^{e_t/}\longrightarrow  \cE^{e_s/}\underset{\cK^{x/}}\times\cK^{y/}
\]
is an equivalence between $\infty$-categories.  
After inspecting Definition~\ref{def.coCart} of a $\pi$-coCartesian morphism, this establishes assertion~(a).

\end{proof}

We will make repeated, and implicit, use of the following characterization of left/right adjoints.
The following proof, which simpler than our original proof, was suggested by a referee.
\begin{lemma}\label{adjointcriterion}
Let $\cC\xra{F} \cD$ be a functor between $\infty$-categories.
\begin{enumerate}

\item
The following conditions on the functor $F$ are equivalent.
\begin{enumerate}
\item $F$ is a right adjoint.

\item For each object $d\in \cD$, the $\infty$-undercategory $\cC^{d/}$ has an initial object.

\end{enumerate}
Furthermore, if $F$ is a right adjoint, the value of its left adjoint on $d\in \cD$ is the value of the forgetful functor $\cC^{d/}\to \cC$ on its initial object.  

\item
The following conditions on the functor $F$ are equivalent.
\begin{enumerate}
\item $F$ is a left adjoint.

\item For each object $d\in \cD$, the $\infty$-overcategory $\cC_{/d}$ has a final object.

\end{enumerate}
Furthermore, if $F$ is a left adjoint, the value of its right adjoint on $d\in \cD$ is the value of the forgetful functor $\cC_{/d} \to \cC$ on its final object.

\end{enumerate}

\end{lemma}

\begin{proof}
The two assertions are equivalent by taking opposites.
We therefore reduce to proving~(1).

Let $d\in \cD$ be an object, and let $(c,d\to Fc)\in \cC^{d/}$ be an object in the $\infty$-undercategory.
For each object $(c',d\to Fc')\in \cC^{d/}$, consider the canonical diagram of mapping spaces in the $\oo$-categories $\cC^{d/}$, $\cC$, $\cD^{d/}$, and $\cD$:
\[
\xymatrix{
\cC^{d/} \bigl( (c,d\to Fc) , (c',d\to Fc') \bigr) )
\ar[rr] \ar[d]
&&
\cC(c,c') \ar[d]
\\
\cD^{d/} \bigl( (d\to Fc) , d\to Fc') \bigr) 
\ar[rr] \ar[d]
&&
\cD(Fc,Fc') \ar[d]
\\
\ast
\ar[rr]^-{\lag d\xra{=} d\rag }
&&
\cD(d,Fc')
.
}
\]
Both of the two inside squares are pullbacks, by definition of the $\infty$-overcategories $\cD^{d/}$ and $\cC^{d/} := \cC\underset{\cD} \times \cD^{d/}$.
Therefore the outer square is a pullback.
Consequently, the top left term in this diagram is terminal if and only if the composite right vertical map is an equivalence:
\[
\cC^{d/}\bigl( (c,d\to Fc) , (c', d\to Fc') \bigr)
\xra{~\simeq~}
\ast
\qquad 
\iff
\qquad
\cC(c,c') \xra{~\simeq~} \cD(d,Fc')
~.
\]
In this logical equivalence, the right canonical map between spaces is the value of the natural transformation:
\[
\xymatrix{
\cC \ar[dr]_-F \ar@(u,u)[rr]^-{\cC(c,-)}
&&
\Spaces
\\
&
\cD \ar[ur]_-{\cD(d,F-)}
&
.
}
\]
In this way, we see that the object $(c,d\to Fc)\in \cC^{d/}$ is initial if and only if the object $c\in \cC$ corepresents the functor
\[
\cC
\xra{~F~}
\cD
\xra{~\cD(d,-)~}
\Spaces
~.
\]
We conclude that $F$ is a right adjoint if and only if the $\infty$-undercategory $\cC^{d/}$ has an initial object for each object $d\in \cD$.
Furthermore, the value of the left adjoint to $F$ on $d\in \cD$ is the value of the forgetful functor $\cC^{d/}\to \cC$ on its initial object.  \end{proof}

\begin{lemma}\label{firsthalf-locallycocart}
Let $\cE \xra{\pi} c_1$ be an $\infty$-category over the 1-cell.
\begin{enumerate}
\item 
The following two conditions on this functor $\pi$ are equivalent.
\begin{enumerate}
\item It is a coCartesian fibration.
\item The canonical functor from the fiber $\cE_{|t} \hookrightarrow \cE$ is a right adjoint.

\end{enumerate}

\item 
The following two conditions on this functor $\pi$ are equivalent.
\begin{enumerate}
\item It is a Cartesian fibration.
\item The canonical functor from the fiber $\cE_{|s} \hookrightarrow \cE$ is a left adjoint.

\end{enumerate}

\end{enumerate}
\end{lemma}

\begin{proof}
Assertion~(1) and assertion~(2) imply one another, as implemented by replacing $(\cE \xra{\pi}c_1)$ by its opposite, $(\cE^{\op}\xra{\pi^{\op}}c_1^{\op}\simeq c_1)$.
We are therefore reduced to proving assertion~(1).  

The canonical identification $(c_1^{t/} \to c_1^{s/})\simeq (\ast \xra{\lag t \rag} c_1)$ determines the first of these identifications among $\infty$-categories 
\[
\cE^{e_s/} \underset{c_1^{s/}} \times c_1^{t/}
~\simeq~ 
\cE^{e_s/}\underset{c_1}\times \{t\} 
~\simeq~ 
(\cE_{|t})^{e_s/} 
.
\]

A consequence of Lemma~\ref{lemma.reformulation} is that $\pi$ is a coCartesian fibration if and only if, for each object $e_s\in \cE_{|s}$ over $s\in c_1$, the fiber product $\infty$-category $\cE^{e_s/} \underset{c_1^{s/}} \times c_1^{t/}$ has an initial object.
The equivalence between~(a) and~(b) then follows from the above identifications, using the criterion of Lemma~\ref{adjointcriterion}.

\end{proof}

\begin{lemma}\label{secondhalf-locallycocart}
Let $\cE\xra{\pi} \cK$ be a functor between $\infty$-categories.
\begin{enumerate}
\item 
Let $y\in \cK$ be an object.  
The following conditions on this data are equivalent.  
\begin{enumerate}
\item The canonical functor $\cE_{|y}\hookrightarrow\cE_{/y}$ is a right adjoint.
\item For each morphism $c_1\xra{\lag x \to y\rag} \cK$, the canonical functor $\cE_{|y} \hookrightarrow \cE_{|c_1}$ is a right adjoint.

\end{enumerate}

\item 
Let $x\in \cK$ be an object.  
The following conditions on this data are equivalent.  
\begin{enumerate}
\item The canonical functor $\cE_{|x}\hookrightarrow\cE^{x/}$ is a left adjoint.
\item For each morphism $c_1\xra{\lag x \to y\rag} \cK$, the canonical functor $\cE_{|x} \hookrightarrow \cE_{|c_1}$ is a left adjoint.

\end{enumerate}

\end{enumerate}

\end{lemma}

\begin{proof}
Assertion~(1) and assertion~(2) imply one another, as implemented by replacing $(\cE \xra{\pi}\cK)$ by its opposite, $(\cE^{\op}\xra{\pi^{\op}}\cK^{\op})$.
We are therefore reduced to proving assertion~(1).

We use the criterion of Lemma~\ref{adjointcriterion}.
Let $\gamma\in \cE_{/y}$ be an object, which is the datum of a diagram of $\infty$-categories:
\[
\xymatrix{
\ast   \ar[d]_-{\lag s \rag}  \ar[rr]^-{\lag e_s\rag}
&&
\cE  \ar[d]^-{\pi}  
\\
c_1  \ar[rr]_-{\lag x \to y\rag}    
&&
\cK.
}
\]
Consider the canonical diagram of $\infty$-categories:
\[
\xymatrix{
(\cE_{|y})^{e_s/}:=\cE_{|y}\underset{\cE_{|c_1}} \times (\cE_{|c_1})^{e_s/} \ar[rr]   \ar[dr]
&&
\cE_{|y}\underset{\cE_{/y}} \times (\cE_{/y})^{\gamma/} =: (\cE_{|y})^{\gamma/}  \ar[dl]
\\
&
\cE_{|y}
&
.
}
\]
The downward functors in this diagram are left fibrations.
For each $e_t\in \cE_{|y}$, the resulting map between fiber spaces is identifiable as the identity map between spaces of morphisms
\[
\cE_{|c_1}(e_s,e_t) \xra{~=~} \cE_{|c_1}(e_s,e_t)~.
\]
We conclude that the top horizontal functor in the above diagram is an equivalence between $\infty$-categories.
The equivalence between conditions~(a) and~(b) follows immediately.

\end{proof}

\begin{lemma}\label{locallycocart}
Let $\pi\colon \cE\to \cK$ be a functor between $\infty$-categories.  
\begin{enumerate}
\item
The following conditions on a functor $\pi$ are equivalent.
\begin{enumerate}
\item It is locally coCartesian.
\item For every object $y\in \cK$, the canonical functor from the fiber to the $\infty$-overcategory, $\cE_{|y}\hookrightarrow \cE_{/y}$, is a right adjoint.
\item 
For each morphism $c_1\to \cK$, the restriction functor between $\infty$-categories of sections
\[
\ev_s\colon \Fun_{/\cK}(c_1,\cE) \longrightarrow \cE_{|s}
\]
admits a fully-faithful left adjoint.  
\end{enumerate}

\item
The following conditions on a functor $\pi$ are equivalent.
\begin{enumerate}
\item It is locally Cartesian.
\item 
For every object $x\in \cK$, the canonical functor from the fiber to the $\infty$-undercategory, $\cE_{|x}\hookrightarrow \cE^{x/}$, is a left adjoint.
\item 
For each morphism $c_1\to \cK$, the restriction functor between $\infty$-categories of sections
\[
\ev_t\colon \Fun_{/\cK}(c_1,\cE) \longrightarrow \cE_{|t}
\]
admits a fully-faithful right adjoint.  
\end{enumerate}

\end{enumerate}

\end{lemma}

\begin{proof}
Assertion~(1) and assertion~(2) imply one another, as implemented by taking opposites.
We are therefore reduced to proving assertion~(1).

The equivalence between condition~(a) and condition~(b) is a direct logical concatenation of Lemmas~\ref{firsthalf-locallycocart} and~\ref{secondhalf-locallycocart}.

We now establish that condition~(b) implies condition~(c).
Let $c_1\to \cK$ be a functor from the 1-cell.
We must show that, for each object $e_s\in \cE_{|s}$, the $\infty$-undercategory $\Fun_{/\cK}(c_1,\cE)^{e_s/}$ has an initial object.  
The restriction functor $\ev_s$ is a Cartesian fibration.
The established implication (a)$\implies$(b), as it concerns (locally) Cartesian fibrations, gives that the canonical functor from the fiber $\infty$-category
\begin{equation}\label{e120}
(\cE_{|t})^{e_s/}:=(\cE_{|c_1})^{e_s/}\underset{\cE_{|c_1}}\times \cE_{|t}\simeq \Fun_{/\cK}(c_1,\cE)_{|e_s} \longrightarrow \Fun_{/\cK}(c_1,\cE)^{e_s/}
\end{equation}
is a left adjoint.  
Because left adjoints compose, we are therefore reduced to showing that the $\infty$-category $(\cE_{|t})^{e_s/}$ has an initial object.
This $\infty$-category has an initial object precisely because the functor $\cE_{|t} \to (\cE_{|c_1})_{/t}\xra{\simeq} \cE_{|c_1}$ is assumed to be a right adjoint.

We now address fully-faithfulness of the left adjoint, the existence of which was just established.
The condition that this left adjoint be fully-faithful is the condition that the functor 
\[
(\ev_s)^{e_s/}\colon \Fun_{/\cK}(c_1,\cE)^{e_s/} 
\longrightarrow
(\cE_{|s})^{e_s/}
\]
carries the initial object in the domain, the existence of which was just established, to the initial object in the codomain.
This latter condition is indeed the case precisely because the initial object in the codomain of $(\ev_s)^{e_s/}$ 
factors through the fully-faithful functor~(\ref{e120}).

We now establish that condition~(c) implies condition~(a).
We must show that, for each functor $c_1\to \cK$, the base change $\cE_{|c_1}\to c_1$ is a coCartesian fibration.
So fix such a functor $c_1\to \cK$.
Through the equivalence (a)$\iff$(b), we are to show that the canonical functor $\cE_{|t} \to (\cE_{|c_1})_{/t} \xra{\simeq} \cE_{|c_1}$ is a right adjoint.
Let $e\in \cE_{|c_1}$ be an object.
We must show that the $\infty$-undercategory $(\cE_{|t})^{e/}$ has an initial object.  
If this object $e$ lies over $t$, this $\infty$-category has $(e\xra{=}e)$ as an initial object.
So suppose $e$ lies over $s$, which is to say $e\in \cE_{|s}$.  
Because $\ev_s$ is a Cartesian fibration, the canonical functor
\[
(\cE_{|t})^{e/}:=(\cE_{|c_1})^{e/}\underset{\cE_{|c_1}}\times \cE_{|t}\simeq \Fun_{/\cK}(c_1,\cE)_{|e} \longrightarrow \Fun_{/\cK}(c_1,\cE)^{e/}
~,
\]
which is fully-faithful,
is a left adjoint.  
The assumed condition~(c) gives the existence of an initial object in the codomain of this functor, which in fact lies in the image of this fully-faithful functor left adjoint.
This completes the desired implication.  
\end{proof}

\begin{remark}\label{st.loc.coCart}
Let $\cE\xra{\pi} \cK$ be a functor between $\infty$-categories.
For each morphism $c_1\xra{\lag x\xra{f} y\rag} \cK$, consider the span among $\infty$-categories
\[
\cE_{|x}    
\xla{~\ev_s~} 
\Fun_{/\cK}(c_1,\cE)
\xra{~\ev_t~}
\cE_{|y}~.
\]
Through Lemma~\ref{locallycocart}, if $\pi$ is locally coCartesian the functor $\ev_s$ has a left adjoint, thereby resulting in a composite functor
\[
f_!\colon \cE_{|x}    
\xra{~(\ev_s)^{\vee}~} 
\Fun_{/\cK}(c_1,\cE)
\xra{~\ev_t~}
\cE_{|y}~;
\]
if $\pi$ is locally Cartesian the functor $\ev_t$ has a right adjoint, thereby resulting in a composite functor
\[
\cE_{|x}    
\xla{~(\ev_s)^{\vee}~} 
\Fun_{/\cK}(c_1,\cE)
\xla{~(\ev_t)^{\vee}~}
\cE_{|y}\colon f^\ast~.
\]

\end{remark}

\begin{lemma}\label{coCart.morphi.via.adjoints}
Let $\cE\xra{\pi}\cK$ be a functor between $\infty$-categories.
\begin{enumerate}
\item
Provided the functor $\pi$ is locally coCartesian, the following conditions on a morphism $c_1\xra{\lag e_s\to e_t\rag}\cE$ are equivalent. 
\begin{enumerate}
\item This morphism is locally $\pi$-coCartesian.

\item The left adjoint $\cE_{/\pi e_t} \to \cE_{|\pi e_t}$ carries this morphism to an equivalence.

\end{enumerate}

\item
Provided the functor $\pi$ is locally Cartesian, the following conditions on a morphism $c_1\xra{\lag e_s\to e_t\rag}\cE$ are equivalent. 
\begin{enumerate}
\item This morphism is locally $\pi$-Cartesian.

\item The right adjoint $\cE^{\pi e_s/} \to \cE_{|\pi e_s}$ carries this morphism to an equivalence.

\end{enumerate}

\end{enumerate}

\end{lemma}
\begin{proof}
Statements~(1) and~(2) are dual to one another, so it suffices to prove~(1).

Denote the given morphism as $e_s\xra{\w{f}}e_t$.
Since $\pi$ is assumed locally coCartesian, by base changing along $c_1 \xra{\lag \pi e_s \xra{\pi\w{f}} \pi e_t\rag} \cK$, we can reduce to the case that $\pi$ is a coCartesian fibration $\cE\to c_1$ over the 1-cell. 
After Lemma~\ref{lemma.reformulation}, this morphism is $\pi$-coCartesian if and only if, when regarded as an object in the $\infty$-category
\[
\cE^{e_s/}\underset{c_1^{s/}}\times c_1^{t/}
~\simeq~
(\cE_{|t})^{e_s/}
~,
\]
it is initial.
By Lemma~\ref{locallycocart}, an object in this $\infty$-category is initial if and only if it is the value on $(e_s\xra{\w{f}}e_t)$ of the left adjoint $\cE_{/t} \to \cE_{|t}$.
The result follows.
\end{proof}

The next result shows that, like exponentiable fibrations (Lemma~\ref{exp-char}(3)), (co)Cartesian fibrations can be detected over $[2]$-points at a time.  
The equivalences of conditions (a) and (c) are equivalent to Proposition 2.4.2.8 of \cite{HTT}; we provide a proof for the reader's convenience.

\begin{prop}\label{just.over.2}
Let $\cE\xra{\pi}\cK$ be a functor between $\infty$-categories.
\begin{enumerate}
\item 
The following conditions on $\pi$ are equivalent.  
\begin{enumerate}
\item $\pi$ is a coCartesian fibration.

\item $\pi$ is a locally coCartesian exponentiable fibration.

\item For each functor $[2]\to \cK$, the base change $\cE_{|[2]}\to [2]$ is a coCartesian fibration.

\item For each functor $[2]\to \cK$, the base change $\cE_{|[2]} \to [2]$ is a locally coCartesian exponentiable fibration.

\item $\pi$ is a locally coCartesian fibration and for each functor $[2]\to \cK$, and each lift $\{0<2\}\xra{\lag e_0\to \ov{e}_2\rag} \cE_{|\{0<2\}}$ along $\pi$, the $\infty$-category $\bigl((\cE^{e_0/})_{|1}\bigr)_{/(e_0\to \ov{e}_2)}$ is nonempty.

\item $\pi$ is locally coCartesian and the following condition is satisfied.
\begin{itemize}
\item[~]

Let $[2]\xra{\lag e_0\xra{f}e_1\xra{g}e_2\rag} \cE$ be a functor, selecting the indicated diagram in $\cE$, for which $(e_0\xra{f} e_2)$ is coCartesian with respect to the base change $\cE_{|\{0<1\}}\to \{0<1\}$ and $(e_1\xra{g} e_2)$ is coCartesian with respect to the base change $\cE_{|\{1<2\}}\to \{1<2\}$.  
Let $(e_0\to \ov{e}_2)$ be a morphism in $\cE$, over the morphism $\{0<2\}\xra{\lag \pi e_0\to \pi e_2\rag} \cK$, that is coCartesian with respect to the base change $\cE_{|\{0<2\}} \to \{0<2\}$.  
The canonical morphism in the fiber $\infty$-category $\cE_{|2}$, 
\[
\ov{e}_2 \longrightarrow e_2~,
\]
is an equivalence.
\end{itemize} 

\end{enumerate}

\item 
The following conditions on $\pi$ are equivalent.  
\begin{enumerate}
\item $\pi$ is a Cartesian fibration.

\item $\pi$ is a locally Cartesian  exponentiable fibration.

\item For each functor $[2]\to \cK$, the base change $\cE_{|[2]}\to [2]$ is a Cartesian fibration.

\item For each functor $[2]\to \cK$, the base change $\cE_{|[2]} \to [2]$ is a locally Cartesian exponentiable fibration.

\item $\pi$ is a locally Cartesian fibration and for each functor $[2]\to \cK$, and each lift $\{0<2\}\xra{\lag e_0\to \ov{e}_2\rag} \cE_{|\{0<2\}}$ along $\pi$, the $\infty$-category $\bigl((\cE^{e_0/})_{|1}\bigr)_{/(e_0\to \ov{e}_2)}$ is nonempty. 

\item $\pi$ is locally Cartesian and the following condition is satisfied.
\begin{itemize}
\item[~]

Let $[2]\xra{\lag e_0\xra{f}e_1\xra{g}e_2\rag} \cE$ be a functor, selecting the indicated diagram in $\cE$, for which $(e_0\xra{f} e_2)$ is Cartesian with respect to the base change $\cE_{|\{0<1\}}\to \{0<1\}$ and $(e_1\xra{g} e_2)$ is Cartesian with respect to the base change $\cE_{|\{1<2\}}\to \{1<2\}$.  
Let $(\ov{e}_0\to e_2)$ be a morphism in $\cE$, over the morphism $\{0<2\}\xra{\lag \pi e_0\to \pi e_2\rag} \cK$, that is Cartesian with respect to the base change $\cE_{|\{0<2\}} \to \{0<2\}$.  
The canonical morphism in the fiber $\infty$-category $\cE_{|0}$, 
\[
e_0 \longrightarrow \ov{e}_0~,
\]
is an equivalence.
\end{itemize}

\end{enumerate}

\end{enumerate}

\end{prop}

\begin{proof}
The assertion concerning coCartesian fibrations implies the assertion concerning the Cartesian fibrations, as implemented by replacing a Cartesian fibration by its opposite.  
We are therefore reduced to proving the assertion concerning coCartesian fibrations.

We establish these implications
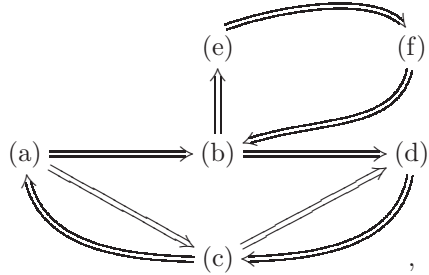
\[
\xymatrix{
&&
{\rm (e)}     \ar@(-,u)@{=>}[rr]
&&
{\rm (f)}  \ar@(d,-)@{=>}[lld]
\\
{\rm (a)}  \ar@{=>}[rr]  \ar@{=>}[drr]
&&
{\rm (b)}  \ar@{=>}[rr]    \ar@{=>}[u]
&&
{\rm (d)}  \ar@(d,r)@{=>}[dll]  
\\
&&
{\rm (c)}  \ar@{=>}[urr]  \ar@(l,d)@{=>}[ull]
&&
,
}
\]
in which the straight ones are quick, as we explain first.

Suppose~(a), that $\pi$ is a coCartesian fibration.
Then, by definition, $\pi$ is a locally coCartesian fibration.
Lemma~\ref{exp-examples} gives that $\pi$ is an exponentiable fibration.
So (a) implies (b).
For the same reason, (c) implies (d).
Also, Lemma~\ref{Cart.base.change} gives that each base change $\cE_{|[2]}\to [2]$ is also a coCartesian fibration.
So (a) implies (c).
Corollary~\ref{exps-pullback} gives that exponentiable fibrations are closed under base change;
locally coCartesian fibrations are manifestly closed under base change.
Therefore (b) implies (d). 
The criterion of Lemma~\ref{exp-char}(6) for being an exponentiable fibration immediately gives that (b) implies (e).

We now establish that (d) implies (c); so suppose (d) is true.
The problem immediately reduces to showing that a locally coCartesian exponentiable fibration $\cE\xra{\pi} [2]$ is a coCartesian fibration.
Through Lemma~\ref{lemma.reformulation}, this is the problem of showing each solid diagram of $\infty$-categories
\[
\xymatrix{
\ast   \ar[d]_-{\lag s \rag}  \ar[rr]^-{\lag e_i\rag}
&&
\cE  \ar[d]^-{\pi}  
\\
c_1  \ar[rr]_-{\lag i\leq j\rag}   \ar@{-->}[urr]^-{\lag e_i\to e_j\rag}
&&
[2]
}
\]
admits a filler that is initial in the pullback $\infty$-category $\cE^{e_i/}\underset{[2]^{i/}}\times [2]^{j/}$.  
Through Lemma~\ref{lemma.reformulation}, the assumption that $\cE\to [2]$ is assumed locally coCartesian directly solves this problem in the cases that $(i,j)\neq (0,1)$.  
So assume $(i,j)=(0,1)$.  
Using that $\cE\to [2]$ is locally coCartesian, choose, through Lemma~\ref{lemma.reformulation}, such a lift $(e_0\to e_1)$, which is initial in the fiber product from the base change, 
\[
(\cE_{|\{0<1\}})^{e_0/}\underset{\{0<1\}^{0/}}\times \{0<1\}^{1/}~.
\]
Since initial functors compose (Lemma~\ref{2.o.3}), initiality of this lift, as an object in $\cE^{e_0/}\underset{[2]^{0/}}\times [2]^{1/}$, is therefore implied by initiality of the canonical functor
\[
(\cE^{e_0/})_{|1}~\simeq~(\cE_{|\{0<1\}})^{e_0/}\underset{\{0<1\}^{0/}}\times \{0<1\}^{1/} \longrightarrow \cE^{e_0/}\underset{[2]^{0/}}\times [2]^{1/}~\simeq~ (\cE^{e_0/})_{|\{1<2\}}~.
\]
We establish initiality of this functor using Quillen's Theorem A.
Let $(e_0\to e')$ be an object in $(\cE^{e_0/})_{|\{1<2\}}$.  
We must show the classifying space of the $\infty$-overcategory
\begin{equation}\label{99}
\bigl((\cE^{e_0/})_{|1}\bigr)_{/(e_0\to e')}
\end{equation}
is contractible.
In the case that $e'\in \cE$ lies over $1\in [2]$, this $\infty$-category~(\ref{99}) has $(e_0\to e_1)$ as an initial object.  The desired contractibility follows.
Now suppose $e'\in \cE$ lies over $2\in [2]$.  
In this case, the $\infty$-category~(\ref{99}) has contractible classifying space precisely because $\cE\to [2]$ is assumed an exponentiable fibration, using Lemma~\ref{exp-char}(6).
This concludes the implication (d)$\implies$(c).

We now establish that (c) implies (a); so suppose (c) is true, that each base change $\cE_{|[2]}\to [2]$ is a coCartesian fibration.
Consider a solid diagram 
\begin{equation}\label{yet.another}
\xymatrix{
\ast   \ar[d]_-{\lag s \rag}  \ar[rr]^-{\lag e_s\rag}
&&
\cE  \ar[d]^-{\pi}  
\\
c_1  \ar[rr]_-{\lag x \to y\rag}   \ar@{-->}[urr]^-{\lag e_s\to e_t\rag}
&&
\cK
}
\end{equation}
among $\infty$-categories.
By assumption, there is a coCartesian lift, as indicated, with respect to the base change $\cE_{|c_1} \to c_1$.  
Denote this lift as $\alpha$.
We show that $\alpha$ is an initial object in the fiber product $\infty$-category $\cE^{e_s/} \underset{\cK^{x/}}\times \cK^{y/}$.

An object in this fiber product is a diagram, $\delta$, 
\[
\xymatrix{
\{s\to +\}  \ar[d]  \ar[rr]^-{\lag e_s \to e_+\rag}
&&
\cE  \ar[d]^-{\pi}  
\\
[2]=\{s\to t\to +\}  \ar[rr]_-{\lag x \to  y\to z\rag}   
&&
\cK
}
\]
extending the solid diagram~(\ref{yet.another}).
By assumption that the base change $\cE_{|[2]}\to [2]$ is a coCartesian fibration, there is a lift
\[
\xymatrix{
\ast   \ar[d]_-{\lag s \rag}  \ar[rr]^-{\lag e_s\rag}
&&
\cE_{|[2]}  \ar[d]    \ar[rr]
&&
\cE  \ar[d]^-{\pi}
\\
c_1  \ar[rr]_-{\lag s \to t\rag}   \ar@{-->}[urr]^-{\lag e_s\to e'_t\rag}
&&
[2]=\{s\to t\to +\}  \ar[rr]_-{\lag x \to  y\to z\rag}  
&&
\cK
}
\]
for which the canonical functor
\[
(\cE_{|[2]})^{e'_t/}    \longrightarrow   (\cE_{|[2]})^{e_s/} \underset{[2]^{s/}}\times [2]^{t/} 
\]
is an equivalence between $\infty$-categories. 
Denote the above lift as $\beta$.
By choice of $\alpha$, it is an initial object in the fiber product $\infty$-category $(\cE_{|c_1})^{e_s/} \underset{c_1^{s/}}\times c_1^{t/}$.
Therefore there is a unique morphism $\alpha \to \beta$ in $(\cE_{|c_1})^{e_s/} \underset{c_1^{s/}}\times c_1^{t/}$.  
Likewise, because $\beta$ is an initial object in the fiber product $(\cE_{|[2]})^{e_s/} \underset{[2]^{s/}}\times [2]^{t/}$, there is a unique morphism $\beta \to \alpha$ in this fiber product.
We conclude an equivalence $\alpha \simeq \beta$ because the canonical functor between fiber products $(\cE_{|c_1})^{e_s/} \underset{c_1^{s/}}\times c_1^{t/} \to (\cE_{|[2]})^{e_s/} \underset{[2]^{s/}}\times [2]^{t/}$ is fully-faithful.
This establishes that $\alpha$ is coCartesian with respect to each base change $\cE_{|[2]}\to [2]$.

We now show that $\alpha$ is $\pi$-coCartesian.  
Notice that the canonical pullback square among $\infty$-categories
\[
\xymatrix{
(\cE_{|[2]})^{e_t/}   \ar[d]   \ar[rr]
&&
\cE^{e_t/}  \ar[d]
\\
(\cE_{|[2]})^{e_s/} \underset{[2]^{s/}}\times [2]^{t/}    \ar[rr]
&&
\cE^{e_s/} \underset{\cK^{x/}}\times \cK^{y/},
}
\]
in which the vertical functors are left fibrations.
Notice, also, that the object $\ast \xra{\delta} \cE^{e_s/} \underset{\cK^{x/}}\times \cK^{y/}$ canonically factors through the bottom horizontal functor in this diagram.
We have established that the fiber over this lift of $\delta$ of the left vertical left fibration is a contractible $\infty$-groupoid.  
Because this square is a pullback, the fiber over $\delta$ of the right vertical left fibration is also a contractible $\infty$-groupoid.
We conclude that $\alpha$ is an initial object in the fiber product $\cE^{e_s/} \underset{\cK^{x/}}\times \cK^{y/}$, as desired.

We now establish that (f) implies (b).
Through the criterion of Lemma~\ref{exp-char}(6), we must show that, for each functor $[2]\to \cK$ and each lift $\{0<2\}\xra{\lag e_0\to e'_2\rag} \cE$ along $\pi$, the $\infty$-category
\[
\bigl((\cE^{e_0/})_{|1}\bigr)_{/(e_0\to e'_2)}
\]
has contractible classifying space.  
Choose a lift $[2]\xra{\lag e_0\to e_1\to e_2\rag} \cE$ with the same value on $0$ as in the assumptions of condition~(f).  
The assumed coCartesian properties of the morphisms $(e_0\to e_1)$ and $(e_1\to e_2)$ imply this lift defines an initial object in the $\infty$-category $\bigl((\cE^{e_0/})_{|1}\bigr)_{/(e_0\to e'_2)}$ provided it is nonempty.  
We are thus reduced to showing this $\infty$-category is nonempty.  
Choose a lift $\{0<2\}\xra{\lag e_0\to \ov{e}_2\rag} \cE$ with the same value on $0$, as in the assumptions of condition~(f).  
The assumed coCartesian property of this morphism $(e_0\to \ov{e}_2)$ determines a natural transformation $\beta\colon (e_0\to \ov{e}_2)\to (e_0\to e'_2)$ together with an identification of the restriction $\beta_{|0} \colon e_0\xra{=}e_0$ as the identity.
This $\beta$ determines a functor between $\infty$-categories:
\[
\bigl((\cE^{e_0/})_{|1}\bigr)_{/(e_0\to \ov{e}_2)} 
\longrightarrow
\bigl((\cE^{e_0/})_{|1}\bigr)_{/(e_0\to e'_2)}~.
\]
We are therefore reduced to showing this domain $\infty$-category is nonempty.  
This is exactly implied by the condition~(f).

We now establish that (e) implies (f).
Consider the assumptions given in condition~(f).  
The assumed condition~(e) states that the $\infty$-category $\bigl((\cE^{e_0/})_{|1}\bigr)_{/(e_0\to \ov{e}_2)}$ is nonempty.  
The assumed coCartesian properties of each of the morphisms $(e_0\to e_1)$ and $(e_1\to e_2)$ give a unique natural transformation $\alpha \colon (e_0\to e_1\to e_2)\to (e_0\to \ov{e}_1\to \ov{e}_2)$ between functors $[2]\to \cE$ together with an identification of the restriction $\alpha_{|0}\colon e_0\xra{=} e_0$.   
In this way, we see that the object $(e_0\to e_1\to e_2)$ determines an initial object in the $\infty$-category $\bigl((\cE^{e_0/})_{|1}\bigr)_{/(e_0\to \ov{e}_2)}$.  
The assumed coCartesian property of the morphism $(e_0\to \ov{e}_2)$ gives a unique natural transformation $\beta\colon (e_0\to \ov{e}_2)\to (e_0\to e_2)$ between functors $\{0<2\}\to \cE$ together with an identification of the restriction $\beta_{|0}\colon e_0\xra{=} e_0$.   
We conclude that $\alpha_{|\{0<2\}}$ is a retraction: $\alpha_{|\{0<2\}}\beta \simeq \id_{(e_0\to \ov{e}_2)}$.
The above initiality of the object in $\bigl((\cE^{e_0/})_{|1}\bigr)_{/(e_0\to \ov{e}_2)}$ determined by $(e_0\to e_1\to e_2)$ gives that $\beta$ is in fact an inverse to $\alpha$: $\beta \alpha_{|\{0<2\}} \simeq \id_{(e_0\to e_2)}$.  
Restricting to $2\in [2]$ reveals that the canonical morphism $\ov{e}_2\to e_2$ is an equivalence, as desired.

\end{proof}

\begin{remark}
We follow up on Remark~\ref{st.loc.coCart}.  
The property of a functor $\cE \ra \cK$ being a coCartesian fibration exactly ensures that the assignment $x \mapsto\cE_{|x}$ can be assembled as a functor $\cK\to \Cat$. 
The criterion of Theorem~\ref{theorem-cocart} breaks this into two parts. 
The first condition ensures that each morphism $c_1\xra{\lag x\xra{f} y\rag}\cK$ defines a functor between fibers $f_!:\cE_{|x} \ra\cE_{|y}$, associated to a lax functor from $\cK$ to $\Cat$. 
Secondly, being an exponentiable fibration then guarantees associativity: for each functor $[2]\xra{\lag x\xra{f}y \xra{g} z\rag} \cK$, the canonical natural transformation $(gf)_! \to g_! f_!$ between functors $\cE_{|x} \ra \cE_{|z}$ is an equivalence; equivalently, the lax functor defined by being locally coCartesian is, in fact, a functor.
\end{remark}

The preceding results can now be assembled to establish our characterization of (co)Cartesian fibrations in terms of exponentiable fibrations -- these are the assertions in Theorem~\ref{main.thm'} concerning (co)Cartesian fibrations. 
We reference the Cartesian symmetric monoidal $\infty$-category $\Cat$, as well as its opposite $\Cat^{\op}$, with the coCartesian symmetric monoidal structure.  
\begin{theorem}\label{theorem-cocart}
\begin{enumerate}
\item[~]

\item 
There is a symmetric monoidal monomorphism between flagged $\infty$-categories:
\[
\Cat~\hookrightarrow~\Corr~.
\]
For each $\infty$-category $\cK$, a functor $\cK\xra{\lag \cE\xra{\sf e.fib} \cK\rag} \Corr$ classifying the indicated exponentiable fibration, factors through $\Cat\hookrightarrow \Corr$ if and only if the exponentiable fibration $\cE\to \cK$ is also a locally coCartesian fibration.

\item 
There is a symmetric monoidal monomorphism between flagged $\infty$-categories:
\[
\Cat^{\op}~\hookrightarrow~\Corr~.
\]
For each $\infty$-category $\cK$, a functor $\cK\xra{\lag \cE\xra{\sf e.fib} \cK\rag} \Corr$ classifying the indicated exponentiable fibration, factors through $\Cat^{\op}\hookrightarrow \Corr$ if and only if the exponentiable fibration $\cE\to \cK$ is also a locally Cartesian fibration.  

\end{enumerate}

\end{theorem}

\subsection{(co)Cartesian-replacement}
We describe, for each $\infty$-category $\cK$, left adjoints to the monomorphisms $\cCart_{\cK} \hookrightarrow \Cat_{/\cK}$ and $\Cart_{\cK} \hookrightarrow \Cat_{/\cK}$. 
This material is a synopsis of the work~\cite{gepner-haugseng-nikolaus}.

For each functor $\cX \to \cY$ between $\infty$-categories, we denote the pullbacks
\[
\xymatrix{
\Ar(\cY)^{|\cX}  \ar[rr]  \ar[d]
&&
\Ar(\cY)  \ar[d]^-{(\ev_s,\ev_t)}
&&
\Ar(\cY)_{|\cX}  \ar[d]   \ar[ll]
\\
\cX\times \cY  \ar[rr]
&&
\cY\times \cY 
&&
\cY\times \cX   .    \ar[ll]
}
\]

\begin{lemma}\label{candidate}
Each functor $\cX\to \cY$ between $\infty$-categories canonically factors as in the diagram
\[
\xymatrix{
\Ar(\cY)^{|\cX}   \ar[drr]_-{\ev_t}
&&
\cX  \ar[rr]^-{\rm right~adjoint}  \ar[ll]_-{\rm left~adjoint}  \ar[d]
&&
\Ar(\cY)_{|\cX} \ar[dll]^-{\ev_s}
\\
&&
\cY  
&&
;
}
\]
in this diagram, $\ev_t$ is a coCartesian fibration and $\ev_s$ is a Cartesian fibration, and the horizontal functors are fully-faithful adjoints as indicated.  

\end{lemma}

\begin{proof}
The functor $\cX\to \cY$ determines a solid diagram of $\infty$-categories:
\[
\xymatrix{
&
\cX  \ar[dl]_-{\id_\bullet}  \ar[dr]^-{\id_\bullet}  \ar[d]^-{=}
&
\\
\Ar(\cX)  \ar[r]_-{\ev_s}  \ar[d]
&
\cX  \ar[d]
&
\Ar(\cX)  \ar[l]^-{\ev_t}  \ar[d]
\\
\Ar(\cY)_{|\cX}  \ar[r]^-{\ev_s}
&
\cY 
&
\Ar(\cY)^{|\cX}         \ar[l]_-{\ev_t}  
}
\]
in which the functor $\id_\bullet\colon \cX = \Fun(\ast,\cX) \to \Fun(c_1,\cX) = \Ar(\cX)$ is pullback along the epimorphism $c_1\to \ast$.  
The asserted canonical factorizations follow.

Consider the objects $(\ast \xra{=} \ast)$ and $(\ast\xra{\lag s \rag} c_1)$ in the $\infty$-category $\Ar(\Cat)$ of arrows in $\Cat$.  
The self-enrichment of the $\infty$-category $\Cat$, induced from the fact that $\Cat$ is Cartesian closed, determines a $\Cat$-enrichment of the $\infty$-category $\Ar(\Cat)$.  
In particular, it makes sense to consider an adjunction in $\Ar(\Cat)$.  
Note that the functor $\ast \xra{\lag s\rag} c_1$ is a left adjoint in an adjunction (in $\Cat$) whose unit transformation is an equivalence.
It follows that $(\ast\xra{=}\ast) \xra{ ( = , \lag s \rag ) } (\ast \xra{\lag s \rag} c_1)$ is a left adjoint in an adjunction in $\Ar(\Cat)$ whose unit 2-cell is an equivalence.  
Therefore, for the object $(\cX\to \cY)$ in $\Ar(\Cat)$, the functor between ${\sf hom}$-$\infty$-categories implemented by precomposition by the right adjoint,
\[
{\sf hom}_{\Ar(\Cat)}\Bigl( (\ast \xra{=}\ast) , (\cX \to \cY) \Bigr)
\longrightarrow
{\sf hom}_{\Ar(\Cat)}\Bigl( (\ast \xra{ \lag s \rag } c_1) , (\cX \to \cY) \Bigr)
~,
\]
is a left adjoint in an adjunction whose counit transformation is an equivalence.  
Now, identify the domain of this functor as $\cX$, the codomain of this functor as $\Ar(\cY)^{|\cX}$, and the functor itself as the of the previous paragraph.  
We conclude that the functor $\cX\to \Ar(\cY)^{|\cX}$ is a fully-faithful left adjoint, as desired.
A dual argument verifies that the functor $\cX \to \Ar(\cY)_{|\cX}$ is a fully-faithful right adjoint.

We wish to show the functor $\ev_t\colon \Ar(\cY)^{|\cX}\to \cY$ is a coCartesian fibration; and that the functor $\ev_s\colon \Ar(\cY)_{|\cX} \to \cY$ is a Cartesian fibration.  
These two problems are logically equivalent, as implemented by taking opposites.
We are therefore reduced to establishing the first.  
Let $\cJ \xra{F} \cY$ be a functor.
The datum of a lift $\gamma\colon \cJ \to \cX$ along the given functor $\cX\to \cY$ is the datum of a diagram of $\infty$-categories:
\[
\xymatrix{
\cJ  \ar[rr]^-{\w{F}_s}  \ar[d]_-{s}
&&
\cX  \ar[d]
\\
\cJ \times c_1  \ar[rr]^-{\ov{F}}  
&&
\cY
\\
\cJ  \ar[u]^-{t}  \ar[urr]_-{F}
&&
.
}
\]
In the case $\cJ = c_1$ is a 1-cell, such a lift is a $\ev_t$-coCartesian morphism if and only if the functor $\w{F}_s$ in the above diagram factors through the epimorphism $c_1\to \ast$, in which case, $\ov{F}$ factors through the epimorphism $(c_1\times c_1)\underset{c_1\times \{s\}}\amalg \ast \xra{\simeq} [2]$.  
To show that $\ev_t$ is a coCartesian fibration, we must then find a filler in each diagram
\[
\xymatrix{
\{0\}  \ar[rr]  \ar[d]
&&
\cX  \ar[d]
\\
\{0<1\} \ar@(u,r)[rr]  \ar[r]
&
[2]  \ar@{-->}[r]
&
\cY
\\
\{1\} \ar[u]  \ar[r]
&
\{1<2\}  \ar[u]  \ar[ur]
&
.
}
\]
There is a unique such filler because the lower square is a pushout.  
This concludes the verification that $\ev_t$ is a coCartesian fibration.  

\end{proof}

\begin{lemma}\label{if.already}
Let $\cE\xra{\pi} \cK$ be a functor between $\infty$-categories.
\begin{enumerate}
\item 
The functor $\pi$ is a coCartesian fibration if and only if the functor $\cE\to \Ar(\cK)^{|\cE}$ has the following properties:
\begin{enumerate}

\item It is a right adjoint.

\item The unit of the resulting adjunction is carried by $\pi$ to an equivalence in $\cK$.

\end{enumerate}
Should the latter clause be true, the left adjoint in this adjunction carries $\pi$-coCartesian morphisms to $\pi$-coCartesian morphisms.

\item
The functor $\pi$ is a Cartesian fibration if and only if the functor $\cE\to \Ar(\cK)_{|\cE}$ has the following properties.
\begin{enumerate}

\item It is a left adjoint.

\item The counit of the resulting adjunction is carried by $\pi$ to an equivalence in $\cK$.

\end{enumerate}
Should the latter clause be true, the right adjoint in this adjunction carries $\pi$-Cartesian morphisms to $\pi$-Cartesian morphisms.  

\end{enumerate}

\end{lemma}

\begin{proof}
The two assertions imply one another, as implemented by taking opposites.
We are therefore reduced to proving the first assertion.

From Lemma~\ref{lemma.reformulation}(1b), $\pi$ is a coCartesian fibration if and only if, for each object $\gamma\in \Ar(\cK)^{|\cE}$ defining a diagram
\[
\xymatrix{
\{s\}  \ar[rr]^-{\lag e_s\rag}  \ar[d]
&&
\cE  \ar[d]^-{\pi}
\\
c_1 \ar[rr]_-{\lag x_s \xra{f} x_t\rag}   \ar@{-->}[urr]^-{\lag e_s\to f_!e_s\rag}
&&
\cK,
}
\]
there exists a filler that is initial when regarded as an object in the fiber product $\infty$-category $\cE^{e_s/}\underset{\cK^{x_s/}}\times \cK^{x_t/}$.  
Such a filler is, in particular, the datum of an object in the $\infty$-undercategory $\cE^{\gamma/}:=\cE\underset{\Ar(\cK)^{|\cE}}\times (\Ar(\cK)^{|\cE})^{\gamma/}$. 
In this way, we see that $\pi$ is a coCartesian fibration if and only if the canonical fully-faithful functor $\cE\to \Ar(\cK)^{|\cE}$ is a right adjoint and there is an identification of the value of $\pi$ on the unit 2-cell as a degenerate 2-cell.

\end{proof}

With Lemma~\ref{if.already}, Lemma~\ref{secondhalf-locallycocart}(1a) has the following generalization.
\begin{cor}\label{free.fib.adj}
Let $\cE \xra{\pi} \cK$ be a functor between $\infty$-categories.
\begin{enumerate}

\item
The following conditions on the functor $\pi$ are equivalent.
\begin{enumerate}
\item 
The functor $\pi$ is a coCartesian fibration.

\item 
For each $\infty$-category $\cJ \to \cK$ over $\cK$, the canonical functor between $\infty$-categories over $\cJ$
\[
\cE_{|\cJ}\longrightarrow  \Ar(\cK)^{|\cE}_{|\cJ}
\]
is a fully-faithful right adjoint.  
Furthermore, for each functor $\cJ\to \cJ'\to \cK$ between $\infty$-categories over $\cK$, the a priori lax commutative diagram involving left adjoints to the above
\[
\xymatrix{
\cE_{|\cJ} \ar[d]
&&
\Ar(\cK)^{|\cE}_{|\cJ} \ar[d] \ar[ll]_-{\sf l.adj}
\\
\cE_{|\cJ'} 
&&
\Ar(\cK)^{|\cE}_{|\cJ'} \ar[ll]_-{\sf l.adj}
}
\]
in fact commutes.

\item 
For each $\infty$-category $\cJ \to \cK$ over $\cK$, the canonical functor between $\infty$-categories of sections
\[
\Fun_{/\cK}(\cJ, \cE)\longrightarrow  \Fun_{/\cK}\bigl(\cJ, \Ar(\cK)^{|\cE}\bigr)
\]
is a fully-faithful right adjoint.  
Furthermore, for each functor $\cJ\to \cJ'\to \cK$ between $\infty$-categories over $\cK$, the a priori lax commutative diagram involving left adjoints
\[
\xymatrix{
\Fun_{/\cK}(\cJ', \cE) \ar[d]
&&
\Fun_{/\cK}\bigl(\cJ', \Ar(\cK)^{|\cE}\bigr) \ar[d] \ar[ll]_-{\sf l.adj}
\\
\Fun_{/\cK}(\cJ, \cE)
&&
\Fun_{/\cK}\bigl(\cJ, \Ar(\cK)^{|\cE}\bigr) \ar[ll]_-{\sf l.adj}
}
\]
in fact commutes.

\end{enumerate}

\item
The following conditions on the functor $\pi$ are equivalent.
\begin{enumerate}
\item 
The functor $\pi$ is a Cartesian fibration.

\item 
For each $\infty$-category $\cJ \to \cK$ over $\cK$, the canonical functor between $\infty$-categories over $\cJ$
\[
\cE_{|\cJ}\longrightarrow  \Ar(\cK)_{|\cE}^{|\cJ}
\]
is a fully-faithful left adjoint.  
Furthermore, for each functor $\cJ\to \cJ'\to \cK$ between $\infty$-categories over $\cK$, the a priori lax commutative diagram involving right adjoints to the above,
\[
\xymatrix{
\cE_{|\cJ} \ar[d]
&&
\Ar(\cK)^{|\cE}_{|\cJ} \ar[d] \ar[ll]_-{\sf r.adj}
\\
\cE_{|\cJ'} 
&&
\Ar(\cK)^{|\cE}_{|\cJ'} \ar[ll]_-{\sf r.adj}
}
\]
in fact commutes.

\item 
For each $\infty$-category $\cJ \to \cK$ over $\cK$, the canonical functor between $\infty$-categories of sections
\[
\Fun_{/\cK}(\cJ, \cE)\longrightarrow  \Fun_{/\cK}\bigl(\cJ, \Ar(\cK)_{|\cE}\bigr)
\]
is a fully-faithful left adjoint.  
Furthermore, for each functor $\cJ\to \cJ'\to \cK$ between $\infty$-categories over $\cK$, the a priori lax commutative diagram involving right adjoints
\[
\xymatrix{
\Fun_{/\cK}(\cJ', \cE) \ar[d]
&&
\Fun_{/\cK}\bigl(\cJ', \Ar(\cK)^{|\cE}\bigr) \ar[d] \ar[ll]_-{\sf r.adj}
\\
\Fun_{/\cK}(\cJ, \cE)
&&
\Fun_{/\cK}\bigl(\cJ, \Ar(\cK)^{|\cE}\bigr) \ar[ll]_-{\sf r.adj}
}
\]
in fact commutes.

\end{enumerate}

\end{enumerate}

\end{cor}

\begin{proof}
Assertion~(1) implies assertion~(2), as implemented by taking opposites.
We are therefore reduced to proving assertion~(1).

The implications (a)$\iff$(b) are directly implied by the equivalence of Lemma~\ref{if.already}(1).  
For each $\infty$-category $\cJ$, an adjunction $\cC\leftrightarrows \cD$ determines an adjunction $\Fun(\cJ,\cC)\rightleftarrows \Fun(\cJ,\cD)$ between $\infty$-categories of functors.
The implication (b)$\implies$(c) follows.
We now establish the implication (c)$\implies$(a).
Through Lemma~\ref{secondhalf-locallycocart}(1a), the case $\cJ \simeq \ast$ gives that (c) implies $\pi$ is a locally coCartesian fibration.
The case that $\cJ\to \cJ'$ is $\{s\}\to c_1$ implies the composition of two $\pi$-locally coCartesian morphisms is a $\pi$-locally coCartesian morphism.
Proposition~\ref{just.over.2}(f) then gives that $\pi$ is in fact a coCartesian fibration.  
\end{proof}

\begin{cor}\label{pre.adj}
Let $\cE\xra{\pi} \cK$ be a functor between $\infty$-categories.
\begin{enumerate}
\item 
For each coCartesian fibration $\cZ\to \cK$, the functor
\[
\Fun_{/\cK}\bigl( \Ar(\cK)^{|\cE},\cZ\bigr) \longrightarrow \Fun_{/\cK}(\cE,\cZ) ~,
\]
which is restriction along the functor $\cE\to \Ar(\cK)^{|\cE}$ over $\cK$, restricts as an equivalence 
\[
\Fun_{/\cK}^{\sf cCart}\bigl( \Ar(\cK)^{|\cE},\cZ\bigr)\xra{~\simeq~} \Fun_{/\cK}(\cE,\cZ)
\]
from the full $\infty$-subcategory consisting of those functors over $\cK$ that carry coCartesian morphisms to coCartesian morphisms.

\item 
For each Cartesian fibration $\cZ\to \cK$, the functor
\[
\Fun_{/\cK}\bigl( \Ar(\cK)_{|\cE},\cZ\bigr) \longrightarrow \Fun_{/\cK}(\cE,\cZ) ~,
\]
which is restriction along the functor $\cE\to \Ar(\cK)_{|\cE}$ over $\cK$, restricts as an equivalence 
\[
\Fun_{/\cK}^{\sf Cart}\bigl( \Ar(\cK)_{|\cE},\cZ\bigr)\xra{~\simeq~} \Fun_{/\cK}(\cE,\cZ)
\]
from the full $\infty$-subcategory consisting of those functors over $\cK$ that carry Cartesian morphisms to Cartesian morphisms.

\end{enumerate}

\end{cor}

\begin{proof}
The two assertions imply one another, as implemented by taking opposites.
We are therefore reduced to proving assertion~(1).

Note the evident functoriality,  $\Ar(\cK)^{|-}\colon \Cat_{/\cK} \to \Cat_{/\cK}$.  
Pulling from the proof of Lemma~\ref{candidate} where, for each $\cU\to \cK$, the coCartesian morphisms of $\ev_t\colon \Ar(\cK)^{|\cU} \to \cK$ are identified, this functor evidently factors
\[
\Ar(\cK)^{|-}\colon \Cat_{/\cK} \longrightarrow \cCart_{\cK}~.  
\]
This functor determines a functor
\[
\Ar(\cK)^{|-}\colon \Fun_{/\cK}(\cE,\cZ) 
\longrightarrow 
\Fun_{/\cK}^{\sf cCart}\bigl(\Ar(\cK)^{|\cE},\Ar(\cK)^{|\cZ}\bigr)~.  
\]
Now, fix a coCartesian fibration $\cZ\xra{\pi'}\cK$.
From the Definition~\ref{def.coCart} of a coCartesian morphism, a $\pi'$-coCartesian morphism is an equivalence whenever $\pi'$ carries it to an equivalence in $\cK$.
From the description of the left adjoint in Lemma~\ref{if.already}, for each functor $\cE\to \cZ$ over $\cK$, there is a canonically commutative diagram of $\infty$-categories over $\cK$:
\[
\xymatrix{
\cE  \ar[r]^-{F}  \ar[d]
&
\cZ  
\\
\Ar(\cK)^{|\cE}  \ar[r]^-{\Ar(\cK)^{|F}}
&
\Ar(\cK)^{|\cZ} .     \ar[u]_-{\rm left \ adjoint}
}
\]
It follows that the diagram of $\infty$-categories
\[
\xymatrix{
\Fun_{/\cK}(\cE,\cZ) \ar[rr]^-{=}  \ar[d]_-{\Ar(\cK)^{|-}}
&&
\Fun_{/\cK}(\cE,\cZ) 
\\
\Fun_{/\cK}^{\sf cCart}\bigl(\Ar(\cK)^{|\cE},\Ar(\cK)^{|\cZ}\bigr)   \ar[rr]^-{\rm left~adjoint}
&&
\Fun_{/\cK}^{\sf cCart}\bigl(\Ar(\cK)^{|\cE},\cZ\bigr)  \ar[u]_-{\rm restriction}
}
\]
commutes, 
in which the bottom horizontal functor is postcomposition with the left adjoint of Lemma~\ref{if.already}.  
We conclude that the functor $\Fun_{/\cK}^{\sf cCart}\bigl( \Ar(\cK)^{|\cE},\cZ\bigr)\to  \Fun_{/\cK}(\cE,\cZ)$ under scrutiny is a retraction.

On the other hand, from the universal property of $\pi'$-coCartesian morphisms, for each functor $\Ar(\cK)^{|\cE} \xra{G} \cZ$ over $\cK$, there is a canonical 2-cell witnessing the lax commutative diagram of $\infty$-categories over $\cK$:
\[
\xymatrix{
\Ar(\cK)^{|\cE}  \ar@(u,u)[rr]^-{G}  \ar[dr]_-{\Ar(\cK)^{|G_{|\cE}}}
&
\Uparrow
&
\cZ
\\
&
\Ar(\cK)^{|\cZ}  \ar[ur]_-{\rm left~adjoint}
&
.
}
\]
For each object $(e,\pi e\xra{f} k)\in \Ar(\cK)^{|\cE}$, this 2-cell specializes as the canonical morphism from the coCartesian pushforward
\[
f_!\bigl( 
G(e,\id_{\pi e})
\bigr)
\longrightarrow
G(e,f)
\]
in the fiber $\infty$-category $(\pi')^{-1}(\pi e)$.
So this 2-cell is invertible if and only if $G$ carries $\ev_t$-coCartesian morphisms to $\pi'$-coCartesian morphisms.
It follows that the diagram of $\infty$-categories
\[
\xymatrix{
\Fun_{/\cK}(\cE,\cZ)   \ar[d]_-{\Ar(\cK)^{|-}}
&&
\Fun_{/\cK}^{\sf cCart}\bigl(\Ar(\cK)^{|\cE},\cZ\bigr)  \ar[ll]^-{\rm restriction}
\\
\Fun_{/\cK}^{\sf cCart}\bigl(\Ar(\cK)^{|\cE},\Ar(\cK)^{|\cZ}\bigr)   \ar[rr]^-{\rm left~adjoint}
&&
\Fun_{/\cK}^{\sf cCart}\bigl(\Ar(\cK)^{|\cE},\cZ\bigr)  \ar[u]_-{=}
}
\]
commutes.
We conclude that the section of the functor $\Fun_{/\cK}^{\sf cCart}\bigl( \Ar(\cK)^{|\cE},\cZ\bigr)\to  \Fun_{/\cK}(\cE,\cZ)$ constructed in the previous paragraph is an inverse.  
This establishes the desired result.

\end{proof}

Corollary~\ref{pre.adj} has the following immediate consequence.  
\begin{theorem}\label{cCart.adjoint}
For each $\infty$-category $\cK$, the monomorphisms
\[
\cCart_{\cK}~ \hookrightarrow~\Cat_{/\cK}
\qquad\text{ and }\qquad
\Cart_{\cK}~\hookrightarrow~\Cat_{/\cK}
\]
are each right adjoints; their left adjoints respectively evaluate as
\[
(-)^{\widehat{~}}_{\sf cCart}\colon \Cat_{/\cK}\longrightarrow \cCart_{\cK}~,\qquad 
(\cE\to \cK)\mapsto \bigl(\Ar(\cK)^{|\cE} \xra{\ev_t}\cK\bigr)
\]
and
\[
(-)^{\widehat{~}}_{\sf Cart}\colon\Cat_{/\cK}\longrightarrow \Cart_{\cK}~,\qquad 
(\cE\to \cK)\mapsto \bigl(\Ar(\cK)_{|\cE} \xra{\ev_s}\cK\bigr)~.
\]

\end{theorem}

\begin{terminology}\label{def.cartesioning}
Let $\cE\xra{\pi}\cK$ be a functor between $\infty$-categories.
We refer to the values of the left adjoint $(\cE\to \cK)^{\widehat{~}}_{\sf cCart}$ as the \emph{coCartesian-replacement (of $\pi$)}.
We refer to the values of the left adjoint $(\cE\to \cK)^{\widehat{~}}_{\sf Cart}$ as the \emph{Cartesian-replacement (of $\pi$)}.

\end{terminology}

\subsection{Left fibrations and right fibrations}\label{sec.l.r.fib}
We show that left fibrations are coCartesian fibrations, and that right fibrations are Cartesian fibrations.
We characterize left/right fibrations in terms of exponentiable fibrations.

We first recall the notion of a left fibration and of a right fibration.
\begin{definition}\label{def.left.fib}
Let $\cE\xra{\pi} \cK$ be a functor between $\infty$-categories.
\begin{enumerate}
\item 
This functor $\pi$ is a \emph{left fibration} if, for each $\infty$-category $\cJ^{\tl} \to \cK$ over $\cK$, the restriction functor between $\infty$-categories of sections
\[
\Fun_{/\cK}(\cJ^{\tl},\cE) \longrightarrow  \Fun_{/\cK}(\ast,\cE)
\]
is an equivalence. 
The $\infty$-category of \emph{left fibrations (over $\cK$)} is the full $\infty$-subcategory
\[
\LFib_{\cK}~\subset~\Cat_{/\cK}
\]
consisting of the left fibrations.

\item 
This functor $\pi$ is a \emph{right fibration} if, for each $\infty$-category $\cJ^{\tr} \to \cK$ over $\cK$, the restriction functor between $\infty$-categories of sections
\[
\Fun_{/\cK}(\cJ^{\tr},\cE) \longrightarrow  \Fun_{/\cK}(\ast,\cE)
\]
is an equivalence.  
The $\infty$-category of \emph{left fibrations (over $\cK$)} is the full $\infty$-subcategory
\[
\RFib_{\cK}~\subset~\Cat_{/\cK}
\]
consisting of the right fibrations.

\end{enumerate}

\end{definition}

\begin{prop}\label{left.is.coCart}
Let $\cE\xra{\pi}\cK$ be a functor between $\infty$-categories.
\begin{enumerate}
\item 
The following conditions on $\pi$ are equivalent.  
\begin{enumerate}

\item $\pi$ is a left fibration.

\item $\pi$ is a conservative coCartesian fibration (in the sense of Definition~\ref{def.efib.cons}).

\item $\pi$ is a conservative locally coCartesian fibration.

\item For each morphism $c_1\to \cK$, the restriction functor $\Fun_{/\cK}(c_1,\cE)\to \cE_{|s}$ is an equivalence between $\infty$-categories.  

\item Each lift $c_1\to \cE$ of a morphism $c_1\to \cK$ is coCartesian with respect to the base change $\ev_s\colon \cE_{|c_1}\to c_1$.  

\item Every morphism $c_1\to \cE$ is $\pi$-coCartesian.

\end{enumerate}

\item 
The following conditions on $\pi$ are equivalent.  
\begin{enumerate}

\item $\pi$ is a right fibration.

\item $\pi$ is a conservative Cartesian fibration.

\item $\pi$ is a conservative locally Cartesian fibration.

\item For each morphism $c_1\to \cK$, the restriction functor $\ev_t\colon \Fun_{/\cK}(c_1,\cE)\to \cE_{|t}$ is an equivalence between $\infty$-categories.  

\item Each lift $c_1\to \cE$ of a morphism $c_1\to \cK$ is Cartesian with respect to the base change $\cE_{|c_1}\to c_1$.  

\item Every morphism $c_1\to \cE$ is $\pi$-Cartesian.

\end{enumerate}

\end{enumerate}

\end{prop}

\begin{proof}
The two assertions imply one another, as implemented by taking opposites.
We are therefore reduced to proving assertion~(1).

We use the logic: (c)$\implies$(d)$\implies$(c) and (a)$\implies$(f)$\implies$(e)$\implies$(c)$\implies$(b)$\implies$(a).

We now establish the implication (c)$\implies$(d).
Using Lemma~\ref{locallycocart}(c), the restriction functor $\ev_s\colon \Fun_{/\cK}(c_1,\cE) \to \cE_{|s}$ is a right adjoint.
Conservativity of the functor $\cE\xra{\pi}\cK$ implies both the domain and the codomain of $\ev_s$ are $\infty$-groupoids.
We conclude that this functor $\ev_s$ is an equivalence, as desired.

We now establish the implication (d)$\implies$(c).
Because equivalences are right adjoints, Lemma~\ref{locallycocart} gives that the functor $\cE\xra{\pi}\cK$ is locally coCartesian.  
Now let $c_1\to \ast \to \cK$ be a morphism that factors through the localization $c_1\to \ast$.
Identify the restriction functor $\ev_s$ as the functor $\Fun_{/\cK}(c_1,\cE) \simeq \Ar(\cE_{|\ast}) \xra{\ev_s}\cE_{|\ast}$.
In general, the functor $\ev_s\colon \Ar(\cE_{|\ast}) \to \cE_{|\ast}$ is a right adjoint, with left adjoint given selecting the equivalences in $\cE_{|\ast}$.  
The assumption that $\ev_s$ is an equivalence then implies the $\infty$-category $\cE_{|\ast}$ is an $\infty$-groupoid.
We conclude the desired conservativity of the functor $\cE\xra{\pi}\cK$.

We now establish the implication (a)$\implies$(f).
So suppose $\pi$ is a left fibration.
Let $c_1\xra{\lag e_s\to e_t\rag} \cE$ be a morphism.
Consider a solid diagram of $\infty$-categories
\[
\xymatrix{
\cJ  \ar[drrr]  \ar[ddr]   \ar@{-->}[dr]
&&&
\\
&
\cE^{e_t/}  \ar[rr]  \ar[d]
&&
\cE^{e_s/}  \ar[d]
\\
&
\cK^{\pi e_t/}  \ar[rr]
&&
\cK^{\pi e_s/}
}
\]
in which the inner square is the canonical one.  
We must show there is a unique filler.  
Denote the left cone $\ov{\cJ}:=\cJ^{\tl}$.
The above unique lifting property is equivalent to the existence of a unique filler in the diagram
\[
\xymatrix{
&&
\ast  \ar[dll]  \ar[d]^-{\lag e_s\rag}
\\
\cJ^{\tl} \ar[rr]  \ar[d]
&&
\cE  \ar[d]^-{\pi}
\\
\ov{\cJ}^{\tl}   \ar[rr]  \ar@{-->}[urr]
&&
\cK
\\
&&
\ast^{\tl}  \ar@(r,r)[uu]_-{\lag e_s \to e_t\rag}  \ar@(l,-)[ull]  \ar[u]^-{\lag \pi e_s \to \pi e_t\rag}
}
\]
Such a unique filler is implied by showing the top horizontal functor among $\infty$-categories of sections
\[
\xymatrix{
\Fun_{/\cK}(\ov{\cJ}^{\tl} , \cE)  \ar[rr]  \ar[dr]
&&
\Fun_{/\cK}(\cJ^{\tl} , \cE)   \ar[dl]
\\
&
\Fun_{/\cK}(\ast,\cE)
&
}
\]
is an equivalence.  
The assumption that $\pi$ is a left fibration gives that the two downward functors are equivalences.  
We conclude that the top horizontal functor is an equivalence, as desired.

The implication (f)$\implies$(e) is immediate from definitions.

We now establish the implication (e)$\implies$(c).
First, it is immediate that $\pi$ is a locally coCartesian fibration. 
From the Definition~\ref{def.coCart} of a $\pi$-coCartesian morphism, $\pi$-coCartesian morphisms that $\pi$ carries to equivalences are themselves equivalences.
Condition~(c) follows.

The implication (c)$\implies$(b) follows directly from Proposition~\ref{just.over.2}(f).  

We now establish the implication (b)$\implies$(a).
Let $\cJ^{\tl}\to \cK$ be a functor.
We must show that the restriction functor 
\begin{equation}\label{10}
\Fun_{/\cK}(\cJ^{\tl},\cE) \longrightarrow  \Fun_{/\cK}(\ast,\cE)
\end{equation}
is an equivalence between $\infty$-categories.  
The fully-faithfulness of the restricted Yoneda functor $\Cat\ra \psh(\bdelta)$ implies that the canonical functor $\colim\bigl(\bDelta_{/\cJ} \to \bDelta \to \Cat) \xra{\simeq} \cJ$ is an equivalence between $\infty$-categories (by, for instance, Lemma 3.5.9 of \cite{pkd}).  
Using that the functor $(-)^{\tl}\colon \Cat \to \Cat$ preserves colimit diagrams, we identify the functor~(\ref{10}) as the functor
\[
\limit \bigl((\bDelta_{/\cJ})^{\op} \to (\bDelta_{/\cK})^{\op} \to (\Cat_{/\cK})^{\op}\xra{\Fun_{/\cK}((-)^{\tl},\cE)} \Cat\bigr) \longrightarrow  \Fun_{/\cK}(\ast,\cE)~.
\]
Using that the $\infty$-groupoid completion $\sB (\bDelta_{/\cJ})\simeq \ast$ is terminal, this map is therefore an equivalence provided it is in the case that $\cJ \in \bDelta$ is an object in the simplex category.

So suppose $\cJ\in \bDelta$.  
Write $\cJ = \cI^{\tl}$ for $\cI$ a finite linearly ordered set; denote the minimal element of $\cJ$ as $\star$.
The functor $\cJ^{\tl}\to \cK$ determines the canonical square among $\infty$-categories of sections
\[
\xymatrix{
\Fun_{/\cK}(\cI^{\tl},\cE) \ar[d]_-{(\ref{10})_{\cI}}
&&
\Fun_{/\cK}(\cJ^{\tl},\cE)  \ar[d]  \ar[ll]  \ar[rr]^-{(\ref{10})_{\cJ}}
&&
\Fun_{/\cK}(\ast,\cE)
\\
\Fun_{/\cK}(\star,\cE)  
&&
\Fun_{/\cK}(\star^{\tl},\cE)  \ar[ll]  \ar[urr]_-{(\ref{10})_{\ast}}
&&
.
}
\]
The square is a pullback because the canonical functor from the pushout $\star^{\tl}\underset{\star} \amalg \cI^{\tl} \xra{\simeq} \cJ^{\tl}$.  
Consequently, we obtain that the functor~(\ref{10}) is an equivalence in the case of $\cJ$ provided~(\ref{10}) is an equivalence in the case of $[0]$ and the case of $\cI$, should $\cI$ not be empty.  
By induction on the number of elements in $\cJ$, we are therefore reduced to the case that $\cJ=[0]$.

So suppose $\cJ=\ast=[0]$.  
Using Lemma~\ref{locallycocart}(c), the assumed locally coCartesian condition on $\pi$ gives that the restriction~(\ref{10}), in this case that $\cJ=\ast$, is a right adjoint.
The assumed conservativity of the functor $\pi$ gives that, in fact, both the domain and the codomain of this functor are $\infty$-groupoids.  
We conclude that this functor is an equivalence, as desired.

\end{proof}

Lemmas~\ref{Cart.base.change} and~\ref{efib.cons.basechange} have this immediate result.
In the statement of this result we make implicit reference to the Cartesian symmetric monoidal structures on the $\infty$-categories $\CAT$ and $\SPACES$.
\begin{cor}\label{LFib.fctr}
Base change defines functors
\[
\LFib\colon \Cat^{\op} \longrightarrow \CAT
\qquad\text{ and }\qquad
\RFib\colon \Cat^{\op} \longrightarrow \CAT~,
\]
as well as 
\[
\LFib^\sim \colon \Cat^{\op} \xra{~\LFib~}\CAT \xra{~(-)^\sim~} \SPACES
\qquad\text{ and }\qquad
\RFib^\sim \colon \Cat^{\op} \xra{~\RFib~}\CAT \xra{~(-)^\sim~} \SPACES~.
\]
Fiber products over a common base defines lifts of these functors
\[
\LFib\colon \Cat^{\op} \longrightarrow \CAlg(\CAT)
\qquad\text{ and }\qquad
\RFib\colon \Cat^{\op} \longrightarrow \CAlg(\CAT)~,
\]
as well as 
\[
\LFib^\sim \colon \Cat^{\op}  \longrightarrow \CAlg(\SPACES)
\qquad\text{ and }\qquad
\RFib^\sim \colon \Cat^{\op} \longrightarrow \CAlg(\SPACES)~.
\]
The functor $\EFib^{\sf cons,\sim}\colon \Cat^{\op} \to \CAlg(\SPACES)$ is representable, in the sense of Theorem~\ref{flagged.thm}, by a full symmetric monoidal $\infty$-subcategory of the flagged $\infty$-category $\Corr$ of Definition~\ref{def.Corr}.

\end{cor}
\qed

The following construction of~\cite{HTT} is an $\infty$-categorical version of the Grothendieck construction.
\begin{construction}\label{def.left.un}
Let $\cK$ be an $\infty$-category.
The \emph{unstraightening} construction (for left fibrations) is the functor
\[
{\sf Un}\colon \Fun(\cK,\Spaces)\longrightarrow \LFib_{\cK}~,\qquad
(\cK\xra{F}\Spaces)\mapsto   \bigl((\Spaces^{\ast/})_{|\cK} \to \cK\bigr)~,
\]
whose values are given by base change of the left fibration $\Spaces^{\ast/}\to \Spaces$ along $F$.
The \emph{unstraightening} construction (for right fibrations) is the functor
\[
{\sf Un}\colon \Fun(\cK^{\op},\Spaces)\longrightarrow \RFib_{\cK}~,\qquad
(\cK^{\op}\xra{G}\Spaces)\mapsto   \bigl(({\Spaces^{\op}}_{/\ast})_{|\cK} \to \cK\bigr)~,
\]
whose values are given by base change of the right fibration ${\Spaces^{\op}}_{/\ast} \to \Spaces^{\op}$ along $F^{\op}$.

\end{construction}

\begin{example}
For $c_1\xra{\lag \cG_s\xra{f}\cG_t\rag} \Spaces$ a functor, its unstraightening (as a left fibration) is the cylinder construction: ${\sf Cyl}(f) \to c_1$.
For $c_1 \xra{\lag \cG_t \xla{f} \cG_s\rag} \Spaces^{\op}$ a functor, its unstraightening (as a right fibration) is the reverse cylinder construction: ${\sf Cylr}(f) \to c_1$.  

\end{example}

The next principal result from~\S2 of~\cite{HTT} states that the unstraightening construction for left/right fibrations is an equivalence.  
Another proof can also be found in~\cite{heuts.moredijk}.
(To state this result we use the Yoneda functor $\cK \xra{\lag \TwAr(\cK)\to \cK\rag} \RFib_\cK$; the proof of this result is tantamount to justifying calling this the Yoneda functor,  which is essentially the content of~\S2 of~\cite{HTT}.)
\begin{theorem}[Straightening-unstraightening for left/right fibrations]\label{st.un}
For each $\infty$-category $\cK$, the unstraightening constructions
\[
\Fun(\cK,\Spaces)\xra{~\sf Un~} \LFib_{\cK}
\qquad \text{ and }\qquad
\Fun(\cK^{\op},\Spaces)\xra{~\sf Un~} \RFib_{\cK}
\]
are each equivalences; their respective inverses are given as
\[
\LFib_{\cK} \longrightarrow  \Fun(\cK,\Spaces)
~,\qquad
(\cE\to \cK)\mapsto \Cat_{/\cK}(\cK^{\bullet/},\cE)
\]
and
\[
\RFib_{\cK} \longrightarrow  \Fun(\cK^{\op},\Spaces)
~,\qquad
(\cE\to \cK)\mapsto \Cat_{/\cK}(\cK_{/\bullet},\cE)~.
\]

\end{theorem}

\begin{cor}\label{LFib.rep}
The functor $\LFib^\sim\colon \Cat^{\op} \to \CAlg(\SPACES)$ is represented by the Cartesian symmetric monoidal $\infty$-category $\Spaces$; specifically, for each $\infty$-category $\cK$, the unstraightening construction implements a canonical equivalence between $\infty$-groupoids
\[
{\sf Un}\colon \Cat(\cK,\Spaces)~\simeq~ \LFib_{\cK}^\sim~.
\]
The functor $\RFib^\sim\colon \Cat^{\op} \to \CAlg(\SPACES)$ is represented by the coCartesian symmetric monoidal $\infty$-category $\Spaces^{\op}$; specifically, for each $\infty$-category $\cK$, the unstraightening construction implements a canonical equivalence between $\infty$-groupoids
\[
{\sf Un}\colon \Cat(\cK^{\op},\Spaces)~\simeq~ \RFib_{\cK}^\sim~.
\]

\end{cor}

The above results assemble to prove the assertions in Theorem~\ref{main.thm'} concerning left/right fibrations, 
\begin{theorem}\label{theorem-l.fib}
\begin{enumerate}
\item[~]

\item 
There is a fully-faithful functor between symmetric monoidal flagged $\infty$-categories:
\[
\Spaces~\hookrightarrow~\Corr~.
\]
For each $\infty$-category $\cK$, a functor $\cK\xra{\lag \cE\xra{\sf e.fib} \cK\rag} \Corr$ classifying the indicated exponentiable fibration, factors through $\Spaces\hookrightarrow \Corr$ if and only if the exponentiable fibration $\cE\to \cK$ is also a conservative locally coCartesian fibration.

\item 
There is a fully-faithful functor between symmetric monoidal flagged $\infty$-categories:
\[
\Spaces^{\op}~\hookrightarrow~\Corr~.
\]
For each $\infty$-category $\cK$, a functor $\cK\xra{\lag \cE\xra{\sf e.fib} \cK\rag} \Corr$ classifying the indicated exponentiable fibration, factors through $\Spaces^{\op}\hookrightarrow \Corr$ if and only if the exponentiable fibration $\cE\to \cK$ is also a conservative locally Cartesian fibration.  

\end{enumerate}
Furthermore, there is a canonical diagram of symmetric monoidal flagged $\infty$-categories,
\[
\xymatrix{
\Spaces^\sim  \ar[r]  \ar[d]
&
\Spaces  \ar[dr]  \ar[d]
&
\\
\Spaces^{\op}  \ar[r]  \ar[dr]
&
\Corr[\Spaces]   \ar[dr]
&
\Cat  \ar[d]
\\
&
\Cat^{\op}  \ar[r]
&
\Corr,
}
\]
in which each morphism is a monomorphism, and each square is a pullback.  

\end{theorem}

\subsection{Sub-left/right fibrations of coCartesian/Cartesian fibrations}
We explain that there is a maximal sub-left/right fibration of a coCartesian/Cartesian fibration, which we identify explicitly.

\begin{cor}\label{coCarts.morphs.compose}
Let $\cE \xra{\pi} \cK$ be a functor between $\infty$-categories.
\begin{enumerate}
\item 
Suppose $\pi$ is a coCartesian fibration.
Consider the subfunctor
\[
\Cat_{/\cK}^{\sf cCart}(-,\cE)\colon (\Cat_{/\cK})^{\op}\longrightarrow  \Spaces~,\qquad (\cJ\to \cK)\mapsto \Cat_{/\cK}^{\sf cCart}(\cJ,\cE)~\subset~\Cat_{/\cK}(\cJ,\cE)~,
\]
whose value on $(\cJ\to \cK)$ consists of those functors $\cJ \to \cE$ over $\cK$ that carry each morphism in $\cJ$ to a $\pi$-coCartesian morphism.  
This functor is represented by a left fibration over $\cK$.  

\item 
Suppose $\pi$ is a Cartesian fibration.
The subfunctor
\[
\Cat_{/\cK}^{\sf Cart}(-,\cE)\colon (\Cat_{/\cK})^{\op}\longrightarrow  \Spaces~,\qquad (\cJ\to \cK)\mapsto \Cat_{/\cK}^{\sf Cart}(\cJ,\cE)~\subset~\Cat_{/\cK}(\cJ,\cE)~,
\]
whose value on $(\cJ\to \cK)$ consists of those functor $\cJ \to \cE$ over $\cK$ that carry each morphism in $\cJ$ to a $\pi$-Cartesian morphism.  
This functor is represented by a right fibration over $\cK$.

\end{enumerate}

\end{cor}

\begin{proof}
The two assertions imply one another, as implemented by taking opposites.
We are therefore reduced to proving the first.

Consider the restricted presheaf
\[
\Cat_{/\cK}^{\sf cCart}( [\bullet] ,\cE) \colon
(\bDelta_{/\cK})^{\op} 
\longrightarrow
(\Cat_{/\cK})^{\op}
\xra{~\Cat_{/\cK}^{\sf cCart}(-,\cE)~}
\Spaces
~.
\]
We will show that $\Cat_{/\cK}^{\sf cCart}( [\bullet] ,\cE)$ is a univalent Segal space over $\cK$.  
This presheaf $\Cat_{/\cK}^{\sf cCart}( [\bullet] ,\cE)$ is a subpresheaf of $\Cat_{/\cK}([\bullet],\cE)$, which represents the univalent Segal space $\cE\to \cK$ over $\cK$.
Note that the monomorphism between spaces
\[
\Cat_{/\cK}^{\sf cCart}( [0] ,\cE)
~\hookrightarrow~
\Cat_{/\cK}( [0] ,\cE)
\]
is an equivalence.
Therefore, to verify that $\Cat_{/\cK}^{\sf cCart}( [\bullet] ,\cE)$ satisfies the complete and Segal conditions it is sufficient to verify that the solid diagram of spaces
\[
\xymatrix{
\Cat_{/\cK}^{\sf cCart}( [2] ,\cE) \ar[d]_-{\rm monomorphism} \ar@{-->}[rr]
&&
\Cat_{/\cK}^{\sf cCart}( \{0<2\} ,\cE)  \ar[d]^-{\rm monomorphism}
\\
\Cat_{/\cK}( [2] ,\cE) \ar[rr]^-{\underset{\cE}\circ}
&&
\Cat_{/\cK}(\{0<2\} ,\cE)
}
\]
admits a filler.
(Note that, because the vertical maps are monomorphisms among spaces, such a filler is unique if it exists.)
Proposition~\ref{just.over.2}(f), which states that the composition of two composible $\pi$-coCartesian morphisms is a $\pi$-coCartesian morphism, gives that such a filler exists.  
Denote the $\infty$-category over $\cK$ presented by this univalent Segal space over $\cK$ as
\[
\cE^{\sf cCart}
\longrightarrow
\cK
~.
\]
By construction, it is equipped with a canonical monomorphism $\cE^{\sf cCart}\hookrightarrow \cE$ over $\cK$.

The commutative diagram of $\infty$-categories
\[
\xymatrix{
(\bDelta_{/\cK})^{\op} \ar[rr]^-{\Cat_{/\cK}^{\sf cCart}( [\bullet] ,\cE)} \ar[d]
&&
\Spaces
\\
(\Cat_{/\cK})^{\op} \ar@{-->}[urr]_-{\Cat_{/\cK}(-,\cE^{\sf cCart})}
&&
}
\]
witnesses a left Kan extension.  
There results a canonical morphism between presheaves on $\Cat_{/\cK}$:
\begin{equation}\label{e101}
\Cat_{/\cK}(-,\cE^{\sf cCart})
\longrightarrow
\Cat_{/\cK}^{\sf cCart}( - ,\cE)
~,
\end{equation}
which we will show is an equivalence.  
This morphism~(\ref{e101}) fits into a commutative diagram of presheaves on $\Cat_{/\cK}$:
\[
\xymatrix{
\Cat_{/\cK}(-,\cE^{\sf cCart})
\ar[rr]^-{(\ref{e101})}  \ar[dr]_-{\rm monomorphism}
&&
\Cat_{/\cK}^{\sf cCart}(-,\cE) \ar[dl]^-{\rm monomorphism}
\\
&
\Cat_{/\cK}(-,\cE)
&
}
\]
in which the downward morphisms are monomorphisms.
It follows that~(\ref{e101}) is a monomorphism.  
It remains to verify, for each $\cJ\to \cK$, that the value of~(\ref{e101}) on $\cJ\to \cK$ is surjective on path-components.  
So let $\cJ \to \cE$ be a functor over $\cK$ that carries each morphism in $\cJ$ to a $\pi$-coCartesian morphism in $\cE$.  
By definition of $\cE^{\sf cCart}$, presented as a univalent Segal space, its space of morphisms consists of the $\pi$-coCartesian morphisms in $\cE$.
Therefore, $\cJ\to \cE$ factors through the monomorphism $\cE^{\sf cCart} \hookrightarrow \cE$, as desired.

By construction, an $\infty$-category over $\cK$ representing this functor has the property that each of its morphisms is a $\pi$-coCartesian morphism.  
We conclude from Proposition~\ref{left.is.coCart} that such a representing $\infty$-category over $\cK$ is a left fibration over $\cK$.
\end{proof}

\begin{notation}\label{def.sub.coCart}
For $\cE\xra{\pi}\cK$ a coCartesian fibration, its \emph{maximal sub-left fibration} 
\[
\cE^{\sf cCart} \longrightarrow \cK
\]
is a left fibration over $\cK$ representing the functor $\Cat_{/\cK}^{\sf cCart}(-,\cE)$ of Corollary~\ref{coCarts.morphs.compose}.
For $\cE\xra{\pi}\cK$ a Cartesian fibration, its \emph{maximal sub-right fibration} 
\[
\cE^{\sf Cart} \longrightarrow \cK
\]
is a right fibration over $\cK$ representing the functor $\Cat_{/\cK}^{\sf Cart}(-,\cE)$ of Corollary~\ref{coCarts.morphs.compose}.

\end{notation}

Inspecting the values of the presheaves in Corollary~\ref{coCarts.morphs.compose}
on objects $\cJ \to \cK$ in which $\cJ=c_0$ or $\cJ=c_1$ leads to the following observation.

\begin{observation}\label{t200}
\begin{enumerate}
\item
[~]

\item
For $\cE\xra{\pi}\cK$ a coCartesian fibration, its \emph{maximal sub-left fibration} is the $\infty$-subcategory
\[
\cE^{\sf cCart} 
~\hookrightarrow~
\cE
\]
consisting of every object in $\cE$; the morphisms of $\cE^{\sf cCart}$ consist of the $\pi$-coCartesian morphisms of $\cE$.

\item
For $\cE\xra{\pi}\cK$ a Cartesian fibration, its \emph{maximal sub-right fibration} is the $\infty$-subcategory
\[
\cE^{\sf Cart} 
~\hookrightarrow~
\cE
\]
consisting of every object in $\cE$; the morphisms of $\cE^{\sf Cart}$ consist of the $\pi$-Cartesian morphisms of $\cE$.

\end{enumerate}

\end{observation}

\begin{remark}
After Observation~\ref{t200}, Corollary~\ref{coCarts.morphs.compose} 
can be interpreted as the statement that the maximal sub-left/right fibrations exist: i.e., that morphisms in each are closed under composition in $\cE$.
	
\end{remark}

Corollary~\ref{coCarts.morphs.compose} has the following immediate consequence.  
In the statement of this result we reference the Cartesian symmetric monoidal structure of $\Spaces$ and of $\Cat$, and the coCartesian symmetric monoidal structure of $\Spaces^{\op}$ and of $\Cat^{\op}$.
\begin{cor}\label{sub.left}
\begin{enumerate}
\item[~]

\item 
The fully-faithful symmetric monoidal functor $\Spaces \hookrightarrow \Cat$ is a symmetric monoidal left adjoint; for each functor $\cK \xra{\lag \cE\xra{\pi} \cK\rag} \Cat$ classifying the indicated coCartesian fibration, postcomposition with the right adjoint is the functor $\cK \xra{\lag \cE^{\sf cCart}\to \cK\rag} \Spaces$ classifying the maximal left fibration of $\pi$.  

\item 
The fully-faithful symmetric monoidal functor $\Spaces^{\op} \hookrightarrow \Cat^{\op}$ is a symmetric monoidal left adjoint; for each functor $\cK \xra{\lag \cE\xra{\pi} \cK\rag} \Cat$ classifying the indicated coCartesian fibration, postcomposition with the right adjoint is the functor $\cK \xra{\lag \cE^{\sf cCart}\to \cK\rag} \Spaces$ classifying the maximal left fibration of $\pi$.  

\end{enumerate}

\end{cor}
\qed

\section{Left final/right-initial fibrations}\label{sec.final}
We introduce \emph{left-final} exponentiable fibrations, and \emph{right-initial} exponentiable fibrations, and show that they are classified by flagged $\infty$-subcategories $\LCorr \subset \Corr \supset \RCorr$.  
We show that they carry universal left/right fibrations.

\subsection{Left final and right-initial correspondences}
We introduce flagged $\infty$-subcategories
\[
\LCorr~\subset~\Corr~\supset~\RCorr
\]
of \emph{left-final} correspondences and \emph{right-initial} correspondences, and identify what they classify.

Here is the basic notion.
\begin{lemma}\label{equiv.final}
Let $\cE\to \cK$ be a functor.
\begin{enumerate}

\item  {\bf Left final:}
The following conditions on $\cE\to \cK$ are equivalent.
\begin{enumerate}
\item For each morphism $c_1\to \cK$, the fully-faithful functor $\cE_{|t} \hookrightarrow \cE_{|c_1}$ is final.
\item For each morphism $c_1\to \cK$, the Cartesian fibration $\Fun_{/\cK}(c_1,\cE) \xra{\ev_s} \cE_{|s}$ is final.
\item For each object $c_0 \xra{~\lag y\rag~} \cK$, the fully-faithful functor $\cE_{|y} \longrightarrow \cE_{/y}:=\cE\underset{\cK}\times \cK_{/y}$ is final.

\end{enumerate}

\item {\bf Right initial:}
The following conditions on $\cE\to \cK$ are equivalent.
\begin{enumerate}
\item For each morphism $c_1\to \cK$, the fully-faithful functor $\cE_{|s} \hookrightarrow \cE_{|c_1}$ is initial.
\item For each morphism $c_1\to \cK$, the coCartesian fibration $\Fun_{/\cK}(c_1,\cE) \xra{\ev_t} \cE_{|t}$ is initial.
\item For each object	 $c_0 \xra{~\lag x\rag~} \cK$, the fully-faithful functor $\cE_{|x} \longrightarrow \cE^{x/}:=\cE\underset{\cK}\times \cK^{x/}$ is initial.

\end{enumerate}

\end{enumerate}

\end{lemma}

\begin{proof}

Assertion~(1) is equivalent to assertion~(2), as implemented by replacing $\cE \to c_1$ by $\cE^{\op} \to c_1^{\op}\simeq c_1$.  
We are therefore reduced to proving assertion~(1).

We employ Quillen's Theorem A for each finality clause.  
\begin{enumerate}
\item[(a)]
Let $c_1\to \cK$ be a morphism.
The functor $\cE_{|t} \to \cE_{|c_1}$ is final if and only if, for each $e_s\in \cE_{|s}$, the classifying space of the $\infty$-undercategory $\sB (\cE_{|c_1})^{e_s/}\simeq \ast$ is terminal.

\item[(b)]
Let $c_1\ra \cK$ select a morphism.  Because the functor $\Fun_{/\cK}(c_1,\cE) \xra{\ev_s} \cE_{|s}$ is a Cartesian fibration, Lemma~\ref{firsthalf-locallycocart} gives that the functor $(\cE_{|c_1})^{e_s/}\simeq \Fun_{/\cK}(c_1,\cE)_{|e_s}\to \Fun_{/\cK}(c_1,\cE)^{e_s/}$ is a left adjoint.
In particular, the map between classifying spaces $\sB(\cE_{|c_1})^{e_s/} \to \sB\Fun_{/\cK}(c_1,\cE)^{e_s/}$ is an equivalence between spaces.

Now, the functor $\Fun_{/\cK}(c_1,\cE) \xra{\ev_s} \cE_{|s}$ is final if and only if, for each $e_s\in \cE_{|s}$, the classifying space of the $\infty$-undercategory $\sB \Fun_{/\cK}(c_1,\cE)^{e_s/}\simeq\sB (\cE_{|c_1})^{e_s/}\simeq  \ast$ is terminal.

\item[(c)]
Let $c_0\xra{\lag y \rag} \cK$ be an object.
An object in the $\infty$-overcategory $\cE_{/y}$ is the datum of a commutative diagram, $\gamma$, among $\infty$-categories:
\[
\xymatrix{
\ast  \ar[rr]^-{e_s}  \ar[d]_-{\lag s\rag }
&&
\cE  \ar[d]
\\
c_1  \ar[rr]
&&
\cK
\\
\ast  \ar[u]^-{\lag t\rag}  \ar[urr]_-{\lag y\rag }
&&
.
}
\]
Given such a diagram, $\gamma$, there is a canonical identification between $\infty$-undercategories: $(\cE_{/y})^{\gamma/}\simeq (\cE_{|c_1})^{e_s/}$.

Now, the functor $\cE_{|y} \to \cE_{/y}$ is final if and only if, for each diagram $\gamma$, the classifying space of the $\infty$-undercategory $\sB (\cE_{/y})^{\gamma/} \simeq \sB (\cE_{|c_1})^{e_s/}  \simeq \ast$ is terminal.

\end{enumerate}

\end{proof}

\begin{definition}\label{def.l.final}
A \emph{left-final/right-initial} fibration is an exponentiable fibration $\cE\ra \cK$ that satisfies any of the equivalent conditions of Lemma~\ref{equiv.final}(1)/(2).
A \emph{left-final/right-initial} correspondence is a left-final/right-initial fibration over the 1-cell.

\end{definition}

The next result offers a broad class of left-final exponentiable fibrations and of right-initial exponentiable fibrations.
\begin{prop}\label{cart.r.initial}
Let $\cE \to \cK$ be an exponentiable fibration.
If $\cE\to \cK$ is a coCartesian fibration, then it is left-final.
If $\cE\to \cK$ is a Cartesian fibration, then it is right-initial.

\end{prop}

\begin{proof}
Both of these statements follow from Lemma~\ref{firsthalf-locallycocart}, using Corollary~\ref{loc.final}.

\end{proof}

\begin{example}
Let $\cA \to \cB$ be a localization between $\infty$-categories.
Then the cylinder ${\sf cyl}(\cA \to \cB) \to c_1$ is both a left-final fibration and a right-initial fibration.

\end{example}

\begin{remark}
There is a potential for confusion of terminology: note that a left-final or right-initial fibration is not necessarily a final or initial functor. For instance, every functor $\cE\ra \ast$ to the terminal category is both a left-final and right-initial fibration, but $\cE \ra \ast$ is final or initial if and only if the classifying space $\sB\cE$ is contractible.\end{remark}

The following is a salient property of left-final, and right-initial, fibrations.

\begin{prop}\label{fib.colim}
Let 
\[
\xymatrix{
\cE  \ar[rr]^-F  \ar[d]_-{\pi}
&&
\cZ
\\
\cK 
}
\]
be a diagram of $\infty$-categories.  
\begin{enumerate}
\item 
Suppose the functor $\pi$ is a left-final fibration.
The left Kan extension $\pi_! F \colon \cK \to \cZ$ exists if and only if, for each $x\in \cK$, the colimit indexed by the fiber $\colim( \cE_{|x} \to \cE \xra{F} \cZ)$ exists in $\cZ$.
Furthermore, should this left Kan extension exist, its values are given as colimits indexed by these fibers:
\[
\cK \ni ~x~\mapsto~ \colim( \cE_{|x} \to \cE \xra{F} \cZ)~\in \cZ~.
\]

\item 
Suppose the functor $\pi$ is a right-initial fibration.
The right Kan extension $\pi_\ast F \colon \cK \to \cZ$ exists if and only if, for each $x\in \cK$, the limit over the fiber $\limit( \cE_{|x} \to \cE \xra{F} \cZ)$ exists in $\cZ$.
Furthermore, should this right Kan extension exist, its values are given as limits over the fibers:
\[
\cK \ni ~x~\mapsto~ \limit( \cE_{|x} \to \cE \xra{F} \cZ)~\in \cZ~.
\]

\end{enumerate}

\end{prop}

\begin{proof}
Assertion~(1) and assertion~(2) are equivalent, as implemented by replacing the given diagram by its opposite.
We are therefore reduced to proving assertion~(1).  

Formally, the left Kan extension $\pi_!F$ exists if and only if, for each $x\in \cK$, the colimit indexed by the $\infty$-overcategory $\colim( \cE_{/x} \to \cE \xra{F} \cZ)$ exists in $\cZ$;
furthermore, should this left Kan extension exist, its values are given as colimits:
\[
\cK \ni ~x~\mapsto~ \colim( \cE_{/x} \to \cE \xra{F} \cZ)~\in \cZ~.
\]
The result follows directly from the Definition~\ref{def.l.final}, using Lemma~\ref{equiv.final}(1c).

\end{proof}

After Example~\ref{cart.r.initial}, Proposition~\ref{fib.colim} restricts to the following familiar result.
\begin{cor}\label{cCart.fib.colim}
Proposition~\ref{fib.colim}(1)/(2) remains valid when \emph{``left-final fibration''/``right-initial fibration''} is replaced by \emph{``coCartesian fibration''/``Cartesian fibration''}.

\end{cor}
\qed

\begin{definition}\label{def.fcorr}
The symmetric monoidal $\infty$-category of \emph{left-final correspondences}, respectively \emph{right-initial correspondences} is the symmetric monoidal flagged $\infty$-subcategory
\[
\LCorr~\subset~ \Corr~\supset~\RCorr
\]
with the same underlying $\infty$-groupoid and those morphisms, which are exponentiable fibrations over $c_1$, that are left-final fibrations and that are right-initial fibrations, respectively.
\end{definition}

\begin{definition}\footnote{This definition was used in the published version of this paper, but is not used in this present corrected version.}
Let $\cE\ra\cK$ be functor between $\oo$-categories and let $\cW\hookrightarrow \cE$ be an $\oo$-subcategory containing
the maximal $\oo$-subgroupoid $\cE^\sim$. For each $\oo$-category $\cX\ra\cK$, the $\infty$-category
\[
\Fun^{\cW}_{/\cK}(\cX,\cE)
\]
is the $\oo$-subcategory of $\Fun_{/\cK}(\cX, \cE)$ consisting of the same objects and only those natural transformations by $\cW$. 
That is, a morphism $f:[1]\ra \Fun_{/\cK}(\cX, \cE)$ belongs to $\Fun^{\cW}_{/\cK}(\cX, \cE)$ if and only if for every $x\in \cX$ the morphism $f(x): [1] \ra \cE$ belongs to $\cW$.\footnote{This is equivalent to the intended definition of Notation 3.7.11 from \cite{pkd}, but unfortunately there is a typo therein.}
\end{definition}

\begin{lemma}\label{lemma.new.loc.cocart}
Consider a natural transformation $\eta$ of $\Cat$-valued functors
\[
\xymatrix{
\cX\ar@/^.5pc/[rr]^-f\ar@/_.5pc/[rr]_-g&\Downarrow&\Cat}~.
\]
If the functor $\eta_x: f(x) \ra g(x)$ is a localization for every $x\in \cX$, then the induced functor on colimits
\[
\underset{\cX}\colim \ f \longrightarrow \underset{\cX}\colim \ g
\]
is a localization.
\end{lemma}

\begin{proof}
Let $g^\sim\subset g$ be the subfunctor given by the composite $(-)^{\sim}\circ g:\cX\ra\Cat\ra \Spaces$, where the second functor assigns to an $\oo$-category its maximal $\oo$-subgroupoid. Consider the functor $\eta^{-1}(g^\sim)$ defined as the limit in $\Fun(\cX,\Cat)$ in the pullback square
\[
\xymatrix{
\eta^{-1}(g^\sim)\ar[r]\ar[d]&f\ar[d]^\eta\\
g^\sim\ar[r]&g~.}
\]
That is, for each $x\in\cX$, the value
\[
\eta^{-1}(g^\sim)(x):= g(x)^\sim\underset{g(x)}\times f(x)
\]
is the $\oo$-subcategory of $f(x)$ which maps to equivalences in $g(x)$. Let $\sB\eta^{-1}(g^\sim)$ be the composite functor $\sB\circ\eta^{-1}(g^\sim)$, where $\sB:\Cat\ra\Spaces$ is the classifying-space functor. The condition that $f(x)\ra g(x)$ is a localization implies that the value of the square
\[
\xymatrix{
\eta^{-1}(g^\sim)\ar[r]\ar[d]&f\ar[d]^\eta\\
\sB\eta^{-1}(g^\sim)\ar[r]&g~.}
\]
is a pushout for every $x\in\cX$. Since colimits in $\Fun(\cX,\Cat)$ are computed object-wise in $\cX$, this implies that the preceding square is a pushout in $\Fun(\cX,\Cat)$. Since the functor $\colim: \Fun(\cX,\Cat)\ra\Cat$ is a left adjoint, it preserves colimits, and the following square is again a pushout:
\[
\xymatrix{
\underset{\cX}\colim \ \eta^{-1}(g^\sim)\ar[r]\ar[d]&\underset{\cX}\colim \ f \ar[d]^\eta\\
\underset{\cX}\colim \ \sB\eta^{-1}(g^\sim)\ar[r]& \underset{\cX}\colim \ g~.}
\]
Note also the equivalence $\underset{\cX}\colim \ \sB\eta^{-1}(g^\sim)\simeq \sB\Bigl(\underset{\cX}\colim \ \eta^{-1}(g^\sim)\Bigr)$, since $\sB$ is itself a left adjoint. This implies the equivalence
\[
\underset{\cX}\colim \ g \simeq \Bigl(\underset{\cX}\colim \ f\Bigr)\Bigl[\bigl(\underset{\cX}\colim \ \eta^{-1}(g^\sim)\bigr)^{-1}\Bigr]
\]
that the $\oo$-category $\underset{\cX}\colim \ g$ is the localization of $\underset{\cX}\colim \ f$ with respect to 
$\underset{\cX}\colim \ \eta^{-1}(g^\sim)$.

%
\end{proof}

Our result can be reformulated as follows.
\begin{cor}\label{loc.fibs}
Consider a commutative diagram of $\oo$-categories
\[
\xymatrix{
\cE\ar[rr]_F\ar[dr]_\pi&&\cE'\ar[dl]^{\pi'}\\
&\cK}
\]
for which $F_{|x}\colon \cE_{|x}\ra\cE'_{|x}$ is a localization for every object $x\in \cK$. Then $F$ is a localization if either:
\begin{itemize}
\item $\pi$ and $\pi'$ are coCartesian fibrations, and $F$ carries $\pi$-coCartesian morphisms to $\pi'$-coCartesian morphisms; or

\item $\pi$ and $\pi'$ are Cartesian fibrations, and $F$ carries $\pi$-Cartesian morphisms to $\pi'$-Cartesian morphisms.  
\end{itemize}
\end{cor}

\begin{proof}
Since $\cE\ra \cE'$ is a localization if and only $\cE^{\op}\ra \cE'^{\op}$ is, it suffices to prove the case when $\pi$ and $\pi'$ are coCartesian fibrations.

Since $\cE\ra \cK$ is a coCartesian fibration, it is the unstraightening (see Construction~\ref{def.un}) of a functor ${\sf str}(\cE):\cK\ra \Cat$, from which $\cE$ can be recovered as a coend
\[
\cE\simeq \cK^{\bullet/} \underset{\cK} \otimes {\sf str}(\cE)
\]
and likewise for $\cE'\ra \cK$. Since the functor $F:\cE\ra \cE'$ preserves coCartesian morphisms over $\cK$, it is induced by a natural transformation of functors ${\sf str}(F): {\sf str}(\cE)\ra {\sf str}(\cE')$, so that there is a commutative diagram
\[
\xymatrix{
\cE\ar[d]_F\ar[r]^-\sim&\cK^{\bullet/} \underset{\cK} \otimes {\sf str}(\cE)\ar[d]^-{{\sf id}\ot {\sf str}(F)}\\
\cE'\ar[r]^-\sim&\cK^{\bullet/} \underset{\cK} \otimes {\sf str}(\cE')~.}
\]
Recall that a coend over $\cK$ is computed, by definition, as a colimit indexed by the twisted arrow $\oo$-category $\TwAr(\cK)$: That is, given contravariant and covariant functors $R:\cK^{\op}\ra\Cat$ and $L:\cK\ra\Cat$, then the coend is
\[
R\underset{\cK}\ot L \simeq \colim\Bigl(\TwAr(\cK)\xra{R\ot L}\Cat\Bigr)
\]
where $R\ot L$ is the functor given by the composite
\[
\TwAr(\cK)\ra \cK^{\op}\times\cK \xra{R \times L} \Cat\times\Cat \xra{\times}\Cat~.
\]
Consequently, observe that if $L\ra L'$ is a natural transformation in $\Fun(\cK,\Cat)$ such that $L(x)\ra L'(x)$ is a localization for every object $x$ in $\cK$, then likewise $R\ot L\ra R\ot L'$ is a natural transformation in $\Fun(\TwAr(\cK),\Cat)$ such that $(R\ot L)(\varphi)\ra (R\ot L')(\varphi)$ is a localization for every morphism $\varphi$ in $\cK$.
We can now apply Lemma~\ref{lemma.new.loc.cocart} to the case $\cX = \TwAr(\cK)$ with the natural transformation $\eta={\sf id}\ot{\sf str}(F)$. By the preceding observation, the natural transformation of $\Cat$-valued functors ${\sf id}\ot{\sf str}(F)$ satisfies the condition of being a localization on every object $\varphi$ of $\TwAr(\cK)$. We therefore obtain that the functor $F$ is a localization.
\end{proof}

The next result, Lemma~\ref{exponentiable.localization}, is a useful consequence of Lemma~\ref{exp-char}(6).  Its proof relies also on the preceding, Corollary~\ref{loc.fibs}.  

\begin{lemma}\label{exponentiable.localization}
For $\cE \ra [2]$ an exponentiable fibration, the restriction functor
\[
\Fun_{/[2]}\bigl([2],\cE\bigr)       
\longrightarrow
\Fun_{/[2]}\bigl(\{0<2\},\cE\bigr)
\]
is a localization.

\end{lemma}
\begin{proof}
Evaluation functors assemble as a diagram of $\infty$-categories:
\[
\xymatrix{
\Fun_{/[2]}\bigl([2],\cE\bigr)     \ar[rr]    \ar[dr]_-{\ev_{02}} \ar[ddr]_-{\ev_{2}}
&&
\Fun_{/[2]}\bigl(\{0<2\},\cE\bigr)   \ar[dl]^-{\ev_{02}}  \ar[ddl]^-{\ev_{2}}
\\
&
\cE_{|\{0\}}\times \cE_{|\{2\}}    \ar[d]_-{\pr}
&
\\
& 
\cE_{|\{2\}}  
&
.
}
\]
For each object $e_0\in \cE_{|\{0\}}$, the horizontal functor implements a functor between fibers over this object:
\[
\xymatrix{
\Fun_{/[2]}\bigl([2],\cE\bigr)_{|\{e_2\}}     \ar[rr]    \ar[dr]_-{\ev_{02}}
&&
\Fun_{/[2]}\bigl(\{0<2\},\cE\bigr)_{|\{e_2\}}   \ar[dl]^-{\ev_{02}} 
\\
&
\cE_{|\{0\}}
&
.
}
\]
Lemma~\ref{exp-char}(6) states that this horizontal functor witnesses a localization between fibers over $\cE_{|\{0\}}$.  
Using that this diagram is a morphism between Cartesian fibrations over $\cE_{|\{0\}}$, Corollary~\ref{loc.fibs} gives that the horizontal functor witnesses a localization.
This is to say that the horizontal functor in the first diagram of the proof 
witnesses a localization between fibers over $\cE_{|\{2\}}$.  
Using that the outer triangle, 
\[
\xymatrix{
\Fun_{/[2]}\bigl([2],\cE\bigr)     \ar[rr]    \ar[dr]_-{\ev_{2}}
&&
\Fun_{/[2]}\bigl(\{0<2\},\cE\bigr)  \ar[dl]^-{\ev_{2}} 
\\
&
\cE_{|\{2\}}
&
,
}
\]
is a morphism between coCartesian fibrations over $\cE_{|\{2\}}$, Corollary~\ref{loc.fibs} gives that the horizontal functor above witnesses a localization, as desired. 
\end{proof}

\begin{lemma}\label{fcorr-exists}
The symmetric monoidal flagged $\infty$-categories $\LCorr$ and $\RCorr$ exist.

\end{lemma}

\begin{proof}
The arguments for $\LCorr$ and $\RCorr$ are dual, so we only present the first.

We must verify that a composition of two morphisms in $\Corr$ that each belong to $\LCorr$ is again a morphism of $\LCorr$.
So let $\cE\to [2]$ be an exponentiable fibration.
Suppose both of the functors
$\ev_0\colon \Fun_{/[2]}(\{0<1\},\cE) \to \cE_{|0}$ and $\ev_1\colon \Fun_{/[2]}(\{1<2\},\cE) \to \cE_{|1}$
are final.
We must show that the functor $\ev_0\colon \Fun_{/[2]}(\{0<2\},\cE) \to \cE_{|0}$ is final. 
Our argument follows the canonical diagram of $\infty$-categories of partial sections and restriction functors among them:
\[
\xymatrix{
\Fun_{/[2]}(\{0<2\},\cE)    \ar[dd]
&&
\Fun_{/[2]}([2],\cE)      \ar[dr]      \ar[dl]  \ar[ll]
\\
&
\Fun_{/[2]}(\{0<1\},\cE)\ar[dr]_{\rm coCart}  \ar[dl]^{\rm final}   
&&
\Fun_{/[2]}(\{1<2\},\cE)   \ar[dr]    \ar[dl]^{\rm final}
\\
\cE_{|0}
&&
\cE_{|1}
&&
\cE_{|2}.
}
\]
In light of the left triangle in this diagram, the 2 out of 3 property for final functors (Lemma~\ref{2.o.3}) reduces finality of the functor $\Fun_{/[2]}(\{0<2\},\cE)  \to \cE_{|0}$ to finality of the functor $\Fun_{/[2]}([2],\cE)  \to \Fun_{/[2]}(\{0<2\},\cE)$ and finality of the composite functor $\Fun_{/[2]}([2],\cE) \to\Fun_{/[2]}(\{0<1\},\cE)  \to  \cE_{|0}$.  
Lemma~\ref{exponentiable.localization} states that the first of these functors is a localization; Proposition~\ref{localization.final} thus gives that this functor is final, as desired.
We now address finality of the composite functor.  
By assumption, the right factor in this composite is a final functor.
Because final functors compose (Lemma~\ref{2.o.3}), we are reduced to showing that the functor $\Fun_{/[2]}([2],\cE) \to\Fun_{/[2]}(\{0<1\},\cE)$ is final.
This is so because final functors are preserved by base change along a coCartesian fibration (Corollary~\ref{finalpullback}).
This concludes the proof that $\LCorr$ exists as a full $\oo$-subcategory of $\Corr$.

We now show that the symmetric monoidal structure on $\Corr$ restricts to one on $\LCorr$.
Note that the monomorphism $\LCorr \to \Corr$ is an equivalence on underlying $\infty$-groupoids.
Thus, we need only verify a factorization of the pairwise symmetric monoidal structure
\[
\xymatrix{
\LCorr \times \LCorr \ar@{-->}[rr]  \ar[d]
&&
\LCorr \ar[d]
\\
\Corr \times \Corr \ar[rr]^-{\otimes}
&&
\Corr.
}
\]
By definition, the vertical arrows are monomorphisms between Segal spaces; on $[0]$-points the vertical monomorphisms are equivalences between spaces.  
Therefore, such a factorization exists, and is unique, provided it does on spaces of $[1]$-points.  
Recall from~\S\ref{sym.stctr} that the symmetric monoidal structure on $\Corr$ is given on objects by products, and on $\cK$-points by fiber products over $\cK$.  
So the existence of the factorization is reduced to the following assertion.
\begin{itemize}
\item[~]
Let $\cE\to c_1$ and $\cE'\to c_1$ be left-final fibrations.  
The fiber product correspondence $\cE\underset{c_1}\times \cE' \to c_1$ is left-final.  
\end{itemize}
From the defining Lemma~\ref{equiv.final}(1a), we must then show that the canonical functor
\begin{equation}\label{e102}
\bigl(\cE\underset{c_1}\times \cE'\bigr)_{|t}
\longrightarrow
\cE\underset{c_1}\times \cE'
\end{equation}
is final.  
We use Quillen's Theorem~A (Theorem~\ref{thm.A}).
So let $(e,e')\in \cE\underset{c_1}\times \cE'$ be an object.
Because the diagonal functor $c_1\to c_1\times c_1$ is a monomorphism, so too is the canonical functor $\cE\underset{c_1}\times \cE' \to \cE\times \cE'$ a monomorphism.  
Because the above displayed functor is a monomorphism, it follows that the resulting functor between $\infty$-undercategories,
\[
\bigl(\cE\underset{c_1}\times \cE'\bigr)_{|t}
\underset{\cE\underset{c_1}\times \cE'}\times
\bigl(\cE\underset{c_1}\times \cE'\bigr)^{(e,e')/}
~ \longrightarrow ~ 
\bigl(\cE\underset{c_1}\times \cE'\bigr)_{|t}
\underset{\cE \times \cE'}\times
(\cE \times \cE')^{(e,e')/}
\]
is fully-faithful.
By inspection, this functor is surjective.
Therefore, the above functor is an equivalence between $\infty$-categories.  
So~(\ref{e102}) is a final if and only if the functor
\[
\bigl(\cE\underset{c_1}\times \cE'\bigr)_{|t}
~=~
\cE_{|t}\times \cE'_{|t}
\longrightarrow
\cE \times \cE'
\]
is final.  
Finality of this functor follows from the fact (Lemma~\ref{final.closure}) that the product of final functors is a final functor.
\end{proof}

\subsection{Universal left/right fibrations over $\LCorr$/$\RCorr$}
We define the relative classifying space of a functor, and show that this construction has some useful properties among left-final/right-initial fibrations.
We explain how this construction detemines universal left/right fibrations over $\LCorr$/$\RCorr$.

\begin{obs/def}\label{conserv.loc}
Let $\cK$ be an $\infty$-category.
Consider the full $\infty$-subcategory
\[
\CAT_{/^{\sf cons}\cK} ~\hookrightarrow~ \CAT_{/\cK}
\]
consisting of those functors $\cE\to \cK$ that are conservative.
This fully-faithful functor is a right adjoint.
The value of its left adjoint on an $\infty$-category $\cE\xra{\pi} \cK$ is its \emph{relative classifying space}, by which we mean the localization
\[
\sB^{\sf rel} \cE~:=~\sB^{\sf rel}_{\cK}\cE~:=~ \cE [ {\cE_{|\cK^\sim}}^{-1} ]  \longrightarrow \cK
\]
on the fibers of $\pi$.  
\end{obs/def}

\begin{lemma}\label{l.final.B.fib}
Let $\cE\xra{\pi} \cK$ be a functor between $\infty$-categories.
If $\pi$ is either a left-final fibration, or a right-initial fibration, then for each functor $[p]\to \cK$ from an object in $\bDelta$, the canonical map between spaces
\[
\sB \Bigl( \Fun_{/\cK}\bigl([p],\cE\bigr)\Bigr) \xra{~\simeq~} \Fun_{/\cK}\bigl([p],\sB^{\sf rel}\cE\bigr)
\]
is an equivalence.

\end{lemma}

\begin{proof}
Suppose $\pi$ is a left-final fibration; the case that $\pi$ is a right-initial fibration follows from this case, as implemented by taking opposites.

We show that the simplicial space over $\cK$,
\begin{equation}\label{66}
(\bDelta_{/\cK})^{\op} \xra{~\Fun_{/\cK}(-,\cE)~} \Cat \xra{~\sB~}\Spaces
\end{equation}
presents an $\infty$-category $\sB'\cE \to \cK$ over $\cK$, equipped with a functor from $\cE$ over $\cK$.  
The result follows upon checking that $\sB'\cE\to \cK$ is conservative, and upon checking that the functor $\cE \to \sB'\cE$ over $\cK$ demonstrates $\sB'\cE \to \cK$ as initial among conservative functors to $\cK$ under $\cE$.

We must show that the functor~(\ref{66}) satisfies the Segal condition over $\cK$, as well as the univalence condition over $\cK$.
Let $[p]\in \bDelta$ be an object with $p>0$ a positive integer.
Let $[p]\to \cK$ be a functor.
Consider the canonical diagram of $\infty$-categories of sections
\[
\xymatrix{
\Fun_{/\cK}\bigl([p],\cE\bigr) \ar[rr]  \ar[d]
&&
\Fun_{/\cK}\bigl(\{p-1<p\},\cE\bigr)  \ar[d]
\\
\Fun_{/\cK}\bigl(\{0<\dots<p-1\},\cE\bigr)  \ar[rr]
&&
\Fun_{/\cK}\bigl(\{p-1\},\cE\bigr).
}
\]
This square is a pullback because the canonical functor from the pushout $\{0<1\}\underset{\{1\}}\amalg \{1<\dots<p\} \xra{\simeq} [p]$ is an equivalence over $\cK$.
Since $\pi$ is a left-final fibration, therefore the right vertical functor is final.
The horizontal functors are coCartesian fibrations.  
It follows from Lemma~\ref{finalpullback} that the left vertical functor is final as well.  
Invoking Lemma~\ref{final.B.equiv}, applying classifying spaces determines the square of spaces
\begin{equation}\label{degenerate.square}
\xymatrix{
\sB\Fun_{/\cK}\bigl([p],\cE\bigr) \ar[rr]  \ar[d]_{\simeq}
&&
\sB\Fun_{/\cK}\bigl(\{p-1<p\},\cE\bigr)  \ar[d]^-{\simeq}
\\
\sB\Fun_{/\cK}\bigl(\{0<\dots<p-1\},\cE\bigr)  \ar[rr]
&&
\sB\Fun_{/\cK}\bigl(\{p-1\},\cE\bigr)
}
\end{equation}
in which the vertical maps are equivalences.  
Therefore, this square of spaces is a pullback.
We conclude that the functor~(\ref{66}) satisfies the Segal condition over $\cK$.

The functor~(\ref{66}) satisfies the univalence condition over $\cK$
because each fiber of $\pi$ is an $\infty$-category (and therefore presents a simplicial space satisfying the Segal and univalence conditions).
We conclude that the functor~(\ref{66}) presents an $\infty$-category $\sB'\cE\to \cK$ over $\cK$.

Now, by construction, the fiber of $\sB'\cE\to \cK$ over an object $\ast\xra{\lag x\rag} \cK$,
\[
(\sB'\cE)_{|x}~:=~(\sB \cE_{|x})~,
\]
is the classifying space of the fiber $\infty$-category.  
In particular, this fiber $\infty$-category is an $\infty$-groupoid.
We conclude that the functor $\sB'\cE\to \cK$ is conservative.

The canonical morphism $\Cat_{/\cK}([\bullet],\cE) \to \sB \Fun_{/\cK}([\bullet],\cE)$ between functors $(\bDelta_{/\cK})^{\op} \to \Spaces$ presents a functor $\cE \to \sB'\cE$ over $\cK$. 
Let $\cZ\to \cK$ be a conservative functor from an $\infty$-category.  
The main result of~\cite{rezk} gives that the canonical map from the $\infty$-category of functors over $\cK$
\[
\Fun_{/\cK}(\cE,\cZ)
\xra{~\simeq~}
\Map\Bigl(\Fun_{/\cK}\bigl([\bullet],\cE\bigr) , \Fun_{/\cK}\bigl([\bullet],\cZ\bigr)\Bigr)
\]
to the $\infty$-category of natural transformations between functors $(\bDelta_{/\cK})^{\op} \to \Cat$, is an equivalence.  
Because $\cZ\to \cK$ is assumed conservative, the functor $\Fun_{/\cK}\bigl([\bullet],\cZ\bigr)$ takes values in $\infty$-groupoids.
From the definition of $\sB$ as the $\infty$-groupoid completion, the canonical restriction functor
\begin{eqnarray}
\nonumber
\Fun_{/\cK}(\cE,\cZ)
&
\xra{~\simeq~}
&
\Map\Bigl(\Fun_{/\cK}\bigl([\bullet],\cE\bigr) , \Fun_{/\cK}\bigl([\bullet],\cZ\bigr)\Bigr)
\\
\nonumber
&
\longleftarrow
&
\Map\Bigl(\sB \Fun_{/\cK}\bigl([\bullet],\cE\bigr) , \Fun_{/\cK}\bigl([\bullet],\cZ\bigr)\Bigr)
\\
\nonumber
&
\xla{~\simeq~}
&
\Fun_{/\cK}(\sB'\cE,\cZ)
\end{eqnarray}
is then an equivalence between $\infty$-categories.  
We conclude that $\sB'\cE\to \cK$ is initial among conservative functors to $\cK$ under $\cE$, which is to give the desired canonical identificaiton
\[
\sB'\cE~\simeq~\sB^{\sf rel}\cE
\]
between $\infty$-categories over $\cK$.
\end{proof}

\begin{remark}

We expect that the conclusion of Lemma~\ref{l.final.B.fib} is valid for a weaker condition on $\pi$ than that of being a left-final fibration or a right-initial fibration.
Specifically, the only place where this condition on $\pi$ was used was for showing the square of spaces~(\ref{degenerate.square}) is a pullback; the left-final/right-initial condition is stronger than necessary for this to be the case.
See Question~\ref{q.cons.B}.

\end{remark}

\begin{remark}
This remark referred to a preceding result in the previous, published version of this paper, and it remains only so as not to change the numbering of results.
\end{remark}

Relative classifying space respects base change in the following sense.
\begin{cor}\label{fib.B.base.change}
For each pullback square
\[
\xymatrix{
\cE'  \ar[r]  \ar[d]
&
\cE  \ar[d]
\\
\cK'  \ar[r]
&
\cK
}
\]
among $\infty$-categories in which $\cE\ra\cK$ is either a left-final fibration or a right-initial fibration, the resulting square among $\infty$-categories
\[
\xymatrix{
\sB^{\sf rel}_{\cK'}\cE'  \ar[r]  \ar[d]
&
\sB^{\sf rel}_{\cK}\cE  \ar[d]
\\
\cK'  \ar[r]
&
\cK
}
\]
is a pullback.  

\end{cor}
\begin{proof}
Let $[p] \to \cK'$ be a functor from an object in $\bDelta$.
Because the given square is a pullback, the canonical functor between $\infty$-categories of sections,
\[
\Fun_{/\cK'}\bigl([p],\cE'\bigr)  \longrightarrow   \Fun_{/\cK}\bigl([p],\cE\bigr) ~,
\]
is an equivalence.
In particular, the resulting map between classifying spaces $\sB\Fun_{/\cK'}\bigl([p],\cE'\bigr)  \to   \sB\Fun_{/\cK}\bigl([p],\cE\bigr)$ is an equivalence.
Through Lemma~\ref{l.final.B.fib}, we conclude that the canonical functor between $\infty$-categories of sections of relative classifying spaces,
\[
\Fun_{/\cK'}\bigl([p],\sB^{\sf rel}_{\cK'}\cE'\bigr)  \longrightarrow   \Fun_{/\cK}\bigl([p],\sB^{\sf rel}_{\cK}\cE\bigr) ~,
\]
is an equivalence.
It follows that the desired square among $\infty$-categories is indeed a pullback.

\end{proof}

The next result shows that $\sB^{\sf rel}$ is computed \emph{fiberwise} in the present context.
\begin{cor}\label{indeed.fibers}
Let $\cE\xra{\pi}\cK$ be a functor between $\infty$-categories that is either left-final or right-initial.
For each point $\ast \xra{\lag x\rag} \cK$, the canonical square
\[
\xymatrix{
\sB\cE_{|x} \ar[r]  \ar[d]
&
\sB^{\sf rel}\cE  \ar[d]
\\
\ast \ar[r]^-{\lag x \rag}
&
\cK
}
\]
is a pullback.
\end{cor}
\qed

\begin{cor}\label{fib.B.exp}
Let $\cE\xra{\pi}\cK$ be a functor between $\infty$-categories.
\begin{enumerate}
\item
If the functor $\pi$ is a left-final fibration, the initial functor to a left fibration over $\cK$,
\[
\xymatrix{
\cE  \ar[rr]    \ar[dr]
&&
\cE^{\widehat{~}}_{\sf l.fib}   \ar[dl]
\\
&
\cK
&
}
\]
witnesses the relative classifying space:
\[
\sB^{\sf rel} \cE \underset{\rm over~\cK}{\xra{~\simeq~}}   \cE^{\widehat{~}}_{\sf l.fib}~.
\]
In particular, this relative classifying space $\sB^{\sf rel} \cE \to \cK$ is a left fibration.

\item
If the functor $\pi$ is a right-initial fibration, the initial functor to a right fibration over $\cK$,
\[
\xymatrix{
\cE  \ar[rr]    \ar[dr]
&&
\cE^{\widehat{~}}_{\sf r.fib}   \ar[dl]
\\
&
\cK
&
}
\]
witnesses the relative classifying space:
\[
\sB^{\sf rel} \cE \underset{\rm over~\cK}{\xra{~\simeq~}}   \cE^{\widehat{~}}_{\sf r.fib}~.
\]
In particular, this relative classifying space $\sB^{\sf rel} \cE \to \cK$ is a right fibration.

\end{enumerate}

\end{cor}

\begin{proof}
Assertion~(1) implies assertion~(2), as implemented by replacing $\cE\to \cK$ by its opposite.
We are therefore reduced to proving assertion~(1).

Let $\cZ\to \cK$ be a left fibration.
Because this functor is, in particular, conservative, the Definition~\ref{conserv.loc} of $\sB^{\sf rel}$ as a localization gives that the canonical map between spaces of morphisms over $\cK$,
\[
\Cat_{/\cK}(\sB^{\sf rel}\cE,\cZ)\longrightarrow\Cat_{/\cK}(\cE,\cZ)~,
\]
is an equivalence.
We are therefore reduced to showing that the functor $\sB^{\sf rel}\cE\to \cK$ is a left fibration.

Invoking Lemma~\ref{left.is.coCart}, we must show that, for each solid diagram of $\infty$-categories
\[
\xymatrix{
\ast \ar[rr]  \ar[d]_-{\lag s\rag}
&&
\sB^{\sf rel}\cE  \ar[d]
\\
c_1  \ar[rr]  \ar@{-->}[urr]
&&
\cK,
}
\]
the $\infty$-category of fillers is a contractible $\infty$-groupoid.
Equivalently, using that the functor $\sB^{\sf rel}\cE\to \cK$ is conservative, we must show that restriction functor between $\infty$-categories of sections
\begin{equation}\label{55}
\Fun_{/\cK}(c_1,\sB^{\sf rel}\cE) \longrightarrow \Fun_{/\cK}(\ast,\sB^{\sf rel}\cE)
\end{equation}
is an equivalence.  
Consider the functor between $\infty$-categories of sections,
\[
\Fun_{/\cK}(c_1,\cE) \longrightarrow \Fun_{/\cK}(\ast,\cE)~.
\]
Precisely because $\cE\to \cK$ is a left-final fibration, this functor is final. Consequently,
Lemma~\ref{final.B.equiv} gives that this functor induces an equivalence on classifying spaces.
Through Lemma~\ref{l.final.B.fib}, this implies the functor~(\ref{55}) is an equivalence, as desired.

\end{proof}

\begin{cor}\label{fib.B.fib.prods}
For 
\[
\xymatrix{
\cE''  \ar[r]  \ar[d]
&
\cE'  \ar[d]
\\
\cE  \ar[r]
&
\cK
}
\]
a pullback diagram of $\infty$-categories in which each functor is either a left-final or right-initial fibration, the square among $\infty$-categories
\[
\xymatrix{
\sB^{\sf rel}_\cK \cE''  \ar[r]  \ar[d]
&
\sB^{\sf rel}_\cK \cE'  \ar[d]
\\
\sB^{\sf rel}_\cK \cE  \ar[r]
&
\cK
}
\]
is a pullback.

\end{cor}

\begin{proof}
Let $[p] \to \cK'$ be a functor from an object in $\bDelta$.
Because the given square is a pullback, the canonical functor involving $\infty$-categories of sections,
\[
\Fun_{/\cK}\bigl([p],\cE''\bigr)  \longrightarrow   \Fun_{/\cK}\bigl([p],\cE\bigr)\times  \Fun_{/\cK}\bigl([p],\cE'\bigr) ~,
\]
is an equivalence.
Because the classifying space functor $\sB\colon \Cat \to \Spaces$ preserves products, the resulting map involving classifying spaces 
\[
\sB\Fun_{/\cK}\bigl([p],\cE''\bigr)  \longrightarrow   \sB \Fun_{/\cK}\bigl([p],\cE\bigr)\times \sB \Fun_{/\cK}\bigl([p],\cE'\bigr)
\]
is an equivalence.  
Through Lemma~\ref{l.final.B.fib}, we conclude that the canonical functor involving $\infty$-categories of sections of relative classifying spaces,
\[
\Fun_{/\cK'}\bigl([p],\sB^{\sf rel}_{\cK}\cE''\bigr)  \longrightarrow   \Fun_{/\cK}\bigl([p],\sB^{\sf rel}_{\cK}\cE\bigr) \times \Fun_{/\cK}\bigl([p],\sB^{\sf rel}_{\cK}\cE'\bigr) ~,
\]
is an equivalence.
It follows that the desired square among $\infty$-categories is indeed a pullback.

\end{proof}

The next result gives universal left/right fibrations over $\LCorr$/$\RCorr$.  
\begin{theorem}\label{fib.B.ladj}
Taking classifying spaces defines morphisms between symmetric monoidal flagged $\infty$-categories
\[
\sB\colon \LCorr \longrightarrow \Spaces
\qquad\text{ and }\qquad
\sB\colon \RCorr \longrightarrow \Spaces^{\op}~.
\]

\end{theorem}

\begin{proof}
Taking opposites, the assertion concerning $\LCorr$ implies that concerning $\RCorr$.
We are therefore reduced to showing the assertion concerning $\LCorr$.

Corollary~\ref{fib.B.exp} gives, for each $\infty$-category $\cK$, a filler in the diagram of $\infty$-categories.
\[
\xymatrix{
\LFib_\cK  \ar[d]
&&
\EFib^{\sf l.final}_\cK  \ar[d]   \ar@{-->}[ll]   
\\
\Cat_{/^{\sf cons}\cK} 
&&
\Cat_{/\cK}  \ar[ll]_-{\sB^{\sf rel}}
}
\]
in which the right vertical arrow is the fully-faithful embedding from the left-final fibrations over $\cK$.
Lemma~\ref{fib.B.base.change} gives that the top horizontal functor in this diagram determines a lift of the functor
\[
\bigl(\LFib\hookrightarrow \EFib^{\sf l.final}\bigr) \colon \Cat^{\op} \longrightarrow \Ar(\CAT) = \Fun(c_1,\CAT)
\]
through the $\infty$-subcategory $\Fun_{(\infty,2)}({\sf Adj},\CAT)\hookrightarrow \Ar(\CAT)$, which we now define.
This $\infty$-subcategory consists of those arrows $\cC\to \cD$ that are right adjoints, and those morphisms $(\cC\to \cD) \to (\cC'\to \cD')$ for which the resulting lax commutative diagram of $\infty$-categories
\[
\xymatrix{
\cC  \ar[d]
&&
\cD    \ar[d]   \ar[ll]^-{\rm left~adjoint}
\\
\cC'  
&
\Uparrow
&
\cD'   \ar@(d,d)[ll]^-{\rm left~adjoint}
}
\]
in fact commutes.
Restricting to left adjoints defines a functor $\Fun_{(\infty,2)}({\sf Adj},\CAT)\hookrightarrow \Ar(\CAT^{\op})^{\op} = \Fun(c_1^{\op},\CAT)\simeq \Ar(\CAT)$ over $\CAT\times \CAT$.
This in turn defines a functor
\[
\bigl(\LFib\xla{\sB^{\sf rel}} \EFib^{\sf l.final}\bigr) \colon \Cat^{\op} \longrightarrow \Fun(c_1^{\op},\CAT)\simeq \Ar(\CAT)~.
\]
Through Corollary~\ref{fib.B.fib.prods}, each value of this functor on an $\infty$-category $\cK$ preserves products, which are fiber products over $\cK$.
There results a lift of the above functor
\[
\bigl(\LFib\xla{\sB^{\sf rel}} \EFib^{\sf l.final}\bigr) \colon \Cat^{\op} \longrightarrow \Fun(c_1^{\op},\CAT)\simeq \Ar\bigl(\CAlg\CAT)\bigr)~.
\]
This establishes the result.  

\end{proof}

\begin{remark}
The utility of Theorem~\ref{fib.B.ladj} is that to construct a functor $\cK \to \Spaces$ it is enough to construct an exponentiable fibration $\cE\to \cK$ then check that this exponentiable fibration is left-final.  In light of Lemma~\ref{equiv.final}, this latter check can take place over morphisms in $\cK$ at a time.  
The advantage here is that the exponentiable fibration $\cE\to \cK$ can be weaker than a coCartesian fibration.  

\end{remark}

\begin{remark}\label{l.adj.efib.cons}
Premised on a 2-categorical enhancement of $\Corr$ (see Question~\ref{q.2.Corr}), we expect that Observation~\ref{conserv.loc} gives a sense in which each pair of symmetric monoidal functors $\LCorr \rightleftarrows \Spaces$ and $\RCorr \rightleftarrows \Spaces$ between flagged $\infty$-categories lifts as an adjunction.

\end{remark}

\subsection{Left/right fibration-replacement}
We describe, for each $\infty$-category $\cK$, left adjoints to the monomorphisms $\LFib_{\cK}\hookrightarrow \Cat_{/\cK}$ and left adjoints to the monomorphisms $\RFib_{\cK} \hookrightarrow  \Cat_{/\cK}$.

\begin{prop}\label{its.fib.B}
Let $\cK$ be an $\infty$-category.
The monomorphisms between Cartesian symmetric monoidal $\infty$-categories
\[
\LFib_{\cK}~ \hookrightarrow~\Cat_{/\cK}
\qquad\text{ and }\qquad
\RFib_{\cK}~\hookrightarrow~\Cat_{/\cK}
\]
are each symmetric monoidal right adjoints; their symmetric monoidal left adjoints respectively evaluate as the left/right fibration-replacement of the coCartesian/Cartesian-replacement:
\[
(-)^{\widehat{~}}_{\sf l.fib}\colon \Cat_{/\cK}\longrightarrow \LFib_{\cK}~,\qquad 
(\cE\to \cK)\mapsto \Bigl(\sB_{\cK}^{\sf rel}\bigl(\Ar(\cK)^{|\cE}\bigr) \xra{\ev_t}\cK\Bigr)
\]
and
\[
(-)^{\widehat{~}}_{\sf r.fib}\colon \Cat_{/\cK}\longrightarrow \RFib_{\cK}~,\qquad 
(\cE\to \cK)\mapsto \Bigl(\sB_{\cK}^{\sf rel}\bigl(\Ar(\cK)_{|\cE}\bigr) \xra{\ev_s}\cK\Bigr)~.
\]

\end{prop}

\begin{proof}
The assertions concerning left fibrations and coCartesian-replacement imply those concerning right fibrations and Cartesian-replacement, as implemented by taking opposites.
We are therefore reduced to proving the assertions concerning left fibrations and coCartesian-replacement.

The named symmetric monoidal monomorphism canonically factors $\LFib_{\cK} \hookrightarrow \cCart_{\cK} \hookrightarrow \Cat_{/\cK}$.
Theorem~\ref{cCart.adjoint} gives that the right factor in this composition is a symmetric monoidal right adjoint, and that its left adjoint is given by coCartesian-replacement.
Using that coCartesian fibrations are left-final fibrations (Proposition~\ref{cart.r.initial}), 
Corollary~\ref{fib.B.exp} identifies a left adjoint to the left factor in the above composition, which is given by left fibration-replacement.  
Lemma~\ref{fib.B.fib.prods} gives that this left adjoint preserves finite products, and is therefore a symmetric monoidal left adjoint.

\end{proof}

\begin{terminology}\label{def.fibration-replacement}
Let $\cE\xra{\pi}\cK$ be a functor between $\infty$-categories.
We refer to the values of the left adjoint $(\cE\to \cK)^{\widehat{~}}_{\sf l.fib}$ as the \emph{left fibration-replacement (of $\pi$)}.
We refer to the values of the left adjoint $(\cE\to \cK)^{\widehat{~}}_{\sf r.fib}$ as the \emph{right fibration-replacement (of $\pi$)}.

\end{terminology}

Proposition~\ref{its.fib.B} has the following immediate consequence.
In the statement of this result we reference the Cartesian symmetric monoidal structure of $\Spaces$ and of $\Cat$, and the coCartesian symmetric monoidal structure of $\Spaces^{\op}$ and of $\Cat^{\op}$.
\begin{cor}\label{l.fib.adjoint}
\begin{enumerate}
\item[~]

\item 
The fully-faithful symmetric monoidal functor $\Spaces \hookrightarrow \Cat$ is a symmetric monoidal right adjoint; for each functor $\cK \xra{\lag \cE\xra{\pi} \cK\rag} \Cat$ classifying the indicated coCartesian fibration, postcomposition with the left adjoint is the functor $\cK \xra{\lag \sB_{\cK}^{\sf rel}\bigl(\Ar(\cK)^{|\cE}\bigr) \to \cK\rag} \Spaces$ classifying the left fibration-replacement of the coCartesian-replacement of $\pi$.  

\item 
The fully-faithful symmetric monoidal functor $\Spaces^{\op} \hookrightarrow \Cat^{\op}$ is a symmetric monoidal right adjoint; for each functor $\cK^{\op} \xra{\lag \cE\xra{\pi} \cK\rag} \Cat$ classifying the indicated Cartesian fibration, postcomposition with the left adjoint is the functor $\cK^{\op} \xra{\lag \sB_{\cK}^{\sf rel}\bigl(\Ar(\cK)_{|\cE}\bigr) \to \cK\rag} \Spaces$ classifying the right fibration-replacement of the Cartesian-replacement of $\pi$.  

\end{enumerate}

\end{cor}
\qed

\section{$\Corr$ as bimodules and as bifibrations}
\subsection{Correspondences as bimodules and bifibrations}

We now show how a correspondence can also be presented as a bifibration or as a bimodule. By a \emph{bifibration}, we mean a pair $(\cA,\cB)$ of $\infty$-categories together with a functor $\cX \to \cA\times \cB$ that satisfies the following properties.
\begin{itemize}

\item The composite functor $\cX \to \cA \times \cB \xra{\pr} \cB$ is a coCartesian fibration.

\item For each $a\in \cA$, the composite functor $\cX_{|\{a\}\times \cB} \to \cX \to \cA \times \cB \xra{\pr} \cB$ is a left fibration.

\item The composite functor $\cX \to \cA\times \cB \xra{\pr} \cA$ is a Cartesian fibration.

\item For each $b\in \cB$, the composite functor $\cX_{|\cA \times \{b\}} \to \cX \to \cA\times \cB \xra{\pr} \cA$ is a right fibration.

\end{itemize}
By a \emph{bimodule}, we mean a pair $(\cA,\cB)$ of $\infty$-categories together with a functor $\cA^{\op}\times \cB\to \Spaces$. 
The following characterization overlaps with Proposition B.3.17 of \cite{HA}.
Our proof will rely on the following technical result (Lemma~\ref{t.p.join}) concerning \emph{parametrized joins} (Terminology~\ref{def.p.join}).

Consider the subdivision
\[
\sd(c_1)
~=~
\Bigl(
\{s\}
\longrightarrow
c_1
\longleftarrow
\{t\}
\Bigr)
~,
\]
which is a full $\infty$-subcategory of the $\infty$-overcategory $\Cat_{/c_1}$.
Consider the resulting (internal) restricted Yoneda functor
\[
\Gamma\colon
\Cat_{/c_1}
\longrightarrow
\Fun\bigl( \sd(c_1)^{\op} , \Cat \bigr)
~,
\]
\[
(\cE\to c_1)
\mapsto 
\Bigl( \sd(c_1)^{\op} \ni (x \to c_1)
\mapsto \Fun_{/c_1}\bigl( x , \cE \bigr) \Bigr)
~=~
\Bigl( 
\cE_{|s} 
\xla{~\ev_s~}
\Fun_{/c_1}\bigl(c_1 , \cE \bigr)
\xra{~\ev_t~}
\cE_{|t} 
\Bigr)
~.
\]
Left Kan extension along the fully-faithful functor $\sd(c_1)\hookrightarrow \Cat_{/c_1}$ defines a left adjoint to this functor $\Gamma$:
\[
\bigstar\colon
\Fun\bigl( \sd(c_1)^{\op} , \Cat\bigr)
\longrightarrow
\Cat_{/c_1}
~,
\]
whose values are given by the parametrized join construction:
\[
(\cA \la \cX \ra \cB)\mapsto  \cA\underset{\cX}\bigstar \cB~:=~ \cA \underset{\cX\times \{s\}} \amalg \cX\times c_1\underset{\cX\times \{t\}}\amalg \cB~.
\]
Lastly, consider the full $\infty$-subcategory
\[
{\sf BiFib}
~\subset~
\Fun\bigl( \sd(c_1)^{\op} , \Cat \bigr)
\]
consisting of the bifibrations.
For $\cA$ and $\cB$ $\infty$-categories, denote the pullback $\infty$-categories
\[
\xymatrix{
{\sf BiFib}(\cA,\cB) \ar[rr] \ar[d]
&&
{\sf BiFib} \ar[d]^-{(\ev_s,\ev_t)}
\\
\ast \ar[rr]^-{\lag \cA , \cB\rag }
&&
\Cat\times \Cat
}
\]
and
\[
\xymatrix{
(\Cat_{/c_1})^{|\cA}_{|\cB} \ar[rr] \ar[d]
&&
\Cat_{/c_1} \ar[d]^-{(s^\ast,t^\ast)}
\\
\ast \ar[rr]^-{\lag \cA , \cB\rag }
&&
\Cat\times \Cat~.
}
\]

\begin{lemma}\label{t.p.join}	
Let $\cA$ and $\cB$ be $\infty$-categories.
Parametrized join and sections restrict as mutual equivalences between $\infty$-categories:
\[
\cA \underset{-}\bigstar \cB
\colon 
{\sf BiFib}(\cA,\cB)
~{}~\simeq~{}~
(\Cat_{/c_1})^{|\cA}_{|\cB}
\colon 
\Gamma
~.
\]
In particular, for $\cA \la \cX \to \cB$ a span of $\infty$-categories, the canonical functor over $\cA\times \cB$,
\[
\cX
\longrightarrow
\Fun_{/c_1}\bigl(c_1,\cA\underset{\cX}\bigstar \cB\bigr)~,
\]
is an equivalence if $\cA \la \cX \ra \cB$ is a bifibration.

\end{lemma}

\begin{proof}
Our proof has two steps. In the first step, we show that the parametrized join functor and $\Gamma$ are adjoint to one another. In the second step, we construct a functor $\w{(-)}$, in $(\ref{e111})$, which we then exhibit as inverse to $\Gamma$. Since an inverse functor to $\Gamma$ is a left adjoint to $\Gamma$, and because left adjoints are unique when they exist, we then conclude that this inverse to $\Gamma$ agrees with $\bigstar$, which completes the proof.

Toward the first step, we first show $\Gamma$ factors through the full $\infty$-subcategory consisting of the bifibrations:
\[
\Gamma\colon
\Cat_{/c_1}
\longrightarrow
{\sf BiFib}
~\subset~
\Fun\bigl( \sd(c_1)^{\op} , \Cat \bigr)
~.
\]
Let $\cE\to c_1$ be an $\infty$-category over the 1-cell.
Recognize the value $\Gamma(\cE)$ as the pullback
\[
\xymatrix{
\Gamma(\cE)  \ar[rr]  \ar[d]
&&
\Ar(\cE)  \ar[d]
\\
\cE_{|s}\times \cE_{|t}  \ar[rr]
&&
\cE\times\cE.
}
\]
Tautologically, the span $\cE \xla{\ev_s}\Ar(\cE) \xra{\ev_t}\cE$ is a bifibration.
Because right/Cartesian/coCartesian/left fibrations are closed under pullbacks, it follows that $\cE_{|s} \la \Gamma(\cE)\to \cE_{|t}$ is a bifibration.

Next, for $\cA \la \cX \ra \cB$ a diagram among $\infty$-categories, there there is a canonical diagram among $\infty$-categories involving the parametrized join:
\begin{equation}\label{e200}
\xymatrix{
\cA \ar[rr] \ar[d]
&&
\cA \underset{\cX}\bigstar \cB \ar[d]
&&
\cB \ar[ll] \ar[d]
\\
\{s\} \ar[rr]
&&
c_1
&&
\{t\} . \ar[ll]
}
\end{equation}
Note that the functor $\{s\}\to c_1$ is a fully-faithful right fibration, and the functor $c_1\la \{t\}$ is a fully-faithful left fibration.  
So each of these functors is an exponentiable fibration.  
It follows from Theorem~\ref{theorem-l.fib} that each square in the above diagram is, in fact, a pullback square. Consequently, because the functors in the bottom horizontal cospan are fully-faithful and jointly surjective, the same is true for the top horizontal cospan.
It follows that the $(\bigstar,\Gamma)$-adjunction restricts as an adjunction:
\[
\cA \underset{-}\bigstar \cB
\colon 
{\sf BiFib}(\cA,\cB)
~\rightleftarrows~
(\Cat_{/c_1})^{|\cA}_{|\cB}
\colon 
\Gamma
~.
\]
The result is therefore established upon showing this $(\bigstar,\Gamma)$-adjunction is an equivalence.
We now turn to the second step, constructing an inverse functor to $\Gamma$.

Consider the resulting composite functor
\[
\bDelta_{/c_1}
~\hookrightarrow~
\Cat_{/c_1}
\xra{~\Gamma~}
\Fun\bigl( \sd(c_1)^{\op} , \Cat \bigr)
~,
\]
\[
\bigl( [p] \to c_1 \bigr)
\mapsto 
\bigl( [p]_{|s} \xla{~\pr~} [p]_{|s} \times [p]_{|t} \xra{~\pr~} [p]_{|t} \bigr)
~,
\]
whose values are as indicated.
There results a restricted Yoneda functor
\begin{equation}\label{e111}
\w{(-)}\colon
\Fun\bigl( \sd(c_1)^{\op} , \Cat \bigr)
\longrightarrow
\PShv(\bDelta_{/c_1} \bigr)
~,
\end{equation}
\[
(\cA\la \cX \ra \cB)
\mapsto 
\Bigl(
\bigl( [p] \to c_1 \bigr)
\mapsto 
\Map^{\sd(c_1)^{\op}}\bigl(~ ([p]_{|s} \la [p]_{|s} \times [p]_{|t} \ra [p]_{|t} )~,~ (\cA \la \cX \ra \cB) ~ \bigr)
\Bigr)
~.
\]
We will denote the value of the functor~(\ref{e111}) on $(\cA\la \cX \ra \cB)$ simply as $\w{\cX}$.

Consider the full $\infty$-subcategory ${\sf BiFib}\subset \Fun \bigl( \sd(c_1)^{\op} , \Cat\bigr)$ consisting of the bifibrations.  
We summarize the functors just constructed as the diagram of $\infty$-categories:
\begin{equation}\label{e121}
\xymatrix{
\sd(c_1) \ar[rr] \ar[d]_-{\sf f.f.}
&&
{\sf BiFib} \ar@{-->}[ddrr]^-{\w{(-)}}
&&
\\
\bDelta_{/c_1} \ar[urr] \ar[d]_-{\sf f.f.}
&&
\\
\Cat_{/c_1} \ar[uurr]_-{\Gamma} \ar[rrrr]^-{\rm restricted~Yoneda}_-{\sf f.f.}
&&
&&
\PShv(\bDelta_{/c_1})
;
}	
\end{equation}
here, the solid diagram commutes and those functors labeled with $\sf f.f.$ are fully-faithful.

Toward showing $\w{(-)}$ is inverse to $\Gamma$, we next show that $\w{(-)}$ takes values in $\Cat_{/c_1} \subset \PShv(\bDelta_{/c_1})$, consisting of those presheaves that satisfy the Segal and univalence conditions. 
By definition of $\w{\cX}$, its value on an object $([p]\to c_1)$ in $\bDelta_{/c_1}$ fits into a pullback square of spaces:
\begin{equation}\label{e112}
\xymatrix{
\w{\cX}([p]\to c_1)
\ar[rr] \ar[d]
&&
\Map\bigl( [p]_{|s}\times [p]_{|t} , \cX \bigr) \ar[d]
\\
\Map\bigl( [p]_{|s} , \cA \bigr)
\times
\Map\bigl( [p]_{|t} , \cB \bigr) \ar[rr]
&&
\Map\bigl( [p]_{|s}\times [p]_{|t} , \cA \bigr) 
\times
\Map\bigl( [p]_{|s}\times [p]_{|t} , \cB \bigr)
.
}
\end{equation}
The following three observations are direct consequences of this square being a pullback.  
\begin{itemize}
\item
If the functor $[p]\to c_1$ factors through $\{s\} \hookrightarrow c_1$, then the fiber $[p]_{|s} = [p]$ is entire and the fiber over $t$ is empty: $[p]_{|t} =\emptyset$.  
In this case, the pullback square of spaces~(\ref{e112}) identifies this value of $\w{\cX}$ as the space of $[p]$-points of $\cA$:
\[
\w{\cX}([p]\to c_1)
\xra{~\simeq~}
\Map([p],\cA)
~.
\]

\item
If the functor $[p]\to c_1$ factors through $\{t\} \hookrightarrow c_1$, then the fiber $[p]_{|s} =\emptyset$ is empty and the fiber over $t$ is entire: $[p]_{|t} = [p]$.
In this case, the pullback square of spaces~(\ref{e112}) identifies this value of $\w{\cX}$ as the space of $[p]$-points of $\cA$:
\[
\w{\cX}([p]\to c_1)
\xra{~\simeq~}
\Map([p],\cB)
~.
\]

\item
Suppose the functor $[p]\to c_1$ is surjective.
Consider the section $c_1\xra{\sigma} [p]$ defined by declaring that the composition $\{s\}\to c_1 \to [p]$ selects the maximum in the fiber $[p]_{|s}$ and that the composition $\{t\} \to c_1 \to [p]$ selects the minimum in the fiber $[p]_{|t}$.
Because $\cX\to \cA\times \cB$ is a bifibration, the canonical square of spaces,
\begin{equation}\label{e113}
\xymatrix{
\w{\cX}([p]\to c_1)
\ar[rr]^-{\sigma^\ast} \ar[d]
&&
\Map\bigl(\{s\}\times\{t\},\cX\bigr) = \cX^{\sim} \ar[d]
\\
\Map\bigl( [p]_{|s} , \cA \bigr) 
\times
\Map\bigl( [p]_{|t} , \cB \bigr) \ar[rr]^-{\sigma^\ast}
&&
\Map\bigl( \{s\} , \cA \bigr) 
\times
\Map\bigl( \{t\} , \cB \bigr)
=
\cA^\sim \times \cB^\sim
,
}
\end{equation}
is a pullback diagram.  
Two applications of Definition~\ref{def.left.fib} give that each solid diagram of spaces
\[
\xymatrix{
\{s\}\times \{t\}  \ar[rr]  \ar[d]_-{\sigma_{|s}\times \sigma_{|t}}
&&
\cX  \ar[d]
\\
[p]_{|s}\times [p]_{|t}  \ar[rr] \ar@{-->}[urr]^-{\exists !}
&&
\cA\times \cB
}
\]
admits a unique filler.

\end{itemize}

We now show that $\w{\cX}$ satisfies the Segal and univalence conditions.  
We first show $\w{\cX}$ satisfies the Segal condition.
So let $[p]\to c_1$ be an object in $\bDelta_{/c_1}$ with $p>0$.
We must show that the canonical diagram of spaces
\begin{equation}\label{e114}
\xymatrix{
\w{\cX}([p] \to c_1) \ar[rr] \ar[d]
&&
\w{\cX}(\{1<\dots<p\}\to c_1) \ar[d]
\\
\w{\cX}(\{0<1\} \to c_1) \ar[rr]
&&
\w{\cX}(\{1\}\to c_1) 
}
\end{equation}
is a pullback diagram. 
There are a few cases to examine separately.  
\begin{itemize}
\item
Suppose $[p]\to c_1$ factors through $\{s\}\hookrightarrow c_1$.
In this case, the preceding observations within this proof identify this square~(\ref{e114}) of spaces as the square
\[
\xymatrix{
\w{\cX}([p] \to c_1) \ar[rr] \ar[d]^-{\simeq}
&&
\ast \ar[d]
\\
\Map([p],\cA) \ar[rr]
&&
\ast
}
\]
which is a pullback square for trivial reasons. 

\item
Suppose $[p]\to c_1$ factors through $\{t\}\hookrightarrow c_1$.
In this case, the preceding observations within this proof identify this square~(\ref{e114}) of spaces as the square
\[
\xymatrix{
\w{\cX}([p] \to c_1) \ar[rr] \ar[d]^-{\simeq}
&&
\ast \ar[d]
\\
\Map([p],\cB) \ar[rr]
&&
\ast
}
\]
which is a pullback square for trivial reasons.

\item
Suppose $[p]\to c_1$ is surjective.
\begin{itemize}
\item
Suppose the composite functor $\{1\}\hookrightarrow [p]\to c_1$ factors through $\{s\}\to c_1$.
In this case, the composite functor $\{1<\dots<p\}\hookrightarrow [p]\to c_1$ is necessarily surjective, and the composite functor $\{0<1\} \hookrightarrow [p]\to c_1$ factors through $\{s\}\hookrightarrow c_1$.
The preceding observations within this proof serve to identify this square~(\ref{e114}) of spaces as the outer square in the diagram
\[
\xymatrix{
\bigl(\Map([p]_{|s},\cA)\times \Map([p]_{|t},\cB) \bigr) \underset{\cA^\sim \times \cB^\sim}\times \cX^\sim
\ar[rr] \ar[d]
&&
\Bigl(\Map(\{1<\dots<p\}_{|s},\cA)\times \Map(\{1<\dots<p\}_{|t},\cB) \Bigr) \underset{\cA^\sim \times \cB^\sim}\times \cX^\sim \ar[d]
\\
\Map([p]_{|s},\cA) \ar[rr] \ar[d]
&& 
\Map(\{1<\dots<p\}_{|s},\cA) \ar[d]
\\
\Map(\{0<1\},\cA) \ar[rr]
&&
\Map(\{1\},\cA)
.
}
\]
The lower square is a pullback because $\cA$ satisfies the Segal condition.
The upper square is pullback because, by assumptions, the canonical inclusion between fibers $\{1<\dots<p\}_{|t} \hookrightarrow [p]_{|t}$ is an equivalence. 
It follows that the outer square is a pullback, as desired.

\item
Suppose the composite functor $\{1\}\hookrightarrow [p]\to c_1$ factors through $\{t\}\to c_1$.
In this case, the composite functor $\{0<1\}\hookrightarrow [p]\to c_1$ is necessarily surjective, and the composite functor $\{1<\dots<p\} \hookrightarrow [p] \to c_1$ factors through $\{t\}$.
The preceding observations within this proof serve to identify this square~(\ref{e114}) of spaces as the square 
\[
\xymatrix{
\bigl(\Map([p]_{|s},\cA)\times \Map([p]_{|t},\cB) \bigr) \underset{\cA^\sim \times \cB^\sim}\times \cX^\sim
\ar[rr] \ar[d]
&&
\Map(\{1<\dots<p\}_{|t},\cB) \bigr) \ar[d]
\\
\cX^\sim
\ar[rr]
&&
\Map(\{1\},\cB)
.
}
\]
By assumptions, $\{0<1\}_{|s} = \{0\}$ and $\{1<\dots<p\}_{|t} \hookrightarrow \{1<\dots<p\}$ are equivalences.
It follows that this square is pullback, as desired.

\end{itemize}

\end{itemize}
This shows that $\w{\cX}$ satisfies the Segal condition.

Inspecting the definition of this Segal space, $\w{\cX}$, reveals that it fits into a pair of pullback squares of simplicial spaces:
\begin{equation}\label{e210}
\xymatrix{
\cA \ar[rr] \ar[d]
&&
\w{\cX} \ar[d]
&&
\cB \ar[ll] \ar[d]
\\
\{s\} \ar[rr]
&&
c_1
&&
\{t\} \ar[ll]
.
}
\end{equation}
Note that the $\infty$-category $c_1$ has the feature that each endomorphism in $c_1$ is an identity morphism.  
Therefore, each endomorphism in the Segal space $\w{\cX}$ over $c_1$ necessarily factors through the fibers, $\cA$ or $\cB$.
Therefore, the Segal space $\w{\cX}$ satisfies the univalence condition since both $\cA$ and $\cB$ present univalent Segal spaces.
We will denote the $\infty$-category over $c_1$ presenting this univalent Segal simplicial space over $c_1$ as $\w{\cX}$ again.
To summarize, we have constricted a pair of functors between $\infty$-categories:
\[
\w{(-)}\colon
{\sf BiFib}
~\rightleftarrows~
\Cat_{/c_1}
\colon \Gamma
~.
\]
As explained above, we seek to show these functors are mutual inverses to one another.
This is to construct invertible 2-cells making the following two diagrams commute:
\begin{equation}\label{e223}
\xymatrix{
{\sf BiFib} \ar[rr]^-{=} \ar[dr]_-{\w{(-)}}
&&
{\sf BiFib} 
\\
&
\Cat_{/c_1} \ar[ur]_-{\Gamma}
&
.
}
\end{equation}
and
\begin{equation}\label{e222}
\xymatrix{
\Cat_{/c_1} \ar[rr]^-{=} \ar[dr]_-{\Gamma}
&&
\Cat_{/c_1}
\\
&
{\sf BiFib} \ar[ur]_-{\w{(-)}}
&
.
}
\end{equation}

Now, in the pullback squares~(\ref{e210}), the top horizontal cospan consists of fully-faithful functors that are jointly surjective,because this is the case for the bottom horizontal cospan.
The pullback squares~(\ref{e113}) assemble as a pullback square of simplicial spaces:
\begin{equation}\label{e221}
\xymatrix{
\w{\cX}([\bullet]\times c_1\xra{\pr} c_1)
\ar[rr] \ar[d]
&&
\Map\bigl([\bullet]\times \{s\}\times\{t\},\cX\bigr)  \ar[d]
\\
\Map\bigl( [\bullet] , \cA \bigr) 
\times
\Map\bigl( [\bullet] , \cB \bigr) \ar[rr]^-{=}
&&
\Map\bigl( [\bullet] \times \{s\} , \cA \bigr) 
\times
\Map\bigl( [\bullet] \times \{t\} , \cB \bigr)
,
}
\end{equation}
Because the bottom morphism is an equivalence, so too is the top morphism.
This is to say that there is a canonical equivalence between $\infty$-categories over $\cA\times \cB$:
\[
\cX
~{}~\simeq~{}~
\Fun_{/c_1}(c_1,\w{\cX})
~.
\]
This identification is evidently functorial in the argument $\cX\in {\sf BiFib}(\cA,\cB)$.
This establishes the sought invertible 2-cell making the diagram~(\ref{e223}) commute.

It remains to construct an invertible 2-cell making the diagram of $\infty$-categories,~(\ref{e222}), commute.
Let $\cE\to c_1$ be an $\infty$-category over the 1-cell.
Consider the bifibration $\Gamma(\cE) \in {\sf BiFib}(\cE_{|s},\cE_{|t})$.
By definition of the value $\w{\Gamma(\cE)}$, which is an $\infty$-category over $c_1$, it is presented by the pullback simplicial space over $c_1$:
\[
\xymatrix{
\w{\Gamma(\cE)}([\bullet]\to c_1)
\ar[rr] \ar[d]
&&
\Map\bigl( [\bullet]_{|s}\times [\bullet]_{|t} , \Gamma(\cE) \bigr) \ar[d]
\\
\Map\bigl( [\bullet]_{|s} , \cE_{|s} \bigr)
\times
\Map\bigl( [\bullet]_{|t} , \cE_{|t} \bigr) \ar[rr]
&&
\Map\bigl( [\bullet]_{|s}\times [\bullet]_{|t} , \cE_{|s} \bigr) 
\times
\Map\bigl( [\bullet]_{|s}\times [\bullet]_{|t} , \cE_{|t} \bigr)
.
}
\]
By definition of the bifibration $\Gamma(\cE)$, the top right simplicial space over $c_1$ can be further identified, making a pullback diagram of simplicial spaces over $c_1$:
\begin{equation}\label{e224}
\xymatrix{
\w{\Gamma(\cE)}([\bullet]\to c_1)
\ar[rr] \ar[d]
&&
\Map_{/c_1}\bigl( [\bullet]_{|s}\times [\bullet]_{|t}\times c_1 , \cE \bigr) \ar[d]
\\
\Map_{/c_1}\bigl( [\bullet]_{|s} , \cE \bigr)
\times
\Map_{/c_1}\bigl( [\bullet]_{|t} , \cE \bigr) \ar[rr]
&&
\Map_{/c_1}\bigl( [\bullet]_{|s}\times [\bullet]_{|t} \times\{s\} , \cE \bigr) 
\times
\Map_{/c_1}\bigl( [\bullet]_{|s}\times [\bullet]_{|t} \times \{t\} , \cE \bigr)
.
}
\end{equation}
Now, consider the canonical morphism between cosimplicial spaces over $c_1$,
\[
[\bullet]_{|s}\underset{[\bullet]_{|s}\times [\bullet]_{|t}}\bigstar [\bullet]_{|t}
~:=~
[\bullet]_{|s}\underset{[\bullet]_{|s}\times [\bullet]_{|t}\times\{s\}}\coprod
[\bullet]_{|s}\times [\bullet]_{|t}\times c_1 
\underset{[\bullet]_{|s}\times [\bullet]_{|t}\times \{t\}}\coprod
[\bullet]_{|t}
~\xra{~\simeq~}~
[\bullet]
~.
\]
This morphism is an equivalence between cosimplicial spaces over $c_1$: for each object $\bigl([p]\to c_1\bigr)\in \bDelta_{/c_1}$, this equivalence witnesses $[p]$ as the join of the finite totally ordered sets $[p]_{|s}$ and $[p]_{|t}$.  
Through this double pushout description of the canonical cosimplicial $\infty$-category over $c_1$, we recognize the pullback term in the diagram~(\ref{e224}) of simplicial spaces over $c_1$ as
\[
\w{\Gamma(\cE)}([\bullet]\to c_1)
~{}~\simeq~{}~
\Map_{/c_1}\bigl([\bullet]_{|s}\underset{[\bullet]_{|s}\times [\bullet]_{|t}}\bigstar [\bullet]_{|t}
, \cE \bigr)
\xla{~\simeq~}
\Map_{/c_1}\bigl( [\bullet] , \cE \bigr)
~.
\]
We conclude an identification between $\infty$-categories over $c_1$:
\[
\cE
~{}~\simeq~{}~
\w{\Gamma(\cE)}
~.
\]
This identification is evidently functorial in the argument $\cE\in \Cat_{/c_1}$.
This establishes the sought invertible 2-cell making the diagram~(\ref{e223}) commute.
\end{proof}

\begin{lemma}\label{equiv.correspondences}
Let $\cE_s$ and $\cE_t$ be $\infty$-categories.
There are canonical equivalences between the following spaces:
\begin{enumerate}

\item the space of correspondences $\Corr(\cE_s,\cE_t)$, i.e., the maximal $\infty$-subgroupoid of
\[
\{\cE_s\}\underset{\Cat}\times \Cat_{/c_1} \underset{\Cat}\times \{\cE_t\}
\]

\item the maximal $\infty$-subgroupoid of $\Cat_{/\cE_s\times \cE_t}$ consisting of those $\cX \to \cE_s\times \cE_t$ that are bifibrations

\item the maximal $\infty$-subgroupoid of $(\cE_s, \cE_t)$-bimodules: $\Fun\bigl(\cE_s^{\op}\times \cE_t , \Spaces\bigr)$.

\end{enumerate}

\end{lemma}

\begin{proof}
Lemma~\ref{t.p.join} implements an equivalence between~(1) and~(2).
We now establish an equivalence between~(1) and~(3).

Twisted arrows organize as a functor in the diagram of $\infty$-categories
\[
\xymatrix{
\Cat_{/c_1}  \ar[rrr]^-{\TwAr_{|s}^{|t}}  \ar[d]^-{{|t}}_-{{|s}}
&&&
\Ar^{\sf l.fib}(\Cat)  \ar[d]^-{\rm target}
\\
\Cat  \times \Cat  \ar[rr]^-{{\op}\times \id}
&&
\Cat\times \Cat  \ar[r]^-{\times}
&
\Cat
}
\]
in which the top right term is the full $\infty$-subcategory of the arrow $\oo$-category $\Ar(\Cat) :=\Fun(c_1,\Cat)$ consisting of left fibrations.
In other words, twisted arrows organize as a functor
\begin{equation}\label{e199}
\xymatrix{
\Cat_{/c_1}  \ar[rr]^-{\TwAr_{|s}^{|t}}  \ar[dr]_-{(|s,|t)}
&&
\Ar^{\sf l.fib}(\Cat)_{|\Cat\times \Cat}  \ar[dl]
\\
&
\Cat  \times \Cat 
&
}
\end{equation}
to the base change of the above diagram.
Note that the equivalence between~(1) and~(3) follows upon showing the horizontal functor in this diagram is an equivalence between $\infty$-categories.  
Indeed,~(1) is the maximal $\infty$-subgroupoid of the fiber of the lefthand downward functor over $(\cE_s,\cE_t)$, while~(3) is the maximal $\infty$-subgroupoid of the fiber of the righthand downward functor over $(\cE_s,\cE_t)$.  
So we are reduced to showing the above horizontal functor is an equivalence between $\infty$-categories, which we do through the following approach.

{\bf Approach}:
We first observe that the functor~(\ref{e199}) implements an equivalence between spaces of morphisms, so long as the domain of such morphisms is an object in $\bDelta_{/c_1}$.
We then construct an adjoint functor $R$ (which we ultimately show is inverse to the functor~(\ref{e199})).  We do so by constructing, for each object $\cX$ in the codomain of~(\ref{e199}), a presheaf $R\cX$ on $\bDelta_{/c_1}$, then checking that this presheaf satisfies the Segal and univalence conditions. This verifies that $R\cX$ defines an object in the domain of~(\ref{e199}).
The assignment $\cX\mapsto R\cX$ will be evidently functorial, and so we lastly argue that $R$ is an inverse to~(\ref{e199}).

Let $\cE\to c_1$ be an object in $\Cat_{/c_1}$.
For each object $\cA\to c_1$ in $\Cat_{/c_1}$, the above diagram of $\infty$-categories determines the diagram of spaces of morphisms:
\begin{equation}\label{e210}
\xymatrix{
\Cat_{/c_1}(\cA,\cE)  \ar[rr]^-{\TwAr_{|s}^{|t}}  \ar[dr]_-{(|s,|t)}
&&
\Ar(\Cat)_{|\Cat\times \Cat}\Bigl( 
\TwAr(\cA)^{|\cA_{|s}^{\op}}_{|\cA_{|t}} 
, 
\TwAr(\cE)^{|\cE_{|s}^{\op}}_{|\cE_{|t}} 
\Bigr)  \ar[dl]^-{\rm forget}
\\
&
\Cat(\cA_{|s}^{\op},\cE_{|s}^{\op})  \times \Cat(\cA_{|t} , \cE_{|t})  
&
.
}
\end{equation}

Observation: We now argue that this horizontal map is an equivalence for each $\cA\to c_1$ in which $\cA\simeq [p]\in \bDelta$ is a finite non-empty linearly ordered set.
Note that the source-target functor 
\begin{equation}\label{e211}
\TwAr([p])^{|[p]_{|s}^{\op}}_{|[p]_{|t}}
 \xra{~(s,t)~}
[p]_{|s}^{\op} \times [p]_{|t}
\end{equation}
is an equivalence between $\infty$-categories.
Suppose the functor $[p]\to c_1$ factors through either $\{s\}$ or $\{t\}$.  
Then $\TwAr([p])^{|[p]_{|s}^{\op}}_{|[p]_{|t}} = \emptyset$ is empty.
In this case, both of the vertical maps in~(\ref{e210}) are equivalences, and the horizontal map in~(\ref{e210}) is an equivalence.  
So suppose $[p]\to c_1$ does not factor through either $\{s\}$ or $\{t\}$.
The $\infty$-category $\TwAr([p])^{|[p]_{|s}^{\op}}_{|[p]_{|t}}$ then has an initial object,
\begin{equation}\label{e212}
\ast 
\xra{~\lag (m-1<m)\rag~}
\TwAr([p])^{|[p]_{|s}^{\op}}_{|[p]_{|t}}
~,
\end{equation}
where $m={\sf Min}\{i\in [p]_{|t}\}$.  
Consider the diagram among spaces extending~(\ref{e210}):
\begin{equation}\label{e213}
\xymatrix{
\Cat_{/c_1}([p],\cE)  \ar[d]_-{\TwAr_{|s}^{|t}}  \ar[rr]^-{(|s,|t)}
&&
\Cat([p]_{|s}^{\op},\cE_{|s}^{\op})  \times \Cat([p]_{|t} , \cE_{|t}) \ar[d]^-=
\\
\Ar(\Cat)_{|\Cat\times \Cat}\Bigl( 
\TwAr([p])^{|[p]_{|s}^{\op}}_{|[p]_{|t}} 
, 
\TwAr(\cE)^{|\cE_{|s}^{\op}}_{|\cE_{|t}} 
\Bigr)  \ar[rr]^-{\rm forget}  \ar[d]_{(\ref{e212})^\ast}
&&
\Cat([p]_{|s}^{\op},\cE_{|s}^{\op})  \times \Cat([p]_{|t} , \cE_{|t})   \ar[d]^{(\ref{e212})^\ast}
\\
\Hom\Bigl( 
\ast
,  
\TwAr(\cE)^{|\cE_{|s}^{\op}}_{|\cE_{|t}}
\Bigr)
\simeq
\Cat_{/c_1}(c_1,\cE)        \ar[rr]^-{(s^\ast,t^\ast)}
&&
\cE_{|s}^\sim \times \cE_{|t}^\sim
.
}
\end{equation}
The source-target functor $\TwAr(\cE)^{|\cE_{|s}^{\op}}_{|\cE_{|t}}\to \cE_{|s}^{\op} \times \cE_{|t}$ is a left fibration.
Using Proposition~\ref{left.is.coCart}, the two expressions~(\ref{e211}) and~(\ref{e212}) grant that the bottom square in the above diagram of spaces is a pullback.  
On the other hand, the canonical functor 
\[
[p]_{|s}
\underset{\{m-1\}} \coprod 
\{m-1<m\}
\underset{\{m\}} \coprod
[p]_{|t}
\xra{~\simeq~}
[p]
\]
is an equivalence between $\infty$-categories over $c_1$.
It follows that the outer square of the diagram~(\ref{e213}) of spaces is a pullback.  
We conclude that the top square in~(\ref{e213}) is a pullback, from which it follows that the horizontal map~(\ref{e210}) is an equivalence in the case that $\cA\in \bDelta$, as desired.

We now use the above observation to construct an inverse, $R$, to the functor~(\ref{e199}).
Let $\cX \to \cE_s^{\op}\times \cE_t$ be a left fibration, which is an object in the codomain of~(\ref{e199}).
Consider the presheaf on $\bDelta_{/c_1}$:
\[
R\cX\colon 
(\bDelta_{/c_1})^{\op}
\longrightarrow
\Spaces
~,\qquad
([p]\to c_1)
\mapsto 
\Ar(\Cat)_{|\Cat\times \Cat}\Bigl( 
\TwAr([p])^{|[p]_{|s}^{\op}}_{|[p]_{|t}} 
, 
\TwAr(\cE)^{|\cE_{|s}^{\op}}_{|\cE_{|t}}  
\Bigr)
~.
\]
We argue that this presheaf satisfies the Segal and univalence conditions.  
To that end, let $([p]\to c_1)\in \bDelta_{/c_1}$ be an object.
Suppose $[p]\to c_1$ factors through $\{s\}$ or $\{t\}$.
Under this supposition, the expression~(\ref{e211}) reveals that the $\infty$-category $\TwAr([p])^{|[p]_{|s}^{\op}}_{|[p]_{|t}} = \emptyset$ is empty, and we have either $[p]_{|t} = \emptyset$ or $[p]_{|s} = \emptyset$.
Then tautologically, the canonical map
\begin{equation}\label{e214}
R\cX\bigl( [p]\to c_1 \bigr)
\xra{~\simeq~}
\Cat\bigl( [p]_{|s}^{\op} , \cE_s^{\op} \bigr)
\times
\Cat\bigl( [p]_{|t} , \cE_t \bigr)
\end{equation}
is an equivalence between spaces.  
Now, suppose $[p]\to c_1$ does not factor through $\{s\}$ or $\{t\}$.  
Recall the expressions~(\ref{e211}) and~(\ref{e212}).
Consider the diagram among spaces
\begin{equation}\label{e215}
\xymatrix{
R\cX\bigl( [p]\to c_1 \bigr) \ar[rr] \ar[d]
&&
\Ar(\Cat)_{|\Cat\times \Cat}\Bigl( 
\ast \la \ast \to \ast 
,  
\cE_{s}^{\op} \la \cX \to \cE_{t}
\Bigr)   
\simeq
\cX^{\sim}   \ar[d]
\\
\Cat([p]_{|s}^{\op},\cE_{s}^{\op})  \times \Cat([p]_{|t} , \cE_{t})   \ar[rr]^{(\ref{e212})^\ast}
&&
\cE_{s}^\sim \times \cE_{t}^\sim
.
}
\end{equation}
Proposition~\ref{left.is.coCart} grants that this square is a pullback.

Now, by direct examination, the expressions~(\ref{e214}) and~(\ref{e215}) immediately reveal that the presheaf $R\cX$ is satisfies the univalence condition, and they also reveal that $R\cX$ satisfies the Segal condition.  
This is to say that the presheaf $R\cX \in \PShv(\bDelta_{/c_1})$ presents an $\infty$-category over $c_1$.  
As so, the expression~(\ref{e214}) determines identifications between $\infty$-categories:
\begin{equation}\label{e216}
(R\cX)_{|s}
~\simeq~
\cE_s
\qquad\text{ and } \qquad
(R\cX)_{|t} 
~\simeq~
\cE_t
~.
\end{equation}

Being defined in terms of limit constructions, the assignment $\cX\mapsto R\cX$ defines a functor over $\Cat\times \Cat$:
\begin{equation}\label{e190}
\xymatrix{
\Ar^{\sf l.fib}(\Cat)_{|\Cat\times \Cat}   \ar[rr]^-{R}  \ar[dr]_-{\rm target}
&&
\Cat_{/c_1}   \ar[dl]^-{(|s,|t)}
\\
&
\Cat  \times \Cat 
&
.
}
\end{equation}
We now argue that $R$ is inverse to~(\ref{e199}).
We do this directly, by constructing natural transformations between identity functors and composites of $R$ and~(\ref{e199}).

By construction, there is a canonical natural transformation between endofunctors on $\Cat_{/c_1}$:
\begin{equation}\label{e198}
\eta\colon 
\id
\longrightarrow
R \circ \TwAr(-)^{|s}_{|t}
~,
\end{equation}
which evaluates on an $\infty$-category $\cE\to c_1$ over $c_1$ as the map between presheaves on $\bDelta_{/c_1}$ whose value on $([p]\to c_1)$ is the canonical map~(\ref{e210}) between spaces:
\[
\Cat_{/c_1}\bigl( [p] , \cE \bigr)
\xra{~(\ref{e210})~}
\Ar(\Cat)_{|\Cat\times \Cat}
\Bigl(
~
[p]_{|s}^{\op} \la [p]_{|s}^{\op}\times [p]_{|t} \to [p]_{|t}
~
,
~
\cE_{|s}^{\op} \la \TwAr(\cE)^{|\cE_{|s}^{\op}}_{|\cE_{|t}} \to \cE_{|t}
~
\Bigr)
.
\]
It was established above that this map~(\ref{e210}) is an equivalence, in this case.
The natural transformation $\eta$ is therefore an equivalence.

By construction, there is a canonical natural transformation between endofunctors of $\Ar(\Cat)_{|\Cat\times \Cat}$:
\begin{equation}\label{e197}
\epsilon\colon 
\TwAr(-)^{|s}_{|t} \circ R
\longrightarrow
\id
~,
\end{equation}
whose value on a left fibration $\cX\to \cE_s^{\op}\times \cE_t$ is the following functor between left fibrations over $\cE_s^{\op}\times \cE_t$:
\begin{equation}\label{e196}
\TwAr(R\cX)^{|\cE_s^{\op}}_{|\cE_t}
\longrightarrow
\cX
~.
\end{equation}
Unwinding the definition of $\TwAr$, the domain of this functor presents the presheaf on $\bDelta_{/\cE_s^{\op}\times \cE_t}$ whose value on $([p]\to\cE_s^{\op},[p]\to \cE_t)$ is the space of extensions
\[
\xymatrix{
[p]^{\op} \coprod [p] \ar[rr] \ar[d]
&&
\cE_s \coprod \cE_t \ar[rr]
&&
R\cX
\\
[p]^{\op} \underset{[p]^{\op}\times [p]} \bigstar [p] \ar@{-->}[urrrr]
&&
&&
.
}
\]
Unwinding the definition of $R$, this space of extensions is identical to the space of lifts
\[
\xymatrix{
&&
\cX \ar[d]
\\
[p]\times [p] \ar[rr] \ar@{-->}[urr]
&&
\cE_s^{\op} \times \cE_t
~.
}
\]
Through this unpacking of the domain of~(\ref{e196}), the functor~(\ref{e196}) presents the map between presheaves on $\bDelta_{/\cE_s^{\op}\times \cE_t}$ that is implemented by restriction along a diagonal functor:
\begin{equation}\label{e195}
{\sf diag}^{\ast}
\colon
\Cat_{/\cE_s^{\op}\times \cE_t}\bigl( [\bullet]\times [\bullet], \cX \bigr)
\longrightarrow
\Cat_{/\cE_s^{\op}\times \cE_t}\bigl( [\bullet], \cX \bigr)
~.
\end{equation}
Notice that, for each $[p]\in \bDelta$, the diagonal functor $[p]\xra{\sf diag}[p]\times [p]$ carries the initial object to the initial object.
By way of Proposition~\ref{left.is.coCart}, it follows that the map~(\ref{e195}) is an equivalence between spaces.  
We have established that the natural transformation $\epsilon$ is a natural equivalence.
We conclude that~(\ref{e199}) is an equivalence between $\infty$-categories, which concludes the equivalence between~(1) and~(3).
\end{proof}

\begin{example}\label{ex.identity.corr.more}
Let $\cC$ be an $\infty$-category. As we have seen in Example \ref{ex.identity.corr}, the \emph{identity correspondence} is the projection $\cC\times c_1\xra{\pr}c_1$.  
As a bifibration, it is $\Ar(\cC)\xra{\ev_{s,t}} \cC\times \cC$.
As a bimodule, it is $\cC^{\op}\times \cC \xra{\cC(-,-)} \Spaces$, the Yoneda functor.  
\end{example}

\begin{remark}
We can likewise describe morphisms of $\LCorr$ or $\RCorr$ in terms of bimodules.
For $\cC$ and $\cD$ two $\infty$-categories, a morphism in $\LCorr$ from $\cC$ to $\cD$ is a bimodule $\cC^{\op} \times \cD\to \Spaces$, for which, for each object $c\in \cC$, the colimit of the restriction $\colim\bigl(\{c\}\times \cD \to \cC^{\op}\times \cD \xra{M} \Spaces)\simeq \ast$ is terminal.

\end{remark}

\subsection{Composition of correspondences, as bimodules and as bifibrations}

Theorem~\ref{representablecorr} gave a composition rule for correspondences. In Lemma~\ref{composition}, we present this composition rule in terms of each of the three equivalent notions of a correspondence named in Lemma~\ref{equiv.correspondences}.

\begin{lemma}\label{composition}
Let $\cE_0$, $\cE_1$, and $\cE_2$ be $\infty$-categories.  
\begin{enumerate}
\item\label{comp.corr}
For $\cE_{01} \to \{0<1\}$ a correspondence from $\cE_0$ to $\cE_1$, and for $\cE_{12}\to \{1<2\}$ a correspondence from $\cE_1$ to $\cE_2$, the composite correspondence from $\cE_0$ to $\cE_2$ is the left vertical functor in the pullback of $\infty$-categories:
\[
\xymatrix{
\cE_{02}  \ar[rr]  \ar[d]
&&
\cE_{01}\underset{\cE_1}\amalg  \cE_{12}  \ar[d]
\\
\{0<2\}  \ar[rr]
&&
\{0<1\}\underset{\{1\}}\amalg \{1<2\} = [2].
}
\]

\item\label{comp.bifib}
For $\cX_{01} \to \cE_0\times \cE_1$ a bifibration over $(\cE_0,\cE_1)$, and for $\cX_{12}\to \cE_1\times \cE_2$ a bifibration over $(\cE_1,\cE_2)$, the composite bifibration over $(\cE_0,\cE_2)$ is the localization
\[
\cX_{012}[W^{-1}]
\]
in which $\cX_{012}$ is the pullback
\[
\xymatrix{
\cX_{012}  \ar[rr]  \ar[d]
&&
\cX_{12}  \ar[d]
\\
\cX_{01} \ar[rr]
&&
\cE_1
}
\]
and $W:=(\cX_{012})_{|\cE_0^\sim\times \cE_2^\sim}$ is the $\infty$-subcategory of $\cX_{012}$ consisting of those morphisms that the canonical functors $\cE_0\la \cX_{012} \to \cE_2$ carry to equivalences.

\item\label{comp.bimod}
For $P_{01}\colon \cE_0^{\op}\times \cE_1\to \Spaces$ a $(\cE_1,\cE_0)$-bimodule, and for $P_{12}\colon \cE_1^{\op}\times \cE_2\to \Spaces$ a $(\cE_2,\cE_1)$-bimodule, the composite $(\cE_2,\cE_0)$-bimodule is the coend over $\cE_1$,
\[
P_{02}~=~P_{12}\underset{\cE_1}\otimes P_{01} ~,
\]
defined as the left Kan extension
\[
\xymatrix{
\cE_0^{\op}\times \TwAr(\cE_1)^{\op} \times \cE_2 \ar[rr]^-{\id\times \ev_{s,t}\times \id}   \ar[d]^-{\pr}
&&
\cE_0^{\op}\times \cE_1\times \cE_1^{\op} \times \cE_2 \ar[rr]^-{P_{01}\times P_{12}}
&&
\Spaces \times \Spaces  \ar[d]^-{\times}
\\
\cE_0^{\op}   \times \cE_2  \ar[rrrr]_-{\sf LKan}^{P_{02}}
&&
&&
\Spaces   .
}
\]

\end{enumerate}

\end{lemma}

\begin{proof}
By definition, the composition rule for correspondences is given as the composite of the maps
\[
\corr(\{0<1\}) \underset{\corr(\{1\})}\times\corr(\{1<2\})\overset{\simeq}\longleftarrow \corr([2])\longrightarrow \corr(\{0<2\})~.
\]
The second map is restriction along $\{0<2\}\hookrightarrow [2]$. As shown in verifying the Segal condition for the functor $\corr_{|\bdelta}$ in Corollary \ref{corr.segal}, the first map is an equivalence with inverse given by sending a diagram
\[
\xymatrix{
\cE_{01}\ar[d]&\cE_1\ar[d]\ar[l]\ar[r]&\cE_{12}\ar[d]\\
\{0<1\}&\{1\}\ar[l]\ar[r]&\{1<2\}}
\]
to the exponentiable fibration $\cE_{01}\amalg_{\cE_1}\cE_{12}\ra [2]$. This verifies the composition rule in (\ref{comp.corr}).

We next verify the composition rule for bifibrations given in (\ref{comp.bifib}). By the proof of Lemma \ref{equiv.correspondences}, the bifibration $\cX_{01}\ra \cE_0\times\cE_1$ is equivalent to $\oo$-category of sections $\cX_{01}\simeq \Fun_{/\{0<1\}}\bigl(\{0<1\}, \cE_{01}\bigr)$ for $\cE_{01}:=\cE_0\underset{\cX_{01}}\bigstar\cE_1$; likewise $\cX_{12} \simeq \Fun_{\{1<2\}}\bigl( \{1<2\} , \cE_{12}\bigr)$ for $\cE_{02} := \cE_1 \underset{\cX_{12}} \bigstar \cE_2$.
Using the already established composition rule in (\ref{comp.corr}), the composition of bifibrations is given by the $\oo$-category of sections
\[
\Fun_{/\{0<2\}}(\{0<2\},\cE_{02})\longrightarrow \cE_0\times \cE_2~.
\]
By Lemma \ref{exponentiable.localization}, the restriction of sections
\[
\Fun_{/[2]}\Bigl([2],\cE_{01}\underset{\cE_1}\amalg\cE_{12}\Bigr)\longrightarrow \Fun_{/\{0<2\}}(\{0<2\},\cE_{02})
\]
is a localization. The source of this localization is equivalent to $\cX_{012}$, so the composition rule (\ref{comp.bifib}) follows.

We lastly verify the composition rule for bimodules as a coend given in (\ref{comp.bimod}). Let $P_{01}$ and $P_{12}$ be the bimodules associated to the exponentiable fibrations $\cE_{01}\ra \{0<1\}$ and $\cE_{12}\ra \{1<2\}$ as in Lemma \ref{equiv.correspondences}. That is, $P_{01}$ is the straightening of the left fibration $\TwAr(\cE_{01})^{|\cE_0^{\op}}_{|\cE_1}$ and likewise for $P_{12}$. From the universal property of left Kan extension, there is a natural functor of left fibrations over $\cE_0^{\op}\times\cE_2$
\[
{\sf Un}\bigl(P_{12}\underset{\cE_1}\ot P_{01}\bigr)\longrightarrow \TwAr\Bigl(\cE_{01}\underset{\cE_1}\amalg\cE_{12}\Bigr)^{|\cE_0^{\op}}_{|\cE_2}
\]
from the unstraightening of the coend of the bimodules to the left fibration over $\cE_0^{\op}\times\cE_2$ associated to exponentiable fibration given by composing $\cE_{01}$ and $\cE_{12}$ according to (\ref{comp.corr}). To check that this functor is an equivalence can be accomplished fiberwise over $\cE_0^{\op}\times\cE_2$. Since the projection ${\sf pr}: \cE_0^{\op}\times \TwAr(\cE_1)^{\op} \times \cE_2\ra \cE_0^{\op}\times\cE_2$ is a coCartesian fibration, left Kan extension along ${\sf pr}$ is computed fiberwise: so the space of maps in ${\sf Un}(P_{12}\underset{\cE_1}\ot P_{01})$ from $e_0\in \cE_0$ to $e_2\in \cE_2$ is equivalent to the coend $\cE_{01}(e_0,-)\ot_{\cE_1} \cE_{12}(-,e_2)$. The space of maps from $e_0\in \cE_0$ to $e_2\in \cE_2$ in $\cE_{01}\underset{\cE_1}\amalg\cE_{12}$ is computed by the identical expression by Lemma \ref{pushout}, so the result follows.
\end{proof}

\begin{remark}\label{corr-explicit}
The $\infty$-category $\Cat$ is to the $\infty$-category $\Corr$ as the category of rings is to the Morita category of rings.
This is justified by the following descriptions.
\begin{itemize}

\item {\bf Objects:}
An object in $\Corr$ is an $\infty$-category $\cA$, viewed as the exponentiable fibration $\cA\to \ast$ over the $0$-cell.

\item {\bf Morphisms:}
A morphism in $\Corr$ from $\cA$ to $\cB$ is a bimodule $\cA^{\op}\times \cB\xra{M} \Spaces$, viewed as the exponentiable fibration
\[
\cE_M  \longrightarrow c_1
\]
over the 1-cell whose fibers are identified $\cE_{|s}\simeq \cA$ and $\cE_{|t} \simeq \cB$, which is defined so that, for each $\infty$-category $\cJ\to c_1$ over the 1-cell, and each solid diagram of $\infty$-categories
\[
\xymatrix{
\cJ_{|s} \ar[rr]  \ar[d]
&&
\cA  \ar[d]
\\
\cJ  \ar@{-->}[rr]
&&
\cE_{M}
\\
\cJ_{|t} \ar[rr]  \ar[u]
&&
\cB,   \ar[u]
}
\]
the space of fillers is the limit
\[
\limit \bigl(\TwAr(\cJ)^{|\cJ_{|s}^{\op}}_{|\cJ_{|t}} \to \cA^{\op}\times \cB \xra{M} \Spaces  \bigr)~.
\]

\item {\bf Composition:}
For $\cA^{\op}\times \cB\xra{M}\Spaces$ a bimodule and for $\cB^{\op}\times \cC \xra{N} \Spaces$ another, their composition is the bimodule $\cA^{\op}\times \cC \xra{M\underset{\cB} \otimes N} \Spaces$ which is the coend along $\cB$, which is the left Kan extension in the diagram of $\infty$-categories:
\[
\xymatrix{
\cA^{\op}\times \TwAr(\cB)^{\op} \times \cC  \ar[rr]^-{(\ev_s,\ev_t)}  \ar[d]
&&
\cA^{\op}\times \cB\times \cB^{\op}\times \cC  \ar[r]^-{M\times N}
&
\Spaces \times \Spaces \ar[r]^-{\times}
&
\Spaces
\\
\cA^{\op}\times \cC   \ar[urrrr]_-{{\sf LKan} ~=~ M\underset{\cB}\otimes N}
&&&&
.
}
\]

\end{itemize}

\end{remark}

\begin{remark}
Through Remark~\ref{corr-explicit}, we believe that the $\infty$-category $\Corr[\Spaces]$ agrees with the $\infty$-category of spans of spaces.  
We do not require this result for our purposes; we encourage an interested reader to make this connection precise. 
See Question~\ref{q.spans}.

\end{remark}

\section{Finality and initiality}
We give a concise exposition of finality and initiality in $\infty$-category theory.
All material in~\S5.1 and~\S5.3 can be found in~\S4 of~\cite{HTT}, though we offer a different presentation here.
The material in~\S5.2 is more original, if not in statement then in technique.  

\subsection{Definitions and basic results}\label{subsec.final}
We use the following definition for finality of a functor.

\begin{definition}\label{def.final.initial}
Let $f\colon \cC \to \cD$ be a functor beteen $\infty$-categories.
This functor $f$ is \emph{final} if, for any functor $\cD \ra \cZ$ to an $\infty$-category, the canonical morphism in $\cZ$,
\[
\colim(\cC \xra{f} \cD\to \cZ) \longrightarrow \colim(\cD \ra \cZ)~,
\]
exists and is an equivalence whenever either colimit exists.
This functor $f$ is \emph{initial} if, for any functor $\cD \to \cZ$ to an $\infty$-category, the canonical morphism in $\cZ$,
\[
\limit(\cD\to \cZ) \longrightarrow \limit(\cC\xra{f}\cD \ra \cZ)~,
\]
exists and is an equivalence whenever either limit exists.
\end{definition}

\begin{remark}
In other literature, final functors are also known as cofinal, or right cofinal, functors. 
Initial functors are also known as coinitial, or left cofinal, functors.
\end{remark}

\begin{example}\label{final.objects}
Consider a functor $\ast \to \cC$ from a terminal $\infty$-category to an $\infty$-category.
This functor is final if and only if it selects a final object in $\cC$.
This functor is initial if and only if it selects an initial object in $\cC$.  

\end{example}

\begin{observation}\label{final.op}
A functor $\cC\to \cD$ between $\infty$-categories is final if and only if its opposite $\cC^{\op} \to \cD^{\op}$ is initial.

\end{observation}

\begin{lemma}\label{2.o.3}
Consider a commutative diagram of $\infty$-categories:
\[
\xymatrix{
\cC  \ar[rr]^-h  \ar[dr]_-f
&&
\cE   
\\
&
\cD  \ar[ur]_-{g}
&
.
}
\]
The following assertions are true.
\begin{enumerate}
\item If $f$ and $g$ are final, then $h$ is final.

\item If $f$ and $g$ are initial, then $h$ is initial.

\item If $f$ and $h$ are final, then $g$ is final.

\item If $f$ and $h$ are initial, then $g$ is initial.

\end{enumerate}

\end{lemma}

\begin{proof}
The assertions~(1) and~(2) are equivalent to one another, and assertions~(3) and~(4) are equivalent to one another, as implemented by replacing each $\infty$-category by its opposite.  
We are therefore reduced to proving assertions~(1) and~(3).

Let $\cE \to \cZ$ be a functor whose colimit exists.
The given triangle of $\infty$-categories determines the commutative diagram 
\[
\xymatrix{
\colim\bigl(\cC\xra{h} \cE \to \cZ\bigr)  \ar[rr]  \ar[dr]
&&
\colim\bigl(\cE \to \cZ\bigr)
\\
&
\colim\bigl(\cD\xra{g} \cE \to \cZ\bigr)  \ar[ur]
&
}
\]
in $\cZ$.
Assumption~(1) implies the diagonal legs of this triangle are equivalences.
We conclude that the top horizontal map is an equivalence.
This establishes assertion~(1).
Likewise, assumption~(3) implies the top horizontal morphism is an equivalence, and also that the downrightward morphism is an equivalence.
We conclude that the uprightward morphism is an equivalence.
This establishes assertion~(3).

\end{proof}

\begin{lemma}\label{final.closure}
Let $\cA\xra{f} \cB$ and $\cC \xra{g} \cD$ be functors.  
\begin{enumerate}
\item If both $f$ and $g$ are final, then their product $\cA\times \cC \xra{f\times g}\cB \times \cD$ is final.  

\item If both $f$ and $g$ are initial, then their product $\cA\times \cC \xra{f\times g}\cB \times \cD$ is initial. 

\end{enumerate}

\end{lemma}
\begin{proof}
After Observation~\ref{final.op}, assertion~(1) and assertion~(2) imply one another, as implemented by replacing each $\infty$-category by its opposite, using the canonical equivalence $(\cX\times \cY)^{\op} \simeq \cX^{\op} \times \cY^{\op}$.  
We are therefore reduced to proving assertion~(1).

Let $\cB\times \cD \xra{F} \cZ$ be a functor to an $\infty$-category.
In light of the factorization $f\times g\colon \cA \times \cC \xra{f\times \id} \cB\times \cC \xra{\id \times g} \cB\times \cD$, Lemma~\ref{2.o.3} gives that it is enough to show that $\cB\times \cC\xra{\id\times g} \cB\times \cD$ is final.  
Provided the colimit for $F$ exists, the comparison morphism in $\cZ$ is canonically identified as the morphism
\[
\colim(\cB\times \cC\xra{\id\times g} \cB\times \cD \xra{F} \cZ)
\simeq 
\colim( \cC\xra{\pr_!(F\circ \id\times g)} \cZ)  \longrightarrow
\colim( \cD\xra{\pr_!(F)} \cZ) 
\simeq
\colim(\cB\times \cD \xra{F} \cZ)~.
\]
where each $\pr_!$ is left Kan extension along projection off of $\cB$:
\[
\xymatrix{
\cB\times \cC  \ar[r]^-{\id\times g}  \ar[d]
&
\cB\times \cD  \ar[d]^-{\pr}
\\
\cC \ar[r]^-g
&
\cD.
}
\]
This diagram of $\infty$-categories is a pullback, and the vertical functors are coCartesian fibrations.
This implies that the canonical morphism $\pr_!(F\circ \id\times g)\to \pr_!(F)\circ g$ between functors $\cC\to \cZ$ is an equivalence.
The result follows from the assumption that $\cC\xra{g}\cD$ is final.  
\end{proof}

We use the next result frequently.
\begin{lemma}\label{just.spaces}
Let $\cC \xra{F} \cD$ be a functor between $\infty$-categories.
\begin{enumerate}
\item
This functor $F$ is final if and only if, for each 
presheaf $\cF\in \PShv(\cD)$, the canonical map between limit spaces
\[
\limit\bigl( \cD^{\op} \xra{\cF} \Spaces \bigr)
\longrightarrow
\limit\bigl( \cC^{\op} \xra{F^{\op}} \cD^{\op} \xra{\cF} \Spaces\bigr)
\]
is an equivalence.

\item
This functor $F$ is initial if and only if, for each
copresheaf $\cF\in {\sf cPShv}(\cD)$, the canonical map between colimit spaces
\[
\colim\bigl( \cC \xra{F} \cD \xra{\cF} \Spaces \bigr)
\longrightarrow
\colim\bigl( \cD\xra{\cF} \Spaces \bigr)
\]
is an equivalence.

\end{enumerate}

\end{lemma}

\begin{proof}
The two assertions imply one another, as implemented by replacing $\cC\to \cD$ by its opposite, $\cC^{\op}\to \cD^{\op}$.
We are therefore reduced to proving assertion~(2).

Initiality of $\cC\xra{F} \cD$ evidently implies that the canonical map between limit spaces is an equivalence.  
We now establish the converse.
Let $\cD\xra{A} \cZ$ be a functor to an $\infty$-category.
Consider the canonical morphism
\[
\limit\bigl( \cD \xra{A} \cZ \bigr)
\longrightarrow
\limit \bigl( \cC \xra{F} \cD \xra{A} \cZ \bigr)
\]
in $\cZ$.
Using that the Yoneda functor is fully-faithful, it is 
enough to show this morphism in $\cZ$ is carried by the 
Yoneda functor to an equivalence in the $\infty$-category 
${\sf cPShv}(\cZ)$ of copresheaves on $\cZ$.
We are therefore reduced to showing, for each $z\in \cZ$, that the map between spaces of morphisms in $\cZ$,
\[
\cZ\Bigl( z , \limit(\cD \xra{A} \cZ) \Bigr)
\longrightarrow
\cZ\Bigl( z , \limit (\cC \xra{F} \cD \xra{A} \cZ ) \Bigr)
\]
is an equivalence.
The universal property of limits in $\infty$-categories is exactly so that this map between morphism spaces is identified as the map between limits of morphism spaces:
\[
\limit \Bigl( \cD \xra{A} \cZ \xra{\cZ(z,-)} \Spaces \Bigr)
\longrightarrow
\limit \Bigl( \cC \xra{F} \cD \xra{A} \cZ \xra{\cZ(z,-)}
\Spaces \Bigr)
~.
\]
This map is an equivalence precisely by assumption on $F$.
\end{proof}

\begin{observation}\label{diagram.final}
A functor $f\colon \cC \to \cD$ is final if and only if the canonical lax commutative diagram
\[
\xymatrix{
\Fun(\cC,\Spaces)  \ar[dr]_-{\colim}   
&
\Downarrow
&
\Fun(\cD,\Spaces)  \ar@(u,u)[ll]_-{f^\ast}  \ar[dl]^-{\colim}
\\
&
\Spaces
&
}
\]
in fact commutes.

\end{observation}

\begin{lemma}\label{final.B.equiv}
If a functor $\cC\to \cD$ is either final or initial, then the resulting map between classifying spaces,
\[
\sB\cC\to \sB \cD~,
\]
is an equivalence.

\end{lemma}

\begin{proof}
Observation~\ref{final.op}, together with the canonical equivalence $\sB\cC \simeq \sB \cC^{\op}$ between classifying spaces, give that the assertion concerning finality implies that concerning initiality.
We are therefore reduced to proving the assertion concerning finality.
Consider the canonical map between classifying spaces
\[
\sB \cC\simeq \colim\bigl(\cC \to \cD \xra{!} \ast \xra{\lag \ast\rag} \Spaces  \bigr)  \longrightarrow \colim \bigl( \cD \xra{!} \ast \xra{\lag \ast\rag} \Spaces \bigr)  \simeq \sB \cD~.
\]
Finality of $\cC\to \cD$ directly gives that this map is an equivalence.

\end{proof}

See Joyal~\cite{joyal}, Lurie~\cite{HTT}, or Heuts--Moredijk~\cite{heuts.moredijk} for a treatment of Quillen's Theorem A (\cite{quillen}) at the generality of $\oo$-categories.
We provide a proof here, which relies on the straightening-unstraightening equivalence, established in~\S2 of~\cite{HTT} as well as in~\cite{heuts.moredijk}, recorded as Theorem~\ref{st.un} in this article.

\begin{remark}\label{st.un.maps}
Through the straightening-unstraightening equivalences (Theorem~\ref{st.un}), for each functor $\cX \xra{f} \cY$ between $\infty$-categories, the associated adjunction $f_!\colon \PShv(\cX) \rightleftarrows \PShv(\cY)\colon f^\ast$ is identified as the adjunction
\[
( \cE_{|\cX} \to \cX) \mapsfrom  (\cE\to \cY)
\qquad,~
\RFib_\cX~\rightleftarrows ~\RFib_\cY
~,\qquad 
(\cE\xra{\pi} \cX)\mapsto (  \cE \xra{f\circ \pi} \cY  )^{\widehat{~}}_{\sf r.fib}~,
\]
the rightward assignment given by replacing the composition by the first right fibration it maps to over the same base (see Proposition~\ref{its.fib.B}).

\end{remark}

\begin{theorem}[Quillen's Theorem A]\label{thm.A}
Let $f:\cC \ra \cD$ be a functor between $\infty$-categories.
This functor $f$ is final if and only if, for each object $d\in \cD$, the classifying space
\[
\sB\bigl(\cC^{d/}\bigr)~\simeq~\ast
\]
is terminal.
This functor $f$ is initial if and only if, for each object $d\in \cD$, the classifying space
\[
\sB\bigl(\cC_{/d}\bigr)~\simeq~\ast
\]
is terminal.
\end{theorem}

\begin{proof}
The two assertions imply each other by taking opposites.  
We are therefore reduced to proving the statement concerning finality.  

By definition, the functor $f$ is final precisely if the lax commutative diagram
\[
\xymatrix{
\Fun(\cC,\Spaces)  \ar[dr]_-{\sf colim}
&
\Downarrow
&
\Fun(\cD, \Spaces)  \ar[dl]^-{\sf colim}  \ar@(u,u)[ll]_-{f^\ast}    
\\
&
\Fun(\ast,\Spaces) = \Spaces
&
}
\]
in fact commutes. 
Through Remark~\ref{st.un.maps}, we identify this lax commutative diagram as 
\begin{equation}\label{lfib.quillen}
\xymatrix{
\LFib_\cC  \ar[dr]_-{\sB}
&
\Downarrow
&
\LFib_\cD  \ar[dl]^-{\sB}  \ar@(u,u)[ll]_-{f^\ast}    
\\
&
\LFib_\ast = \Spaces
&
.
}
\end{equation}
Consider the Yoneda functor,
\[
\cD^{\op}\xra{\TwAr(\cD)^{\op}} \LFib_\cD~,\qquad d\mapsto \bigl(\cD^{d/}\to \cD\bigr)~.
\]
This Yoneda functor strongly generates: that is, the diagram
\[
\xymatrix{
\cD^{\op}  \ar[rr]^-{\TwAr(\cD)^{\op}}  \ar[d]_-{\TwAr(\cD)^{\op}}
&&
\LFib_{\cD} 
\\
\LFib_{\cD}  \ar[urr]_-{\id}
}
\]
exhibits ${\sf id}$ as the left Kan extension.
In particular, each object in $\LFib_{\cD}$ is canonically a colimit of a diagram that factors through the fully-faithful functor $\cD^{\op}\xra{\TwAr(\cD)^{\op}} \LFib_{\cD}$.  
Consider the restriction of~(\ref{lfib.quillen}) along this Yoneda functor 
\[
\xymatrix{
\LFib_\cC  \ar[dr]_-{\sB}
&
\Downarrow
&
\cD^{\op}  \ar[dl]^-{\sB\cD^{\bullet/}\simeq\ast}  \ar@(u,u)[ll]_-{\cC^{\bullet/}}    
\\
&
\LFib_\ast = \Spaces
&
.
}
\]
Note that each arrow in~(\ref{lfib.quillen}) is a left adjoint, and therefore preserves colimits.  
Therefore, finality of $f$ is equivalent to, for each object $d\in \cD$, the canonical map between spaces
\[
\sB\bigl(\cC^{d/}\bigr) \longrightarrow  \ast
\]
being an equivalence, which concludes this proof.

\end{proof}

\begin{cor}\label{loc.final}
Let $\cC \to \cD$ be a functor between $\infty$-categories.
If this functor is a right adjoint, then this functor is final.
If this functor is a left adjoint, then this functor is initial.

\end{cor}

\begin{proof}
The functor $\cC\to \cD$ is a left adjoint if and only if, for each $d\in \cD$, the $\infty$-overcategory $\cC_{/d}$ has a final object.
The first statement follows directly from Quillen's Theorem~A, because the classifying space of an $\infty$-category with a final object is terminal.
Likewise, the functor $\cC\to \cD$ is a right adjoint if and only if, for each $d\in \cD$, the $\infty$-undercategory $\cC^{d/}$ has an initial object.  
The second statement follows directly from Quillen's Theorem~A, because the classifying space of an $\infty$-category with an initial object is terminal.

\end{proof}

\subsection{Auxiliary finality results}
We finish this section with several useful finality/initiality properties of functors.  We expect that these results are known to experts. However, with the exception of Quillen's Theorem B and Corollary~\ref{finalpullback} (see the citations discussion just before we state this result, Theorem~\ref{thm.B}), to our knowledge they do not appear in prior literature.

In what follows, each assertion concerning finality has an evident version concerning initiality. These assertions for initiality are implied by taking opposites.

\begin{prop}\label{localization.final}
A localization $f:\cC \ra \cD$ between $\oo$-categories is both final and initial.
\end{prop}

\begin{proof}
Let $f\colon \cC\to \cD$ be a localization.
Then the functor between opposites $\cC^{\op}\ra \cD^{\op}$ is also a localization.
So it is sufficient to show that $f$ is final.  

The commutative diagram of $\infty$-categories
\[
\xymatrix{
\cC \ar[dr]  \ar[rr]^-{f}
&
&
\cD  \ar[dl]
\\
&
\ast
&
}
\]
determines the commutative diagram of $\infty$-categories
\[
\xymatrix{
\Fun(\cC,\Spaces)  \ar[dr]_-{\sf colim}  \ar[rr]^-{f_!}
&
&
\Fun(\cD,\Spaces)  \ar[dl]^-{\sf colim}  
\\
&
\PShv(\ast) = \Spaces.
&
}
\]
Because $f$ is a localization, the right adjoint $f^\ast$ to $f_!$ in the above diagram is fully-faithful.  
Therefore the unit ${\sf id} \to f^\ast f_!$ of the $(f_!,f^\ast)$-adjunction is an equivalence.  
It follows that the identity 2-cell ${\sf id}_{\sf colim}$ factors as a composition of an invertible 2-cell and an, a priori, non-invertible 2-cell: 
\[
\xymatrix{
\Fun(\cD,\Spaces)  \ar[drr]_-{\sf colim}
&
&
\Fun(\cC,\Spaces)  \ar[d]_-{\sf colim}    \ar@(u,u)[ll]_-{f_!}
&
\Downarrow
&
\Fun(\cD,\Spaces)  \ar[dll]^-{\sf colim}  
\ar@(u,u)[ll]_-{f^\ast}    
\\
&&
\PShv(\ast) = \Spaces.
&&
}
\]
By the 2-out-of-3 property for equivalences, we conclude that the 2-cell in this diagram is, in fact, an equivalence, which is the assertion of finality of $f$.  
\end{proof}

\begin{prop}\label{final.cocart}
Consider a pullback diagram of $\oo$-categories
\[
\xymatrix{
\cE_0\ar[d]_{\ov{p}}\ar[r]^{\ov{f}}&\cE\ar[d]_p^{\rm coCart.}\\
\cC_0\ar[r]^f&\cC}
\]
in which $p$ is a coCartesian fibration, as indicated.
For any cocomplete $\oo$-category $\cZ$, the a priori lax commutative diagram of $\infty$-categories
\[
\xymatrix{
\Fun(\cE_0,\cZ)  \ar[d]_-{\ov{p}_!}
&
\Downarrow
&
\Fun(\cE,\cZ)  \ar@(u,u)[ll]_-{\ov{f}^\ast}  \ar[d]^-{p_!}
\\
\Fun(\cC_0,\cZ) 
&&
\Fun(\cC,\cZ)  \ar[ll]_-{f^\ast}
}
\]
in fact commutes; equivalently, there is a canonical equivalence
\[
\ov{p}_!\ov{f}^\ast \xra{~\simeq~} f^\ast p_!
\]
between functors $\Fun(\cE, \cZ) \ra \Fun(\cC_0,\cZ)$.
\end{prop}
\begin{proof}
Let $F\colon \cE\to \cZ$ be a functor, and let $c\in \cC_0$ be an object.
We must show that the canonical morphism in $\cZ$ between values
\[
(\ov{p}_!\ov{f}^\ast F)(c) \longrightarrow  (f^\ast p_! F)(c)
\]
is an equivalence.  
Because coCartesian fibrations are closed under base change, then $p$ being a coCartesian fibration implies $\ov{p}$ is a coCartesian fibration as well.
We therefore recognize the values of the left Kan extensions in the above expression as the morphism in $\cZ$ involving colimits over fiber $\infty$-categories:
\[
\colim\bigl((\cE_0)_{|c} \to \cE_0 \xra{\ov{f}} \cE \xra{F} \cZ \bigr)  \longrightarrow \colim\bigl((\cE)_{|fc} \to \cE \xra{F} \cZ \bigr)~.
\]
Because the given square among $\infty$-categories is a pullback, the canonical functor $(\cE_0)_{|c} \to (\cE)_{|fc}$ between fiber $\infty$-categories is an equivalence between $\infty$-categories over $\cE$.  
In this way, we recognize that the above morphism in $\cZ$ is an equivalence, as desired.

\end{proof}

We have the following corollary, that finality is preserved under pullbacks along coCartesian fibrations.
This result also appears in~\cite{HTT} as Proposition~4.1.2.15.

\begin{cor}\label{finalpullback}
Consider a pullback diagram of $\oo$-categories
\[
\xymatrix{
\cE_0\ar[d]_{\ov{p}}\ar[r]^{\ov{f}}&\cE\ar[d]_p^{\rm coCart.}\\
\cC_0\ar[r]^f &\cC~.}
\]
If $f$ is final and $p$ is a coCartesian fibration, then $\ov{f}$ is final.  

\end{cor}
\begin{proof}
The commutative diagram of $\infty$-categories
\[
\xymatrix{
\cE_0        \ar[d]_{\ov{p}}          \ar[rr]^{\ov{f}}
&&
\cE      \ar[d]^p
\\
\cC_0     \ar[rr]^f   \ar[rd]_-{!}
&&
\cC  \ar[dl]^-{!}
\\
&
\ast
&
}
\]
determines the a priori lax commutative diagram of $\infty$-categories:
\[
\xymatrix{
\Fun(\cE_0,\Spaces)        \ar[d]_{\ov{p}_!}         
&
\Downarrow
&
\Fun(\cE,\Spaces)      \ar[d]^{p_!}    \ar@(u,u)[ll]_-{\ov{f}^\ast}
\\
\Fun(\cC_0,\Spaces)      \ar[rd]_-{\sf colim}
&
\Downarrow
&
\Fun(\cC, \Spaces)  \ar[dl]^-{\sf colim}     \ar@(u,u)[ll]^-{f^\ast}
\\
&
\Spaces
&
.
}
\]
Proposition~\ref{final.cocart} gives that the upper 2-cell is invertible.  
Finality of $f$ exactly gives that the lower 2-cell is invertible.
It follows that the composite 2-cell is invertible, so that there is a canonical commutative diagram of $\infty$-categories
\[
\xymatrix{
\Fun(\cE_0,\Spaces)        \ar[dr]_-{\colim}         
&
&
\Fun(\cE,\Spaces)      \ar[dl]^-{\colim}    \ar[ll]_-{\ov{f}^\ast}
\\
&
\Spaces
&
.
}
\]
This is precisely the statement that $\ov{f}$ is final.

\end{proof}

Using Quillen's Theorem A, we give a proof of Quillen's Theorem B for $\oo$-categories. See also \cite{barwick-q} for a first treatment of Quillen's Theorem B in the context of quasi-categories, and \cite{heuts.moredijk} for a more central treatment, and Theorem~4.23 of~\cite{MG2} for a more model-independent treatment.

\begin{theorem}[Quillen's Theorem B]\label{thm.B}
Let $\cC\ra\cD$ be a functor between $\infty$-categories such that for each morphism $d\ra d'$ in $\cD$, the functor between $\infty$-overcategories
\[
\cC_{/d}\longrightarrow \cC_{/d'}
\] 
induces an equivalence on classifying spaces: $\sB\cC_{/d}\xra{\simeq} \sB\cC_{/d'}$. 
In this case, for each object $d\in \cD$ the canonical diagram of classifying spaces
\[
\xymatrix{
\sB(\cC_{/d})    \ar[r]      \ar[d]
&
\sB\cC      \ar[d]
\\
\{d\}     \ar[r]
& 
\sB\cD
}
\]
is a pullback diagram of spaces.
\end{theorem}

\begin{proof}
Consider the $\oo$-category $\Ar(\cD)^{|\cC}$ of arrows from $\cC$ to $\cD$, defined as the pullback
\[
\xymatrix{
\Ar(\cD)^{|\cC}  \ar[r]  \ar[d]
&
\Ar(\cD)  \ar[d]^-{\ev_s}
\\
\cC  \ar[r]
&
\cD.
}
\]
The functor $\Ar(\cD) \xra{\ev_s}\cD$ is a left adjoint, and its right adjoint $\cD \to \Ar(\cD)$ is given by the identity arrows of $\cD$.
This adjunction base changes as an adjunction $\cC \rightleftarrows \Ar(\cD)^{|\cC}$.
Since adjoint functors induce equivalences on classifying spaces, we obtain an equivalence
\[
\sB\cC\simeq \sB(\Ar(\cD)^{|\cC})~.
\]
The right adjoint $\cC\ra \Ar(\cD)^{|\cC}$ lies over the target functor $\ev_t: \Ar(\cD)^{|\cC}\ra \cD$. Taking classifying spaces, this gives a commutative diagram of spaces:
\[
\xymatrix{
\sB\cC  \ar[rr]^\simeq  \ar[dr]
&&
\sB(\Ar(\cD)^{|\cC})  \ar[dl]^-{\ev_{t}}
\\
&
\sB\cD
&
.
}
\]
From this equivalence, we reduce to showing that applying the classifying space functor $\sB$ to the pullback square of $\infty$-categories
\begin{equation}\label{co.cart.ing}
\xymatrix{
\cC_{/d}  \ar[r]  \ar[d]
&
\Ar(\cD)^{|\cC}  \ar[d]^-{\ev_t}
\\
\{d\} \ar[r]
&
\cD
}
\end{equation}
gives a pullback square of spaces.  
This square~(\ref{co.cart.ing}) factors as
\begin{equation}\label{next}
\xymatrix{
\cC_{/d}  \ar[r]  \ar[d]
&
\Ar(\cD)^{|\cC}  \ar[d]
\\
\sB\bigl(\cC_{/d}\bigr)  \ar[r]  \ar[d]
&
\sB^{\sf rel}\bigl(\Ar(\cD)^{|\cC}\bigr)  \ar[d]
\\
\{d\} \ar[r]
&
\cD
}
\end{equation}
where $\sB^{\sf rel}\bigl(\Ar(\cD)^{|\cC}\bigr)$ is the relative classifying space (Definition \ref{conserv.loc}) of the functor $\Ar(\cD)^{|\cC}\xra{\ev_t}\cD$. Since $\Ar(\cC)^{|\cC}\ra \cD$ is coCartesian fibration, it is a left-final fibration by Proposition \ref{cart.r.initial}. Consequently this bottom square
\begin{equation}\label{l.fib.ing}
\xymatrix{
\sB\bigl(\cC_{/d}\bigr)  \ar[r]  \ar[d]
&
\sB^{\sf rel}\bigl(\Ar(\cD)^{|\cC}\bigr)    \ar[d]
\\
\{d\} \ar[r]
&
\cD
}
\end{equation}
is a pullback by application of Lemma \ref{l.final.B.fib}, which asserts that relative classifying spaces are computed fiberwise for left-final fibrations. Since the functor $\Ar(\cD)^{|\cC}\ra \sB^{\sf rel}(\Ar(\cD)^{|\cC})$ is a localization, it induces an equivalence on classifying spaces. We are thereby reduced to showing that the value of the classifying space functor $\sB$ on the square (\ref{l.fib.ing}) is a pullback.

We now use our single assumption, that morphisms $d\ra d'$ induce equivalences $\sB(\cC_{/d})\xra{\sim}\sB(\cC_{/d'})$: this implies that the right vertical functor in~(\ref{l.fib.ing}), which is a priori only a left fibration, is a Kan fibration.  
The functor $\sB \cC_{/\bullet}$ classifying this right vertical left fibration of~(\ref{l.fib.ing})
\[
\xymatrix{
\cD \ar[rr]^-{\sB \cC_{/\bullet}}    \ar[d]
&&
\Spaces  
\\
\sB\cD  \ar@{-->}[urr]_-{\exists !}
}
\]
therefore factors through the canonical epimorphism $\cD \to \sB\cD$ to its classifying space.
It follows that the above triangle among $\infty$-categories witnesses this unique extension as the left Kan extension of the straightening $\sB \cC_{/\bullet}$ along $\cD\to \sB \cD$.  
This left Kan extension classifies the left fibration $\sB \bigl(\sB^{\sf rel}\bigl(\Ar(\cD)^{|\cC}\bigr)\Bigr)  \to \sB \cD$, which is the map given by taking classifying spaces on the right vertical functor in~(\ref{l.fib.ing}).  
Unstraightening the left fibrations then establishes that the diagram of $\oo$-categories
\[
\xymatrix{
\sB^{\sf rel}\bigl(\Ar(\cD)^{|\cC}\bigr) \ar[r]  \ar[d]
&
\sB \Bigl(\sB^{\sf rel}\bigl(\Ar(\cD)^{|\cC}\bigr)\Bigr)   \ar[d]
\\
\cD    \ar[r]
&
\sB \cD
}
\]
is a pullback.
Horizontally concatenating this pullback square with the pullback square~(\ref{l.fib.ing}) gives that the composite square
\[
\xymatrix{
\sB\bigl(\cC_{/d}\bigr) \ar[r]  \ar[d]
&
\sB \Bigl(\sB^{\sf rel}\bigl(\Ar(\cD)^{|\cC}\bigr)\Bigr)   \ar[d]
\\
\{d\}    \ar[r]
&
\sB \cD
}
\]
is a pullback, which establishes the last reduction.
\end{proof}

The following is an application of Quillen's Theorem B and the special property of the relative classifying space for left-final and right-initial fibrations, that it is computed fiberwise.

\begin{lemma}\label{quillenthmapp}
Let
\[
\xymatrix{
\cX'   \ar[r]   \ar[d]
&
\cX      \ar[d]
\\
\cY'      \ar[r]
& 
\cY
}
\]
be a pullback diagram of $\oo$-categories. If the right vertical functor $\cX\ra \cY$ is both a left-final fibration and a right-initial fibration, then the diagram induced by taking classifying spaces
\[
\xymatrix{
\sB\cX'   \ar[r]   \ar[d]
&
\sB \cX      \ar[d]
\\
\sB \cY'      \ar[r]
& 
\sB \cY
}
\]
is a pullback diagram of spaces.
\end{lemma}

\begin{proof}
Let $\cE \to \cK$ be a functor between $\infty$-categories.
Suppose that $\cE\to \cK$ is both a left-final fibration and a right-initial fibration.
Corollary~\ref{fib.B.exp} gives that the relative classifying space $\sB^{\sf rel}\cE\to \cK$ is both a left fibration and a right fibration.  
Therefore, for each morphism $k\to k'$ in $\cK$, the coCartesian monodromy maps and the Cartesian monodromy maps for this left/right fibration $\sB^{\sf rel}\cE\to \cK$ implement a pair of maps between fiber spaces:
\[
\sB \cE_{|k} 
~ \overset{\simeq}{\rightleftarrows} ~
\sB \cE_{|k'}~.
\]
Theorem~\ref{theorem-l.fib} states, in particular, that the $\infty$-category $\Spaces^\sim$ classifies left/right fibrations.  
In particular, these two monodromy maps are mutual inverse equivalences to one another.

Applying the above discussion to each of the functors $\cX \to \cY$ and $\cX' \to \cY'$, Quillen's Theorem B applies to identify the respective fibers of $\sB \cX \to \sB \cY$ and $\sB \cX' \to \sB \cY'$: the fiber of the first over $y\in \cY$ is $\sB \cX_{|y}$ and the fiber of the second over $y'\in \cY'$ is $\sB \cX'_{|y'}$.  
The canonical comparison map between these fibers is an equivalence, then, because the given square among $\infty$-categories is a pullback.

\end{proof}

\subsection{Left/right fibrations via lifting criteria}\label{sec.l.fib.lifting}
We show that left/right fibrations are characterized as those functors between $\infty$-categories that have a lifting property with respect to initial/final functors.

\begin{prop}\label{left.initial.lift}
Let $\cE\xra{\pi} \cK$ be a functor between $\infty$-categories.
\begin{enumerate}
\item The following two conditions on $\pi$ are equivalent.  
\begin{enumerate}
\item $\pi$ is a left fibration.

\item For each solid diagram of $\infty$-categories
\[
\xymatrix{
\cJ_0  \ar[rr]  \ar[d]_-{\rm initial}
&&
\cE  \ar[d]^-\pi  
\\
\cJ  \ar[rr]   \ar@{-->}[urr]^-{\exists !}
&&
\cK
}
\]
in which the left vertical functor is initial, the $\infty$-category of fillers is a contractible $\infty$-groupoid; equivalently, for each initial functor $\cJ_0 \to \cJ$, the restriction functor between $\infty$-categories of sections
\[
\Fun_{/\cK}(\cJ,\cE) \longrightarrow  \Fun_{/\cK}(\cJ_0,\cE)
\]
is an equivalence.  

\end{enumerate}

\item The following two conditions on $\pi$ are equivalent.  
\begin{enumerate}
\item $\pi$ is a right fibration.

\item For each solid diagram of $\infty$-categories
\[
\xymatrix{
\cJ_0  \ar[rr]  \ar[d]_-{\rm final}
&&
\cE  \ar[d]^-\pi  
\\
\cJ   \ar[rr]  \ar@{-->}[urr]^-{\exists !}
&&
\cK
}
\]
in which the left vertical functor is final, the $\infty$-category of fillers is a contractible $\infty$-groupoid; that is, for each final functor $\cJ_0 \to \cJ$, the restriction functor between $\infty$-categories of sections
\[
\Fun_{/\cK}(\cJ,\cE) \longrightarrow  \Fun_{/\cK}(\cJ_0,\cE)
\]
is an equivalence.  

\end{enumerate}

\end{enumerate}

\end{prop}

\begin{proof}
The two assertions imply one another, as implemented by taking opposites.
We are therefore reduced to proving assertion~(1).

Condition~(b) implies condition~(a) because, for each $\infty$-category $\cJ$, the functor $\ast \to \cJ^{\tl}$, which selects the cone point, is initial.  

It remains to establish the implication (a)$\implies$(b).
Suppose $\pi$ is a left fibration.
Let $\cJ_0\to \cJ$ be an initial functor between $\infty$-categories.
The problem is to show the restriction functor between $\infty$-categories of sections
\[
\Fun_{/\cK}(\cJ,\cE)  \longrightarrow \Fun_{/\cK}(\cJ_0,\cE)
\]
is an equivalence.
Because $\cE\to \cK$ is assumed a left fibration, this functor is identified as the functor
\[
\Fun_{/\cK}(\cJ^{\widehat{~}}_{\sf l.fib},\cE)  \longrightarrow \Fun_{/\cK}((\cJ_0)^{\widehat{~}}_{\sf l.fib},\cE)
\]
involving left fibration-replacements of $\cJ_0\to \cK$ and $\cJ\to \cK$.
Thus, it is sufficient to show the functor over $\cK$
\[
(\cJ_0)^{\widehat{~}}_{\sf l.fib} \longrightarrow \cJ^{\widehat{~}}_{\sf l.fib}
\]
between left fibration-replacements over $\cK$ is an equivalence.
Proposition~\ref{its.fib.B} recognizes this functor over $\cK$ as the functor
\[
\sB^{\sf rel}_\cK\bigl( \cJ_0 \underset{\cK}\times \Ar(\cK)\bigr)  \longrightarrow \sB^{\sf rel}_{\cK} \bigl(\cJ\underset{\cK}\times \Ar(\cK)\bigr)
\]
between relative classifying spaces of the coCartesian fibration-replacements of $\cJ_0\to \cK$ and $\cJ\to \cK$.  
Because equivalences between left fibrations are detected on fibers (Lemma~\ref{l.fib.fibers}, using Lemma~\ref{left.is.coCart}), we are reduced to showing that, for each $x\in \cK$, the map between fibers
\[
\bigl(\sB^{\sf rel}_\cK\bigl( \cJ_0 \underset{\cK}\times \Ar(\cK)\bigr)\bigr)_{|x} 
\longrightarrow 
\bigl(\sB^{\sf rel}_{\cK} \bigl(\cJ\underset{\cK}\times \Ar(\cK)\bigr)\bigr)_{|x} 
\]
is an equivalence.
Lemma~\ref{fib.B.base.change} identifies this map as the map between classifying spaces
\[
\sB \cJ_0\underset{\cK}\times \cK_{/x} 
\longrightarrow
\sB \cJ\underset{\cK}\times \cK_{/x}~.
\]
Through Lemma~\ref{final.B.equiv}, this map being an equivalence is implied by finality of the canonical functor
\[
\cJ_0\underset{\cK}\times \cK_{/x} 
\longrightarrow
\cJ\underset{\cK}\times \cK_{/x}~.
\]
Using that $\cJ_0\to \cJ$ is assumed final, this finality is implied by Corollary~\ref{finalpullback}.
This concludes the proof.

\end{proof}

\begin{cor}\label{fib.B.final}
Let $\cJ \to \cK$ be a functor between $\infty$-categories.

\begin{enumerate}
\item 
The canonical functor $\cJ \to \cJ^{\widehat{~}}_{\sf l.fib}$, to the left fibration-replacement over $\cK$, is initial.
Furthermore, it is a final object in the full $\infty$-subcategory of $(\Cat_{/\cK})^{\cJ\to \cK/}$ consisting of the initial functors $\cJ \to \cJ'$ over $\cK$.

\item 
The canonical functor $\cJ \to \cJ^{\widehat{~}}_{\sf r.fib}$, to the right fibration-replacement over $\cK$, is final.
Furthermore, it is a final object in the full $\infty$-subcategory of $(\Cat_{/\cK})^{\cJ\to \cK/}$ consisting of the final functors $\cJ \to \cJ'$ over $\cK$.

\end{enumerate}

\end{cor}

\begin{proof}
The two assertions imply one another, as implemented by taking opposites.
We are therefore reduced to proving the first assertion.

Proposition~\ref{its.fib.B} witnesses the composite functor over $\cK$,
\[
\cJ \longrightarrow \Ar(\cK)^{|\cJ} \longrightarrow \sB^{\sf rel}_{\cK}\bigl(\Ar(\cK)^{|\cJ}\bigr)~,
\]
as the canonical functor to the left fibration-replacement $\cJ \to \cJ^{\widehat{~}}_{\sf l.fib}$.
The right factor in this composition is a localization.  Therefore, Proposition~\ref{localization.final} gives that this right factor is initial.  
Lemma~\ref{candidate} gives that the left factor in this composition is a left adjoint.  Therefore, Lemma~\ref{loc.final} gives that this left factor is initial.
Through Lemma~\ref{2.o.3}, we conclude that the composite functor is initial, as desired.

Now, let $\cJ \xra{\rm initial} \cJ' \to \cK$ and $\cJ \xra{\rm initial} \cJ'' \to \cK$ be two objects in the named $\infty$-subcategory of $(\Cat_{/\cK})^{\cJ\to \cK/}$.
Notice that the canonical square among spaces involving the space of morphisms in this $\infty$-subcategory,
\[
\xymatrix{
(\Cat_{/\cK})^{\cJ\to \cK/}\Bigl((\cJ \to \cJ' \to \cK),(\cJ \to \cJ'' \to \cK)  \Bigr) \ar[rr]  \ar[d]
&&
\Cat_{/\cK}(\cJ',\cJ'')  \ar[d]
\\
\ast  \ar[rr]^-{\lag \cJ\to \cJ''\rag}
&&
\Cat_{/\cK}(\cJ,\cJ'')
}
\]
is a pullback.  
Taking $(\cJ \to \cJ'' \to \cK) =(\cJ \to \cJ^{\widehat{~}}_{\sf l.fib} \to \cK)$, Proposition~\ref{left.initial.lift} gives that the right vertical map in this square is an equivalence.  
Using that the square is a pullback, the space of morphisms in $(\Cat_{/\cK})^{\cJ\to \cK/}$ from $(\cJ \to \cJ'\to \cK)$ to $(\cJ \to \cJ^{\widehat{~}}_{\sf l.fib} \to \cK)$ is contractible provided $\cJ \to \cJ'$ is initial.  
We conclude that the object $(\cJ \to \cJ^{\widehat{~}}_{\sf l.fib} \to \cK)$ of the named full $\infty$-subcategory of $(\Cat_{/\cK})^{\cJ\to \cK/}$ is final, as desired.

\end{proof}


\begin{thebibliography}{99}

\bibitem[AF1]{fact} Ayala, David; Francis, John. Factorization homology of topological manifolds. J. Topol. 8 (2015), no. 4, 1045--1084.

\bibitem[AF2]{pkd} Ayala, David; Francis, John. Poincar\'e/Koszul duality.  Comm. Math. Phys. 365 (2019), no. 3, 847--933.

\bibitem[AF3]{flagged} Ayala, David; Francis, John. Flagged higher categories. Topology and quantum theory in interaction, 137--173, Contemp. Math., 718, Amer. Math. Soc., Providence, RI, 2018.

\bibitem[AFR1]{striation} Ayala, David; Francis, John; Rozenblyum, Nick. A stratified homotopy hypothesis. Accepted, Journal of the European Mathematical Society. arXiv:1409.2857.

\bibitem[AFR2]{emb1a} Ayala, David; Francis, John; Rozenblyum, Nick. Factorization homology I: Higher categories. Adv. Math. 333 (2018), 1042--1177.

\bibitem[AFT1]{aft1} Ayala, David; Francis, John; Tanaka, Hiro Lee. Local structures on stratified spaces. Adv. Math. 307 (2017), 903--1028.

\bibitem[AFT2]{aft2} Ayala, David; Francis, John; Tanaka, Hiro Lee. Factorization homology of stratified spaces. Selecta Math. (N.S.) 23 (2017), no. 1, 293--362.

\bibitem[Ba1]{barwick-q} Barwick, Clark. On the Q-construction for exact $\oo$-categories. Prepint (2013). Available at arXiv:1301.4725

\bibitem[Ba2]{barwick} Barwick, Clark. Spectral Mackey functors and equivariant algebraic K-theory (I). Adv. Math. 304 (2017), 646--727.

\bibitem[BS]{barwick.shah} Barwick, Clark; Shah, Jay. Fibrations in $\oo$-category theory. 2016 MATRIX annals, 17--42, MATRIX Book Ser., 1, Springer, Cham, 2018.

\bibitem[Co]{conduche} Conduch\'e, Fran\c{c}ois. Au sujet de l'existence d'adjoints \'a droite aux foncteurs ``image r\'eciproque'' dans la cat\'egorie des cat\'egories. C. R. Acad. Sci. Paris S\'er. A-B 275 (1972), A891--A894.

\bibitem[GR]{gaitsgory.rozenblyum} Gaitsgory, Dennis; Rozenblyum, Nick. A study in derived algebraic geometry. Vol. I. Correspondences and duality. Mathematical Surveys and Monographs, 221. American Mathematical Society, Providence, RI, 2017

\bibitem[GR2]{gaitsgory.rozenblyum2} Gaitsgory, Dennis; Rozenblyum, Nick. A study in derived algebraic geometry. Vol. II. Deformations, Lie theory and formal geometry. Mathematical Surveys and Monographs, 221. American Mathematical Society, Providence, RI, 2017

\bibitem[GHN]{gepner-haugseng-nikolaus} Gepner, David; Haugseng, Rune; Nikolaus, Thomas. Lax colimits and free fibrations in $\oo$-categories. Doc. Math. 22 (2017), 1225--1266.

\bibitem[Gi]{giraud} Giraud, Jean. M\'ethode de la descente. Bull. Soc. Math. France M\'em. 2 1964 viii+150 pp.

\bibitem[Ha]{haugseng} Haugseng, Rune. Iterated spans and ``classical'' topological field theories. Iterated spans and classical topological field theories. Math. Z. 289 (2018), no. 3-4, 1427--1488.

\bibitem[HM]{heuts.moredijk} Heuts, Gijs; Moerdijk, Ieke. Left fibrations and homotopy colimits. Math. Z. 279 (2015), no. 3--4, 723--744.

\bibitem[Jo1]{joyal1} Joyal, Andr\'e. Quasi-categories and Kan complexes. Special volume celebrating the 70th birthday of Professor Max Kelly. J. Pure Appl. Algebra 175 (2002), no. 1--3, 207--222.

\bibitem[Jo2]{joyal} Joyal, Andr\'e. Notes on quasi-categories. Preprint.

\bibitem[JT]{joyaltierney}  Joyal, Andr\'e; Tierney, Myles. Quasi-categories vs Segal spaces. Categories in algebra, geometry and mathematical physics, 277--326, Contemp. Math., 431, Amer. Math. Soc., Providence, RI, 2007. 

\bibitem[Lu1]{HTT} Lurie, Jacob. Higher topos theory. Annals of Mathematics Studies, 170. Princeton University Press, Princeton, NJ, 2009. xviii+925 pp.

\bibitem[Lu2]{HA} Lurie, Jacob. Higher algebra. Preprint.

\bibitem[MG1]{MG1} Mazel-Gee, Aaron. A user's guide to co/Cartesian fibrations. Preprint, 2015.

\bibitem[MG2]{MG2} Mazel-Gee, Aaron. All about the Grothendieck construction. Preprint, 2015.

\bibitem[Qu]{quillen} Quillen, Daniel. Higher algebraic $K$-theory: I, Algebraic $K$-theory, I: Higher $K$-theories (Proc. Conf., Battelle Memorial Inst., Seattle, Wash., 1972), Springer, Berlin, 1973, pp. 85--147.

\bibitem[Re1]{rezk}  Rezk, Charles. A model for the homotopy theory of homotopy theory. Trans. Amer. Math. Soc. 353 (2001), no. 3, 973--1007.

\bibitem[Re2]{rezk-n} Rezk, Charles. A Cartesian presentation of weak $n$-categories. Geom. Topol. 14 (2010), no. 1, 521--571.




\end{thebibliography}
\end{document}